\documentclass[a4paper, 12pt, oneside]{amsart}
\usepackage[utf8]{inputenc}
\usepackage{amsmath,amsthm,amsfonts,amssymb}
\usepackage{url}

\usepackage{graphicx} 

\usepackage{mathtools} 

\usepackage[shortlabels]{enumitem}

\usepackage{upgreek}

\usepackage{tabularx}
\usepackage{tabularray}
\usepackage{geometry}

\newcommand{\lcn}[1]{\prescript{#1}{}{\nabla}}

\DeclareMathOperator{\assoc}{assoc}
\DeclareMathOperator{\Ham}{Ham}

\DeclarePairedDelimiter\pg{[}{]}
\DeclarePairedDelimiter\po{ \{ }{ \} }

\newcommand{\rest}[2]{\left.#1\right|_{#2}}

\newcommand{\dif}{\mathrm{d}}

\DeclareMathOperator{\en}{\mathrm{End}}

\newcommand{\spann}{\mathrm{span}}
\newcommand{\ric}{\mathrm{Ric}}

\newcommand{\e}{\mathrm{e}}

\newcommand{\pr}{\mathrm{pr}}
\newcommand{\kil}{\mathrm{Kill}}
\newcommand{\aff}{\mathrm{Aff}}
\newcommand{\jac}{\mathrm{Jac}}

\newcommand{\cyc}{\mathrm{cyclic}}
\newcommand{\cg}{\text{g}}
\newcommand{\wn}{\dot{\nabla}}
\newcommand{\wnt}{\dot{\nabla}^0}

\newcommand\varlist {,\makebox[1em][c]{.\hfil.\hfil.},}

\newcommand{\N}{{\mathbb{N}}}

\newcommand{\R}{{\mathbb{R}}}

\newcommand{\cC}{\mathcal{C}}

\newcommand{\cCi}{\cC^\infty}

\newcommand{\al}{\alpha}
\newcommand{\be}{\beta}

\newcommand{\GL}{\mathrm{GL}}

\newcommand{\SO}{\mathrm{SO}}

\newcommand{\SSS}{\mathrm{S}}

\newcommand{\supp}{\mathrm{supp}}

\DeclareMathOperator{\Jac}{Jac}

\DeclareMathOperator{\grad}{grad}
\DeclareMathOperator{\tr}{tr}
\DeclareMathOperator{\rk}{rk}

\DeclareMathOperator{\im}{im}

\DeclareMathOperator{\Sym}{Sym}
\DeclareMathOperator{\sym}{sym}

\DeclareMathOperator{\Ann}{Ann}

\newtheorem{theorem}{Theorem}[section]
\newtheorem{corollary}[theorem]{Corollary}
\newtheorem{lemma}[theorem]{Lemma}
\newtheorem{proposition}[theorem]{Proposition}
\newtheorem{example}[theorem]{Example}

\newtheorem{innercustomthm}{Theorem}
\newenvironment{customthm}[1]
{\renewcommand\theinnercustomthm{#1}\innercustomthm}
{\endinnercustomthm}

\newtheorem{innercustomprop}{Proposition}
\newenvironment{customprop}[1]
{\renewcommand\theinnercustomprop{#1}\innercustomprop}
{\endinnercustomprop}

\newtheorem{innercustomlem}{Lemma}
\newenvironment{customlem}[1]
{\renewcommand\theinnercustomlem{#1}\innercustomlem}
{\endinnercustomlem}

\theoremstyle{definition}
\newtheorem{definition}[theorem]{Definition}
\newtheorem*{definition*}{Definition}
\theoremstyle{remark}

\newtheorem{remark}[theorem]{Remark}

\newcommand{\symall}{\mathfrak{X}^{\bullet}_\text{sym}(M)}
\newcommand{\symr}{\mathfrak{X}^r_\text{sym}(M)}
\newcommand{\symbi}{\mathfrak{X}^2_\text{sym}(M)}

\title[]{Symmetric Poisson geometry,\\totally geodesic foliations\\
and Jacobi-Jordan algebras}

\author[F. Moučka]{Filip Moučka}
\author[R. Rubio]{Roberto Rubio}

\address{Universitat Aut\`onoma de Barcelona, 08193 Barcelona, Spain;\newline \indent Faculty of Nuclear Sciences and Physical Engineering, Czech Technical\newline\indent University in Prague, 115 19 Prague 1, Czech Republic}
\email{filip.moucka@autonoma.cat; mouckfil@cvut.cz}

\address{Universitat Aut\`onoma de Barcelona, 08193 Barcelona, Spain}
\email{roberto.rubio@uab.es}

\thanks{This project has been supported by MICIU/AEI/10.13039/501100011033 and the EU FEDER under the grants PID2022-137667NA-I00 (GENTLE) and CNS2024-154695 (DÉCOLLAGE). The first author has been partially supported by the Grant Agency of the Czech Technical University in Prague, grant No. SGS25/163/OHK4/3T/14. The second author has also received support from the 
MICIU/AEI and the EU FSE under the Ramón y Cajal fellowship RYC2020-030114-I and from the AGAUR under the grant 2021-SGR-01015.}

\begin{document}

\setcounter{tocdepth}{1}

\begin{abstract}
 We introduce symmetric Poisson structures as pairs consisting of a symmetric bivector field and a torsion-free connection satisfying an integrability condition analogous to that in usual Poisson geometry. Equivalent conditions in Poisson geometry have inequivalent analogues in symmetric Poisson geometry and we distinguish between symmetric and strong symmetric Poisson structures. We prove that symmetric Poisson structures correspond to locally geodesically invariant distributions together with a characteristic metric, whereas strong symmetric Poisson structures correspond to totally geodesic foliations together with a leaf metric and a leaf connection. We introduce, using the \mbox{Patterson-Wal}ker metric, a dynamics on the cotangent bundle and show its connection to symmetric Poisson geometry, the parallel transport equation and the Newtonian equation for conservative systems. Finally, we prove that linear symmetric Poisson structures are in correspondence with Jacobi-Jordan algebras, whereas strong symmetric correspond to those that are moreover associative.
\end{abstract}

\maketitle

\vspace{-.6cm}

\tableofcontents

\section{Introduction}

Riemannian, or more generally \mbox{(pseudo-)Rie}mannian, and symplectic geometry are among the most established areas in mathematics. They describe geometric structures: metrics and symplectic forms. These are, respectively, symmetric and skew-symmetric $2$-forms with a non-degeneracy condition (and, in the case of symplectic structures, the extra integrability condition of being closed). 

Symplectic geometry degenerates, or is extended, in two ways. First, by allowing the closed $2$-form to be degenerate, which gives rise to a presymplectic structure. Second, by considering the inverse of a symplectic structure, that is, a skew-symmetric bivector field, and allowing it to be degenerate. The latter, together with an integrability condition generalizing closedness, is the starting point of Poisson geometry, a much richer theory than presymplectic geometry. 

For a metric, its degeneration results in the study of symmetric bilinear forms with non-trivial kernel, which are well known. However, we have not found a general study of degenerations of metrics as symmetric bivectors, that is, of what should be called symmetric Poisson geometry.

The aim of this paper is to overtake this study of symmetric bivector fields by bringing in the natural analogy with Poisson geometry, that is, establishing an integrability condition and describing its geometric and dynamical significance.

Using an analogue of the most common formulations of the integrability condition in Poisson geometry results in a void condition or the vanishing of the symmetric bivector field (Section \ref{sec:int-Poisson}). We resort to symmetric Cartan calculus \cite{SymCartan} to find suitable integrability conditions. There are two key points that deserve special mention.

First, symmetric Cartan calculus depends upon the choice of a connection on $M$. This is actually very natural if we recall that non-degenerate bivector fields are equivalent to \mbox{(pseudo-)Riem}annian metrics, whose features are two-fold:
\begin{center}
 \SetTblrInner{rowsep=0.6ex}
\begin{tblr}{width=\textwidth, colspec={m{0.25\textwidth,c}|m{0.675\textwidth,l}}}
metric-related & lengths and angles, geodesics as length minimizing curves, gradients, volume and integration, Hodge star\\
\hline
connection-related &curvature, parallel transport and geodesics, totally geodesic submanifolds, differentiation of tensor fields\\
\end{tblr} 
\end{center}
\sloppy This split is not apparent because of the existence and uniqueness of the \mbox{Levi-Civ}ita connection. However, when a \mbox{(pseudo-)Riem}annian metric degenerates as a symmetric bivector field, there is no Levi-Civita connection and we have to incorporate a compatible connection into the definition. The choice of a connection, which we can assume to be torsion-free, determines the symmetric derivative operator, the symmetric Lie derivative and the symmetric bracket on vector fields, $[\,\,,\,]_s$, which we extend to the symmetric Schouten bracket on symmetric multivector fields $\mathfrak{X}^\bullet_{\sym}(M)$. 

Second, for a pair $(\vartheta,\nabla)$ consisting of $\vartheta\in \mathfrak{X}^2_{\sym}(M)$ and a torsion-free connection $\nabla$, the integrability conditions that mirror the most common ones in Poisson geometry are
\begin{align*}\label{eq:two-conditions-intro}
 [\vartheta,\vartheta]_s&=0 &&\text{and} & X_{\{f,\,\cg\}}=&[X_f,X_\cg]_s\end{align*}
for $f,\cg\in\cCi(M)$, where $\{f,\cg\}:=\vartheta(\dif f,\dif\cg)$, $X_f:=\vartheta(\dif f)$, which is the gradient of $f$ in this context, and $[\,\, ,\,]_s$ also denotes the symmetric Schouten associated to $\nabla$. Unlike in the usual Poisson case, these two conditions are not equivalent. The latter implies the former, thus we arrive at the following definition:

\begin{definition*}
A \textbf{symmetric Poisson structure} is a pair $(\vartheta,\nabla)$ such that $[\vartheta,\vartheta]_s=0$. If, moreover, $X_{\{f,\,\cg\}}=[X_f,X_\cg]_s$, we call it a strong symmetric Poisson structure.
\end{definition*}  
A good source of recognizable examples is provided in Section \ref{sec:non-deg-sym-Poisson}: non-degenerate symmetric Poisson structures correspond to the inverses of Killing tensors (see \cite{SymCartan} and the references therein) together with the torsion-free connection they are Killing for, whereas non-degenerate strong symmetric Poisson structures are inverses of pseudo-Riemannian metric together with their Levi-Civita connection.

The next step is the geometric interpretation of the integrability conditions for a pair $(\vartheta, \nabla)$. When regarded as a map $\vartheta: T^*M\to TM$, we have a distribution $\im \vartheta\subseteq TM$ and a $\cCi(M)$-module $\mathcal{F}_{\vartheta}:=\vartheta(\Omega^1(M))\subseteq \mathfrak{X}(M)$. In general, the distribution and module do not determine a partition of the manifold, as the condition $[\vartheta,\vartheta]_s=0$ does not even imply Lie involutivity, but only that the characteristic module is preserved by the symmetric bracket, that is, $[\mathcal{F}_\vartheta, \mathcal{F}_\vartheta]_s\subseteq \mathcal{F}_\vartheta$. On the other hand, the stronger condition $X_{\{f,\,\cg\}}=[X_f,X_\cg]_s$ does imply the Lie involutivity of $\mathcal{F}_\vartheta$ (Proposition \ref{prop: ssPs-involutivity}). This motivates the introduction of an a priori, intermediate class, involutive symmetric Poisson structures, those $(\vartheta, \nabla)$ such that $\mathcal{F}_\vartheta$ is involutive for the Lie bracket. We thus have:
\begin{equation*}
 \left\{\begin{array}{c}
\text{strong symmetric}\\
 \text{Poisson structures}
 \end{array}\right\}\subseteq\left\{\begin{array}{c}
\text{involutive symmetric}\\
 \text{Poisson structures}
 \end{array}\right\}\subseteq\left\{\begin{array}{c}
\text{symmetric}\\
 \text{Poisson structures}
 \end{array}\right\}.
\end{equation*}

The main questions are what is the geometric interpretation of these three classes of structures and whether these are different classes. The only immediate answer is given by Example \ref{ex: SO(3)}, which is a non-involutive symmetric Poisson structure.

The distribution $\im \vartheta$ comes equipped, at each $m\in M$, with the
 canonical linear \mbox{(pseudo-)Rie}mannian metric $g_{\vartheta_m}$ given, for $\alpha,\beta\in T^*_mM$, by
\begin{equation*}\label{eq: vartheta-char-metric-intro}
 g_{\vartheta_m}(\vartheta(\alpha),\vartheta(\beta)):=\vartheta(\alpha,\beta). 
\end{equation*}
For $\nabla$-geodesics that are $\vartheta$-admissible (those that admit a curve $a:I\rightarrow T^*M$ such that $\vartheta(a)$ is the velocity of the geodesic), we prove that the square of the speed, with respect to $g_{\vartheta_m}$, is constant (Proposition \ref{prop: const-speed}), thus extending a well-known fact for \mbox{(pseudo-)Rie}mannian manifolds.

When $(\vartheta,\nabla)$ is moreover involutive, we can describe what happens on a leaf $N$: the metrics $g_{\vartheta_m}$ extend to a \mbox{(pseudo-)Rie}mannian metric $g_N$, and the connection $\nabla$ restricts to a torsion-free connection $\nabla^N$.

This leaf-wise structure allows us to prove our first main result, which should be compared with the symplectic foliation of usual Poisson geometry.
\begin{customthm}{\ref{thm:nabla-geodesic-part-of-inv-sym-Poisson}}
 The characteristic partition of an involutive symmetric Poisson structure $(\vartheta,\nabla)$ is totally geodesic. Moreover, on any leaf $N$, the pair $(g^{-1}_N,\nabla^N)$ is a non-degenerate symmetric Poisson structure. In addition, if $(\vartheta,\nabla)$ is strong, $\nabla^N$ is the Levi-Civita connection of $g_N$, that is, $(g^{-1}_N,\nabla^N)$ is also strong.
\end{customthm}

Poisson geometry originates from the study of Hamiltonian dynamics formulated in terms of the canonical Poisson bracket on $\cCi(T^*M)$, which comes from the inverse of the canonical symplectic form. The symmetric analogue of this symplectic form is the so-called Patterson-Walker metric $g_\nabla$, first introduced in \cite{PatRE}, which depends upon the choice of a connection $\nabla$ and can be understood in the framework of symmetric Cartan calculus, \cite[Sec. 5]{SymCartan}. By considering the inverse of $g_\nabla$ we obtain the symmetric Poisson bracket $\{\,,\}_\nabla$ on $\cCi(T^*M)$, which allows us to introduce an analogue of Hamiltonian dynamics, which we call Patterson-Walker dynamics. For a given $H\in \cCi(T^*M)$ we study the integral curves of $\{ H,\,\}_\nabla$. By making different choices, we recover the parallel transport equation, the gradient extension of a dynamical system and the Newtonian equation for conservative systems. The study of this dynamics comes with two unexpected consequences. First, a relation between the symmetric Schouten bracket and the Patterson-Walker metric.
\begin{customprop}{\ref{prop: sym-Schouten-PW}}
Let $\nabla$ be a torsion-free connection on $M$. The vertical lift of a symmetric multivector field is an algebra isomorphism between the commutative algebras $(\mathfrak{X}^\bullet_\emph{sym}(M),[\,\,,\,]_s)$ and $(\mathcal{P}ol(T^*M),\lbrace\,,\rbrace_\nabla)$.
\end{customprop}
\noindent Second, the geometrical description of symmetric Poisson structures as locally geodesically invariant distributions, which we introduce as those for which every geodesic whose velocity is on the distribution at a point, must have its velocity on the distribution in a neighbourhood (Definition \ref{def:loc-geod-inv}). Our second main result is:

\begin{customthm}{\ref{thm: gen-Lewis}}
 The characteristic distribution of a symmetric Poisson structure is locally geodesically invariant. 
\end{customthm}

Having this clear geometric interpretation in Theorems \ref{thm:nabla-geodesic-part-of-inv-sym-Poisson} and \ref{thm: gen-Lewis} is the most solid justification for our definition of (strong) symmetric Poisson structures formulated in terms of symmetric Cartan calculus.

We finish the paper by looking at examples of symmetric Poisson structures, including squares of autoparallel vector fields, as well as left-invariant symmetric Poisson structures. The most interesting class is that of linear symmetric Poisson structures on the dual of a vector space. By considering the linear function $\iota_u\in\cCi(V^*)$ for $u\in V$, we prove our third and last main result.

\begin{customthm}{\ref{thm: JJA-1-to-1}}
 The assignment $\iota_{u\cdot v}:=\lbrace \iota_u,\iota_v\rbrace$ gives a bijection
 \begin{equation*}
 \left\{\begin{array}{c}
\text{linear (strong) symmetric Poisson}\\
 \text{structures $(\lbrace\,,\rbrace,\nabla^\emph{Euc})$ on }V^*
 \end{array}\right\}\overset{\sim}{\longleftrightarrow}\left\{\begin{array}{c}
 \text{(associative) Jacobi-Jordan}\\
 \text{algebra structures $\cdot$ on }V
 \end{array}\right\}.
 \end{equation*}
\end{customthm}

This is the analogue of the fact that linear Poisson structures on a vector space $V^*$ correspond to Lie algebra structures on $V$. It establishes a bridge to Jacobi-Jordan algebras \cite{JJaBur}, which allows us to bring classification results in low dimensions and translate geometric invariants into algebraic ones. In particular, we bring a normal form for non-strong linear symmetric Poisson structures in dimension $5$ and find the first example of an involutive symmetric Poisson structure that is non-strong for any choice of connection (Example \ref{ex: R5}). This also shows that symmetric, involutive symmetric and strong symmetric Poisson structures are indeed different classes.

In this work, we have laid the foundations of symmetric Poisson geometry from the Poisson viewpoint. A natural continuation is the study of symmetric Poisson maps, product structures, submanifolds and cohomology. However, there is an intriguing (pseudo-)Riemannian side too, where natural notions such as the \textit{symmetric Poisson scalar curvature} $\mathcal{R}_{(\vartheta,\nabla)}:=\tr(\vartheta\otimes \ric_\nabla)$ and the \textit{symmetric Poisson Laplacian} 
$\Delta_{(\vartheta,\nabla)}f:=\tr(\vartheta\otimes\nabla\dif f)$ can be defined, or a connection with sub-Riemannian geometry can be established. These directions will be explored in future work.

%\medskip
\vspace{3pt}

\textbf{Acknowledgements.} We would like to thank the Weizmann Institute of Science for their hospitality during our visit, when symmetric Poisson structures were first discovered, as well as Gil Cavalcanti, Marius Crainic, Marco Gualtieri, Andrew Lewis and Rui Loja Fernandes for helpful discussions on the subject. 

%\medskip
\vspace{3pt}

\textbf{Notation and conventions:} All manifolds and maps are smooth. We consider an arbitrary manifold $M$ (of positive dimension $n\in\N$). We denote its smooth functions by $\cCi(M)$, its tangent bundle by $TM$, its cotangent bundle by $T^*M$, its vector fields by $\mathfrak{X}(M)$ and its differential forms by $\Omega^\bullet(M)$. By a connection on $M$, we mean an affine connection (that is, a connection on $TM$), which can induce other connections (for instance, on $T^*M$ by duality). We denote all the induced connections with the same symbol. We recall that every such connection has an associated torsion-free connection, which we denote by $\nabla^0$ and is given by
\begin{equation*}
 \nabla^0_XY:=\nabla_XY-\frac{1}{2}T_\nabla(X,Y),
\end{equation*}
where $T_\nabla\in\Omega^2(M,TM)$ is the torsion tensor of $\nabla$.

We often use polarization in our proofs. That is, given real vector spaces $V$ and $V'$, two totally symmetric multilinear maps $\varphi, \psi:V\times\ldots\times V\rightarrow V'$ are the same if and only if $\varphi(u\varlist u)=\psi(u\varlist u)$ for every $u\in V$. From now on, we will denote by $V$ a real finite-dimensional vector space.

We use the Einstein summation convention and the Kronecker $\delta$ notation throughout the text.

\section{\for{toc}{Symmetric bivector fields and the symmetric Schouten bracket}\except{toc}{Symmetric bivector fields and the\\ symmetric Schouten bracket}} 

Just as Poisson structures are skew-symmetric bivector fields that satisfy an integrability condition, symmetric bivector fields are the basis for defining symmetric Poisson structures. We will see that their integrability condition is a richer question.

\subsection{The symmetric Poisson bracket and the gradient map}

We denote the space of symmetric bivector fields on $M$ by $\symbi:=\Gamma(\Sym^2TM)$ and use $\vartheta$ for a generic element. We will consistently use the notation $f$, $\cg$, $h\in \cCi(M)$.

Any $\vartheta \in \symbi$ determines, via 
\begin{equation}\label{eq:sym-bracket}
\{ f,\cg\}:=\vartheta(\dif f,\dif \cg) 
\end{equation}
 an $\R$-bilinear map $\{\,,\}:\cCi(M)\times \cCi(M)\rightarrow \cCi(M)$ satisfying
\begin{align*}
&\{ f,\cg\}=\{ \cg,f\}, & &\{ f,\cg h\}=\{ f,\cg\} h+\cg\{ f,h\},
\end{align*}
which we will call a \textbf{symmetric Poisson bracket}. In fact, \eqref{eq:sym-bracket} gives a one-to-one correspondence between $\symbi$ and such brackets.

Moreover, $\vartheta$ determines the natural map $\grad:\cCi(M)\rightarrow \mathfrak{X}(M)$,
\[
\grad:=\vartheta\circ\dif,
\]
or, equivalently in terms of the bracket, $\grad f:=\po{f,\,\,}$. This is called the \textbf{gradient} of $f\in\cCi(M)$. We will also use the shorthand notation
\[X_f:=\grad f.\]

\begin{example}[Flat \mbox{(pseudo-)Rie}mannian manifolds]\label{ex: flat}
Consider the cartesian space $\R^{p+q}$ with the (pseudo-)Euclidean metric $g$ of signature $(p,q)$. The inverse of the metric is the symmetric bivector field
\begin{equation*}
g^{-1}=\sum_{i=1}^p\partial_{x^i}\otimes\partial_{x^i}-\sum_{j=1}^q\partial_{x^{p+j}}\otimes\partial_{x^{p+j}}.
\end{equation*}
The corresponding bracket is given, for $f$, $\emph{g}\in \cCi(\mathbb{R}^{p+q})$, by
\begin{equation*}
\{ f,\emph{g}\}=\sum_{i=1}^p\frac{\partial f}{\partial x^i}\frac{\partial \emph{g}}{\partial x^i}-\sum_{j=1}^q\frac{\partial f}{\partial x^{p+j}}\frac{\partial \emph{g}}{\partial x^{p+j}}.
\end{equation*}
The gradient of $f$ takes the form
\begin{equation*}
X_f=\sum_{i=1}^p\frac{\partial f}{\partial x^i}\partial_{x^i}-\sum_{j=1}^q\frac{\partial f}{\partial x^{p+j}}\partial_{x^{p+j}},
\end{equation*}
which specializes to the usual gradient when $q=0$.
More generally, if $(M,g)$ is a flat \mbox{(pseudo-)Rie}mannian manifold of signature $(p,q)$, there exist local coordinates around every point of $M$ on which the above expressions hold.
\end{example}

\begin{remark}
 Note that, for $X\in\mathfrak{X}(M)$, the symmetric bivector fields $X\otimes X$ and $\frac{1}{2} X\odot X$ are equal. As in Example \ref{ex: flat}, we will mostly use the notation $X\otimes X$.
\end{remark}

\subsection{Integrability for Poisson structures}\label{sec:int-Poisson}

The most common formulations of the integrability of a Poisson structure $\pi\in\mathfrak{X}^2(M):=\Gamma(\wedge^2TM)$ are: 
\begin{enumerate}[(i), leftmargin=35pt]
 \item $[\pi,\pi]=0$ for the Schouten bracket \cite{SchSSB}, a natural extension of the Lie bracket on vector fields to skew-symmetric multivector fields.
 \item the vanishing of the Jacobiator for the Poisson bracket $\{f,\cg\}_\pi:=\pi(\dif f,\dif\cg)$.
 \item\sloppy the Hamiltonian map $\Ham\!: \cCi(M)\to \mathfrak{X}(M)$, given by $\Ham f:=\pi(\dif f)$, being an algebra morphism, that is, $\Ham \{ f,\cg\}_\pi=[\Ham f,\Ham \cg]$.
\end{enumerate}

For a symmetric bivector field $\vartheta$, the direct analogues of these three properties are not suitable. To start with, the Schouten bracket for symmetric multivector fields, also introduced in \cite{SchSSB} and whose definition we will recall in Section \ref{sec: analog-can-2-form}, is anti-commutative (not graded anti-commutative), so the condition $[\vartheta,\vartheta]=0$ is trivially satisfied. With respect to the Jacobiator of the symmetric Poisson bracket $\{\,,\}$ in \eqref{eq:sym-bracket}, there are two possible options (which coincide in the skew-symmetric case). The first is the cyclic Jacobiator,
\begin{equation}\label{eq:Jacobiator}
\Jac(f,\cg,h):=\{ f, \{\cg, h\}\} + \{ \cg, \{h, f\}\} + \{h,\{ f,\cg\}\}, 
\end{equation}
for which one can check 
\[
\Jac(f^2,f,f)=2f\Jac(f,f,f)+4\{f,f\}^2,
\] 
so the vanishing of the Jacobiator implies $\{f,f\}=0$. By polarization, we get that the bracket, and hence $\vartheta$, vanishes. The second option, which is the vanishing of the Jacobiator
\[
\Jac'(f,\cg,h):=\{ f, \{\cg, h\}\} - \{ \{f,\cg\}, h\} - \{ \cg, \{f,h\}\},
\]
is actually equivalent to the analogue of the third condition,
\begin{align*}
 X_{\{f,\, \cg\}}&= [X_f,X_\cg].
\end{align*}
Similarly as for the first Jacobiator, we find
\[\jac'(f^2,f,f)=2f\jac'(f,f,f)-4\po{f,f}^2,\]
hence the vanishing of the second Jacobiator also implies $\vartheta=0$.

We will thus need a different notion to talk about integrability for symmetric bivector fields and, as we have mentioned in the Introduction, we resort to symmetric Cartan calculus, which intrinsically depends on the choice of a torsion-free connection. We recall its formalism next.

\subsection{Symmetric Cartan calculus \cite{SymCartan}} We use the analogy with the Cartan calculus of the exterior algebra $\Omega^\bullet(M)$. This algebra is endowed with the exterior derivative: the only graded derivation that squares to zero and is geometric, in the sense that it is of degree $1$ and acts as $(\dif f)(X)=Xf$ for $X\in\mathfrak{X}(M)$.

We denote by $\Upsilon^r (M)$, for $r\in \N$, the $\cCi(M)$-module of $r$-symmetric forms (sections of $\Sym^r T^*M$), and by $\Upsilon^\bullet(M)$ the symmetric algebra, which is generated by all symmetric forms and is endowed with the symmetric product $\odot$. We denote by `$\sym$' the usual symmetric projection defined on $\otimes^r T^*M$, which fixes symmetric elements.

By \cite[Sec. 2.2]{SymCartan}, any geometric derivation of the symmetric algebra $\Upsilon^\bullet(M)$ is determined by a connection $\nabla$ on $M$ via the formula
\begin{equation}\label{eq: sym-der-def}
 \nabla^s\varphi:=(r+1)\sym(\nabla\varphi)
\end{equation}
for $\varphi\in\Upsilon^r(M)$. We call $\nabla^s$ the \textbf{symmetric derivative} corresponding to~$\nabla$. Without loss of generality, see \cite[Prop. 2.9]{SymCartan}, we may choose $\nabla$ in \eqref{eq: sym-der-def} to be torsion-free. Elements in the kernel of the symmetric derivative are precisely Killing tensors for $\nabla$. The space of Killing $r$-tensors we denote by $\kil^r_\nabla(M)$ and the direct sum of these spaces over $r\in\N$ by $\kil^\bullet_\nabla(M)$. 

From $\nabla^s$, we can introduce the \textbf{symmetric Lie derivative}
\begin{equation}\label{eq: sym-Lie-der}
 L^s_X:=[\iota_X,\nabla^s]=\iota_X\circ\nabla^s-\nabla^s\circ\iota_X
\end{equation}
and derive the \textbf{symmetric bracket} by imposing
\begin{equation}\label{eq: sym-bracket-derived}
 \iota_{\pg{X,Y}_{s}}=[L^s_X,\iota_Y],
\end{equation}
which yields
\begin{equation}\label{eq: sym-br-def}
 \pg{X,Y}_s:=\nabla_XY+\nabla_YX.
\end{equation}

\subsection{The symmetric Schouten bracket}\label{sec:sym-Schouten} 

Once we have recalled the symmetric bracket, we introduce the symmetric Schouten bracket as its extension by the derivation property and commutativity. We denote the space of degree-$r$ symmetric multivector fields by $\symr:=\Gamma(\Sym^r TM)$ and the space of all symmetric multivector fields, that is, the direct sum of $\symr$ for $r\in \N$, by $\symall$ (note that this is not the space of sections of a finite-rank vector bundle).

\begin{definition}\label{def: s-Schouten}
For a connection $\nabla$ on $M$, the symmetric \textbf{Schouten bracket} is the unique $\R$-bilinear map
\begin{equation*}
[\,\,,\,]_s:\symall \times \symall \rightarrow \symall 
\end{equation*}
such that 
\begin{itemize}
\item $[X,f]_s=Xf$ and, on vector fields, it is given by \eqref{eq: sym-br-def},
\item $[\mathcal{X},\,\,]_s$ is a degree-$(r-1)$ derivation of $(\symall,\odot)$ for $\mathcal{X}\in\symr$,
\item $[\mathcal{X},\mathcal{Y}]_s=[\mathcal{Y},\mathcal{X}]_s$ for every $\mathcal{X}$, $\mathcal{Y}\in\symall$.
\end{itemize}
\end{definition}

The existence and uniqueness of the symmetric Schouten bracket follows by its values on decomposable elements. Using the axioms, for $X_j, Y_i\in\mathfrak{X}(M)$, we get
\begin{multline}\label{eq: sym-Schouten-decomposable}
[X_1\odot\ldots\odot X_r,Y_1\odot\ldots\odot Y_l]_{s}\\ =\sum_{j=1}^r\sum_{i=1}^l\pg{X_j,Y_i}_{s}\odot X_1\odot \ldots\hat{X}_j\ldots\odot X_r\odot Y_1\odot \ldots \hat{Y}_i\ldots \odot Y_l.
\end{multline}

There is an explicit expression for the bracket of two generic elements, where we use the symbol $\star$ to denote the factors we are taking the trace on.

\begin{proposition}\label{prop: s-Schouten-explicit}
The symmetric Schouten bracket corresponding to a connection $\nabla$ on $M$ is given, for $\mathcal{X}$, $\mathcal{Y}\in\mathfrak{X}^\bullet_\emph{sym}(M)$, by 
\begin{equation}\label{eq: s-Schouten-formula}
[\mathcal{X},\mathcal{Y}]_s=\tr(\iota_\star\mathcal{X}\odot\nabla_\star\mathcal{Y}+\nabla_\star\mathcal{X}\odot\iota_\star\mathcal{Y}).
\end{equation}
\end{proposition}

\begin{proof}
We show that the expression on the right-hand side of \eqref{eq: s-Schouten-formula},
\begin{equation*}
[[\mathcal{X},\mathcal{Y}]]:=\tr(\iota_\star\mathcal{X}\odot\nabla_\star\mathcal{Y}+\nabla_\star\mathcal{X}\odot\iota_\star\mathcal{Y}),
\end{equation*}
satisfies all the axioms of Definition \ref{def: s-Schouten}. The result will then follow from the uniqueness of the symmetric Schouten bracket. The bracket $[[\,\,,\,]]$ is clearly an $\R$-bilinear map satisfying $[[\mathcal{X},\mathcal{Y}]]=[[\mathcal{Y},\mathcal{X}]]$. Moreover,
\begin{align*}
 [[X,f]]&=\tr(X\otimes \dif f)=Xf,\\
 [[X,Y]]&=\tr(X\otimes\nabla Y+\nabla X\otimes Y)=\pg{X,Y}_s. 
\end{align*}
Since $[[\mathcal{X},\,\,]]:\symall\rightarrow\symall$ is a degree-$(r-1)$ linear map for $\mathcal{X}\in\mathfrak{X}^r_\text{sym}(M)$, it remains to check that $[[\mathcal{X},\,\,]]$ is a derivation. 
This follows easily from the associativity of the symmetric product and the fact that both the contraction by a $1$-form and the covariant derivative are derivations of the algebra of symmetric multivector fields.
\end{proof}

In Appendix \ref{app: compl-s-Schouten}, we use the symmetric Schouten bracket to characterize the Killing tensors of a \mbox{(pseudo-)Rie}mannian metric and also show that it can be interpreted as a derived bracket. 

\section{Symmetric Poisson structures}

We now introduce integrability conditions based on the analogy between classical and symmetric Cartan calculus. Since there is a background connection $\nabla$ (which, as we recalled, we can always assume to be torsion-free), we will include it in our definitions. 

\subsection{Definition of a symmetric Poisson structure} We first make use of the symmetric Schouten bracket.

\begin{definition}
 We call a pair $(\vartheta,\nabla)$ consisting of $\vartheta\in\symbi$ and a torsion-free connection $\nabla$ on $M$ a \textbf{symmetric Poisson structure} if $[\vartheta,\vartheta]_s=0$.
\end{definition}

As we saw in Section \ref{sec:int-Poisson}, the vanishing of any Jacobiator is only possible for the zero bracket. However, we can still explore what the condition $[\vartheta,\vartheta]_s=0$ means in terms of the symmetric Poisson bracket $\{\, , \}$.

\begin{proposition}\label{prop: sym-Poisson-bracket}
Given $\vartheta\in\mathfrak{X}^2_\emph{sym}(M)$ and a torsion-free connection $\nabla$ on $M$. For $\alpha$, $\beta$, $\eta\in\Omega^1(M)$, we have
\begin{equation}\label{eq:theta-theta-in-terms-of-nabla}
\frac{1}{2}[\vartheta,\vartheta]_{s}(\alpha,\beta,\eta)=(\nabla_{\vartheta(\alpha)} \vartheta)(\beta,\eta)+\cyc(\alpha,\beta,\eta)
\end{equation}
Moreover, $(\vartheta,\nabla)$ is a symmetric Poisson structure if and only if we have the following for the cyclic Jacobiatior of the corresponding symmetric Poisson bracket:
\begin{equation}\label{eq:int-sym-Poisson-Jac}
\jac(f,\emph{g},h)=\dif h(\pg{X_f,X_\emph{g}}_s)+\cyc(f,\emph{g},h)
\end{equation}
\end{proposition}

\begin{proof}
 By Proposition \ref{prop: s-Schouten-explicit}, we get 
 \begin{align*}
 \frac{1}{2}[\vartheta,\vartheta]_s(\alpha,\beta,\eta)&=\tr(\iota_\star\vartheta\odot\nabla_\star\vartheta)(\alpha,\beta,\eta)=\frac{3!}{2}(\sym\nabla_{\vartheta(\,\,)}\vartheta)(\alpha,\beta,\eta)\\
 &=(\nabla_{\vartheta(\alpha)}\vartheta)(\beta,\eta)+\cyc(\alpha,\beta,\eta).
 \end{align*}
 From this, the fact that $\Omega^1(M)$ is locally generated by exact $1$-forms and polarization, it follows that $(\vartheta,\nabla)$ is a symmetric Poisson structure if and only if
 \begin{equation*}
 0=(\nabla_{X_f}\vartheta)(\dif f,\dif f)=X_f(\po{f,f})-2\vartheta(\nabla_{X_f}\dif f,\dif f).
 \end{equation*}
 The first term on the right-hand side is equal to $\po{f,\po{f,f}}$, whereas for the second term we have the following:
 \begin{equation*}
 \vartheta(\nabla_{X_f}\dif f,\dif f)=(\nabla_{X_f}\dif f)(X_f)=\po{f,\po{f,f}}-\dif f(\nabla_{X_f}X_f).
 \end{equation*}
Therefore, $(\vartheta,\nabla)$ is a symmetric Poisson structure if and only if
\begin{equation*}
 \po{f,\po{f,f}}=\dif f(\pg{X_f,X_f}_s).
\end{equation*}
The result follows by polarization.
\end{proof}

Proposition \ref{prop: sym-Poisson-bracket} allows us to give an equivalent definition of a symmetric Poisson structure using the bracket instead of the symmetric bivector field. It is a pair $(\lbrace\,,\rbrace,\nabla)$ consisting of an $\R$-bilinear map 
\[\{ \,,\}: \cCi(M)\times\cCi(M)\rightarrow \cCi(M)\]
and a torsion-free connection $\nabla$ on $M$ such that
 \begin{itemize}
 \item $\{ f,\cg\}=\{ \cg,f\}$,
\item $\{ f,\cg h\}=\{ f,\cg\} h+\cg\{ f,h\} $,
 \item $\jac(f,\cg,h)=\dif h(\pg{X_f,X_\cg}_s)+\cyc(f,\cg,h)$.
\end{itemize}

\begin{remark}
 Given a skew-symmetric bivector field $\pi\in\mathfrak{X}^2(M)$, we have 
 \begin{equation*}
 [\pi,\pi](\alpha,\beta,\eta)=-(\nabla_{\pi(\alpha)}\pi)(\beta,\eta)+\cyc(\alpha,\beta,\eta)
 \end{equation*}
 for any choice of a torsion-free connection $\nabla$ on $M$. Therefore, $\pi$ is a Poisson structure if and only if
 \begin{equation}\label{eq: Poisson-nabla}
 (\nabla_{\pi(\alpha)}\pi)(\beta,\eta)+\cyc(\alpha,\beta,\eta)=0,
 \end{equation}
 which is analogous to the vanishing of \eqref{eq:theta-theta-in-terms-of-nabla}. In contrast to the symmetric case, the choice of $\nabla$ here is auxiliary (and the vanishing of \eqref{eq: Poisson-nabla} is independent of it).
\end{remark}

\subsection{Strong symmetric Poisson structures} Following our strategy to define integrability for $(\vartheta,\nabla)$, the condition that $\Ham:\cCi(M)\rightarrow\mathfrak{X}(M)$ is an algebra morphism is translated into the language of symmetric Cartan calculus as 
\[\grad:(\cCi(M),\lbrace\,,\rbrace)\rightarrow (\mathfrak{X}(M),\pg{\,\,,\,}_{s})\]
being an algebra morphism, that is, $X_{\{ f,\cg\}}=\pg{X_f,X_\cg}_{s}$. It happens if and only if
\[\po{\po{f,\cg},h}=\dif h(\pg{X_f,X_\cg}_{s}),\]
which is, a priori, a stronger condition than \eqref{eq:int-sym-Poisson-Jac} for symmetric Poisson structures.

\begin{definition}\label{def: ssPs}
We call a pair $(\lbrace\,,\rbrace,\nabla)$ consisting of an $\R$-bilinear map
\begin{equation*}
\{ \,,\}:\cCi(M)\times\cCi(M)\rightarrow \cCi(M)
\end{equation*}
 and a torsion-free connection $\nabla$ on $M$ a \textbf{strong symmetric Poisson structure}~if 
\begin{itemize}
\item $\{ f,\cg\}=\{ \cg,f\}$,
\item $\{ f,\cg h\}=\{ f,\cg\} h+\cg\{ f,h\}$,
\item $X_{\{ f,\,\cg\}}=\pg{X_f,X_\cg}_{s}$.\label{ssP3}
\end{itemize}
\end{definition}

We can also describe these in terms of the associated bivector field.

\begin{proposition}\label{prop: strong-sym-Poisson-bivector}
A strong symmetric Poisson structure is equivalently given by a pair $(\vartheta,\nabla)$ consisting of $\vartheta\in \mathfrak{X}^2_\emph{sym}(M)$ and a torsion-free connection $\nabla$ on $M$ such that
\begin{equation*}
\nabla_{X_f}\vartheta=0.
\end{equation*}
\end{proposition}

\begin{proof}
For every $f,\cg\in\cCi(M)$, we have
\begin{align*}
 \dif \cg(X_{\po{f,f}}) &=X_\cg(\vartheta(\dif f,\dif f))=(\nabla_{X_\cg}\vartheta)(\dif f,\dif f)+2\vartheta(\nabla_{X_\cg}\dif f,\dif f)\\
 &=(\nabla_{X_\cg}\vartheta)(\dif f,\dif f)+2(\nabla_{X_\cg}\dif f)(X_f),
\end{align*}
while on the other hand,
\begin{align*}
 \dif \cg(\pg{X_f,X_f}_s)&=2\dif \cg(\nabla_{X_f}X_f)=2(\nabla_{X_f}\vartheta)(\dif f,\dif \cg)+2\vartheta(\nabla_{X_f}\dif f,\dif \cg)\\
 &=2(\nabla_{X_f}\vartheta)(\dif f,\dif \cg)+2(\nabla_{X_f}\dif f)(X_\cg).
\end{align*}
Using the torsion-freeness of $\nabla$, that is, $(\nabla_X \dif h)(Y)=(\nabla_Y\dif h)(X)$, we get
\begin{equation*}
 \dif \cg(X_{\po{f,f}}-\pg{X_f,X_f}_s)= (\nabla_{X_\cg}\vartheta)(\dif f,\dif f)-2(\nabla_{X_f}\vartheta)(\dif f,\dif \cg).
\end{equation*}
By polarization, the gradient map $\grad:\cCi(M)\rightarrow\mathfrak{X}(M):f\mapsto X_f$ is an algebra morphism if and only if
\begin{equation}\label{eq: ssPS-proof}
 (\nabla_{X_\cg}\vartheta)(\dif f,\dif h)=(\nabla_{X_f}\vartheta)(\dif \cg,\dif h)+(\nabla_{X_h}\vartheta)(\dif \cg,\dif f).
\end{equation}
Finally, swapping the role of $\cg$ and $h$ in \eqref{eq: ssPS-proof} and adding these two equations gives the result.
\end{proof}

As exact $1$-forms locally generate the space $\Omega^1(M)$, we can, using Proposition \ref{prop: strong-sym-Poisson-bivector}, formulate the integrability condition for strong symmetric Poisson as follows.

\begin{corollary}\label{cor: ssPs-integrability}
 A pair $(\vartheta,\nabla)$ is a strong symmetric Poisson structure if and only if $\nabla_{\vartheta(\alpha)}\vartheta=0$ for all $\alpha\in \Omega^1(M)$. 
 \end{corollary}

The following example is the analogue of the zero Poisson structure.

\begin{example}[Torsion-free connections]\label{ex: tf-connections}
 Trivially, every manifold carries a family of strong symmetric Poisson structures
 \begin{equation*}
 \{ (0,\nabla)\,|\,\nabla \text{ is a torsion-free connection on }M\}.
 \end{equation*}
\end{example}

By Corollary \ref{cor: ssPs-integrability}, if the bivector field $\vartheta$ is parallel with respect to a torsion-free connection $\nabla$, that is, $\nabla\vartheta=0$, the pair $(\vartheta,\nabla)$ is a strong symmetric Poisson structure. This gives the following hierarchy of the integrability conditions:
\begin{equation*}
 \left\{\begin{array}{c}
\text{parallel symmetric}\\
 \text{bivector fields}
 \end{array}\right\}\subseteq\left\{\begin{array}{c}
\text{strong symmetric}\\
 \text{Poisson structures}
 \end{array}\right\}\subseteq\left\{\begin{array}{c}
\text{symmetric}\\
 \text{Poisson structures}
 \end{array}\right\}.
\end{equation*}

\begin{remark}\label{rk: parallel}
Note that the condition of being parallel is quite restrictive. In particular, it implies that $\vartheta:T^*M\rightarrow TM$ is of constant rank, hence, singularities, which make classical Poisson geometry so fruitful, would not be possible.
\end{remark}

Let us now show that the inclusions are actually strict.

\begin{example}\label{ex: inclusion}
 Consider the Euclidean connection $\nabla^\emph{Euc}$ on the cartesian plane $\R^2$ and $\vartheta\in \mathfrak{X}^2_\emph{sym}(\R^2)$ given by 
\begin{equation*}
\vartheta:=h\,\partial_x\otimes\partial_x,
\end{equation*}
where $h\in\cCi(\R^2)$ is an arbitrary non-constant function depending only on the variable $y$. As $\nabla^\emph{Euc}_{\partial_y}\vartheta=\frac{\partial h}{\partial y}\,\partial_x\otimes\partial_x\neq 0$, the bivector field is not parallel. However, for $f\in\cCi(\R^2)$,
\begin{equation*}
 \nabla^\emph{Euc}_{X_f}\vartheta=h\frac{\partial f}{\partial x}\nabla^\emph{Euc}_{\partial_x}\vartheta=0,
\end{equation*}
so the pair $(\vartheta,\nabla^\emph{Euc})$ is a strong symmetric Poisson structure on $\R^2$.
\end{example}

\begin{example}\label{ex: SO(3)}
We consider the natural connection $\nabla$ on $\SO(3)$ determined by the condition that all left-invariant vector fields are parallel. Let $(X_1,X_2,X_3)$ be the standard basis of left-invariant vector fields on $\SO(3)$ satisfying
\begin{equation*}
 [X_i,X_j]=\varepsilon_{ijk}X_k,
\end{equation*}
where $\varepsilon_{ijk}$ is the Levi-Civita symbol. We define $\vartheta\in\mathfrak{X}^2_\emph{sym}(\SO(3))$ by
\begin{equation*}
 \vartheta=X_1\otimes X_1+X_2\otimes X_2.
\end{equation*}
As $\pg{X_i,X_j}_s=0$, it follows easily from \eqref{eq: sym-Schouten-decomposable} that $[\vartheta,\vartheta]_s=0$, that is, $(\vartheta,\nabla^0)$ is a symmetric Poisson structure on $\SO(3)$. The associated torsion-free connection to $\nabla$ is given by
\begin{equation*}
 \nabla^0_{X_i}X_j=\frac{1}{2}[X_i,X_j]=\frac{1}{2}\varepsilon_{ijk}X_k.
 \end{equation*}
 As a covariant derivative is a derivation of the symmetric product, we find
 \begin{equation*}
 \nabla^0_{X_1}\vartheta=2(\nabla^0_{X_1}X_1)\odot X_1+2(\nabla^0_{X_1}X_2)\odot X_2=X_3\odot X_2\neq 0.
 \end{equation*}
 By Corollary \ref{cor: ssPs-integrability} and the fact that $X_1=\vartheta(\alpha)$, where $\alpha\in\Omega^1(\SO(3))$ is given by $\alpha(X_i):=\delta_{1i}$, we get that $(\vartheta,\nabla^0)$ is not a strong symmetric Poisson structure.
\end{example}

\subsection{Non-degenerate symmetric Poisson structures}\label{sec:non-deg-sym-Poisson}

 We consider now \mbox{non-deg}enerate structures. Apart from being a meaningful family of examples, they establish connections of symmetric Poisson geometry with Killing tensors and \mbox{(pseudo-)Rie}mannian geometry.

A non-degenerate symmetric bivector field $\vartheta\in\symbi$ determines the \mbox{non-deg}enerate symmetric $2$-form $\vartheta^{-1}\in\Upsilon^2(M)$, i.e., a \mbox{(pseudo-)Rie}mannian metric on $M$, and vice versa. If we look at the integrability condition, we have the following.

\begin{proposition}\label{prop: sym-Poisson-nondeg}
The assignment $(\vartheta,\nabla)\mapsto(\vartheta^{-1},\nabla)$ gives a bijection:
\begin{equation*}
\left\{\begin{array}{c}
\text{non-degenerate symmetric}\\
 \text{Poisson structures}
 \end{array}\right\}\overset{\sim}{\longleftrightarrow}\left\{\begin{array}{c}
 \text{non-degenerate}\\
 \text{Killing }2\text{-tensors}
 \end{array}\right\},
\end{equation*}
where we attach the connection for which the $2$-tensor is Killing.
\end{proposition}

\begin{proof}
It follows from Proposition \ref{prop: sym-Poisson-bracket} and the non-degeneracy of $\vartheta$ that $(\vartheta,\nabla)$ is a symmetric Poisson structure if and only if
\begin{equation*}
 0=(\nabla_{X}\vartheta)(\vartheta^{-1}(Y),\vartheta^{-1}(Z))+\cyc(X,Y,Z)
\end{equation*}
for every $X,Y,Z\in\mathfrak{X}(M)$. For each term, we have
\begin{equation}\label{eq: inverse-parallel}
\begin{split}
(\nabla_{X}\vartheta)&(\vartheta^{-1}(Y),\vartheta^{-1}(Z))\\
&=\vartheta^{-1}(\nabla_XY,Z)-(\nabla_X\vartheta^{-1}(Y))(Z)\\
&=\vartheta^{-1}(\nabla_XY,Z)-(\nabla_X\vartheta^{-1})(Y,Z)-\vartheta^{-1}(\nabla_XY,Z)\\
 &=-(\nabla_X\vartheta^{-1})(Y,Z).
\end{split}
\end{equation}
Therefore, being a symmetric Poisson structure reads, in terms of the inverse symmetric $2$-form, as
\begin{equation*}
 0=(\nabla_X\vartheta^{-1})(Y,Z)+\cyc(X,Y,Z)=(\nabla^s\vartheta^{-1})(X,Y,Z),
\end{equation*}
that is, $\vartheta^{-1}\in\kil^2_\nabla(M)$.
\end{proof}

Similarly, non-degenerate strong symmetric Poisson structures can be identified with \mbox{(pseudo-)Rie}mannian metrics.

\begin{proposition}\label{prop: ssPs-nondeg}
 Let $\vartheta\in\mathfrak{X}^2_\emph{sym}(M)$ be non-degenerate. Then $(\vartheta,\nabla)$ is a strong symmetric Poisson structure if and only if $\nabla$ is the Levi-Civita connection of $\vartheta^{-1}$. Consequently, the assignment $(\vartheta,\nabla)\mapsto\vartheta^{-1}$ gives a bijection:
\begin{equation*}
\left\{\begin{array}{c}
\text{non-degenerate strong}\\
 \text{symmetric Poisson structures}
 \end{array}\right\}\overset{\sim}{\longleftrightarrow}\left\{\begin{array}{c}
 \text{\mbox{(pseudo-)Rie}mannian}\\
 \text{metrics}
 \end{array}\right\}.
\end{equation*}
\end{proposition}

\begin{proof}
 As gradient vector fields locally generate the space of all vector fields when $\vartheta$ is non-degenerate, it follows from Proposition \ref{prop: strong-sym-Poisson-bivector} that $(\vartheta, \nabla)$ is a strong symmetric Poisson structure if and only if $\vartheta$ is parallel with respect to $\nabla$. By \eqref{eq: inverse-parallel}, this is equivalent to $\vartheta^{-1}$ being parallel. Since $\nabla$ is torsion-free, it is necessarily the \mbox{Levi-Civ}ita connection of the \mbox{(pseudo-)Rie}mannian metric $\vartheta^{-1}$.
\end{proof}

Thus, under the assumptions of being strong and non-degenerate, the information of a symmetric Poisson structure $(\vartheta,\nabla)$ is all contained in $\vartheta$. Indeed, symmetric Poisson structures can be regarded as a degeneration of \mbox{(pseudo-)Rie}mannian metrics, analogously to how symplectic structures degenerate into standard Poisson structures.

\sloppy The next example shows that, when we are not in the strong case, \mbox{(pseudo-)Rie}mannian metrics can form a symmetric Poisson structure with a connection other than the Levi-Civita one.

\begin{example}\label{ex: non-deg-Kill}
Consider the torsion-free connection $\nabla$ on $\R^2$ given by
\begin{align*}
\nabla_{\partial_x}\partial_x&=\nabla_{\partial_y}\partial_y:=0, & \nabla_{\partial_x}\partial_y&:=\partial_x+\partial_y
\end{align*}
and the split signature metric $g:=(\e^{2y}\,\dif x)\odot (\e^{2x}\,\dif y)$. We find that
\begin{equation*}
 (\nabla_{\partial_x}g)(\partial_y,\partial_y)=-2g(\nabla_{\partial_x}\partial_y,\partial_y)=-2g(\partial_x,\partial_y)=-2\e^{2(x+y)}\neq 0,
\end{equation*}
that is, $\nabla$ is not the Levi-Civita connection of $g$. On the other hand,
\begin{equation*}
 \nabla^s(\e^{2y}\,\dif x)=2\e^{2y}\,\dif y\odot\dif x+\e^{2y}\,\nabla^s\dif x=2\e^{2y}\,\dif y\odot\dif x-2\e^{2y}\,\dif y\odot\dif x=0.
 \end{equation*}
 Analogously, we have $\nabla^s(\e^{2x}\,\dif y)=0$. This means that $\e^{2y}\,\dif x$ and $\e^{2x}\,\dif y$ are Killing 1-tensors, hence $g$ is a Killing $2$-tensor. In particular, by Proposition \ref{prop: sym-Poisson-nondeg}, the pair $(g^{-1},\nabla)$ is indeed a symmetric Poisson structure.
\end{example}

\section{\for{toc}{Geometric interpretation of symmetric Poisson}\except{toc}{Geometric interpretation\\ of symmetric Poisson structures}}\label{sec: geo-int}
In this section, we provide a geometric interpretation of symmetric Poisson structures in terms of distributions 
 and (pseudo-)Riemannian geometry. Some of the proofs of Sections \ref{sec: geo-int} and \ref{sec: geo-int-strong} have been moved to Appendix \ref{app: some-proofs} for convenience.

\subsection{The characteristic distribution/module of a symmetric bivector field}
A symmetric bivector field naturally determines a distribution and a $\cCi(M)$-submodule of $\mathfrak{X}(M)$.

\begin{definition}
 For $\vartheta\in\symbi$, its \textbf{characteristic distribution} is the subset
 \begin{equation*}
 \im\vartheta:=\{ \vartheta(\zeta)\in TM\,|\, \zeta\in T^*M\}\subseteq TM,
 \end{equation*}
and its \textbf{characteristic module} is the $\cCi(M)$-submodule of $\mathfrak{X}(M)$ given by
 \begin{equation*} \mathcal{F}_\vartheta:=\{\vartheta(\alpha)\in\mathfrak{X}(M)\,|\,\alpha\in\Omega^1(M)\}.
\end{equation*}
\end{definition}
By linearity of $\vartheta$, $\im \vartheta$ is indeed a distribution on $M$. As such, it is smooth: given $u\in(\im\vartheta)_m$, any $\zeta\in T^*_mM$ satisfying $\vartheta(\zeta)=u$ can be extended to $\alpha\in\Omega^1(M)$. 

The characteristic distribution is recovered from the characteristic module by
\begin{equation}\label{eq: char-d-m}
 \im\vartheta=\bigcup_{m\in M}\lbrace X_m\in T_mM\,|\,X\in\mathcal{F}_\vartheta\rbrace.
\end{equation}
On the other hand, although $\mathcal{F}_\vartheta$ is a $\cCi(M)$-submodule of the space of global sections $\Gamma(\im\vartheta)$, these modules are not equal in general. 

\begin{example}\label{ex: sing-d}
For $\vartheta:=x^2\,\partial_x\otimes\partial_x\in \mathfrak{X}^2_{\sym}(\R)$, the characteristic distribution is 
 \begin{align*}
 (\im\vartheta)_0&=\{ 0\}, & (\im\vartheta)_x&=T_x\R,
 \end{align*}
 for $x\neq 0$, whereas the characteristic module $\mathcal{F}_\vartheta$ is generated by the vector field $x^2\partial_x$ as a $\cCi(\R)$-module. 
 Note that $x\partial_x\in\Gamma(\im\vartheta)$ but $x\partial_x\not\in \mathcal{F}_\vartheta$.
\end{example}

A key notion in this discussion is regularity.

\begin{definition}
 We call $\vartheta\in\symbi$ \textbf{regular} or \textbf{singular} when the characteristic distribution of $\vartheta$ is regular (constant rank) or singular (otherwise).
\end{definition}

\begin{lemma}\label{lem: reg-d}
 If $\vartheta\in\mathfrak{X}^2_\emph{sym}(M)$ is regular, we have that $\Gamma(\im\vartheta)=\mathcal{F}_\vartheta$. 
\end{lemma}

\begin{proof}
 For a proof, see Appendix \ref{app: some-proofs}.
\end{proof}

Thus, when $\vartheta$ is regular, the characteristic distribution and module give exactly the same information.

\subsection{The characteristic distribution/module of a symmetric Poisson structure}
We show that the characteristic distribution and module of a symmetric Poisson structure satisfy a compatibility relation with the background connection.

\begin{proposition}\label{prop: sPs-char-d}
 Given a symmetric Poisson structure $(\vartheta,\nabla)$, the characteristic module is preserved by the symmetric bracket, that is,
 \begin{equation*}
 \pg{\mathcal{F}_\vartheta,\mathcal{F}_\vartheta}_s\subseteq\mathcal{F}_\vartheta.
 \end{equation*} 
\end{proposition}

\begin{proof}
For $\alpha,\beta\in\Omega^1(M)$, we find
 \begin{align*}
\pg{\vartheta(\alpha),\vartheta(\beta)}_s&=\nabla_{\vartheta(\alpha)}\vartheta (\beta)+\nabla_{\vartheta(\beta)}\vartheta (\alpha)\\
&=(\nabla_{\vartheta(\alpha)}\vartheta)(\beta)+\vartheta(\nabla_{\vartheta(\alpha)}\beta)+(\nabla_{\vartheta(\beta)}\vartheta)(\alpha)+\vartheta(\nabla_{\vartheta(\beta)}\alpha)
 \end{align*}
 By Proposition \ref{prop: sym-Poisson-bracket}, we get
\begin{align*}
 \pg{\vartheta(\alpha),\vartheta(\beta)}_s&=\vartheta(\nabla_{\vartheta(\alpha)}\beta+\nabla_{\vartheta(\beta)}\alpha)-(\nabla_{\vartheta(\,)}\vartheta)(\alpha,\beta)\\
 &=\vartheta(\nabla_{\vartheta(\alpha)}\beta+\nabla_{\vartheta(\beta)}\alpha-(\nabla\vartheta)(\alpha,\beta)),
\end{align*}
hence the characteristic module is preserved by the symmetric bracket.
\end{proof}

This fact will have a clear geometric interpretation (Corollary \ref{cor: char-reg-sPs}) when $\vartheta$ is regular. First recall that, given a connection $\nabla$ on $M$, a smooth distribution $\Delta\subseteq TM$ is \textbf{geodesically invariant} (\cite{LewACD}) if each geodesic $\gamma :I\rightarrow M$ (for $I$ an open interval) such that $\dot{\gamma}(t_0)\in \Delta_{\gamma(t_0)}$ for some $t_0\in I$ satisfies $\dot{\gamma}(t)\in \Delta_{\gamma(t)}$ for all $t\in I$. There is a relation between the symmetric bracket and geodesically invariant distributions similar to \textit{Frobenius' theorem} for integrable distributions.

\begin{theorem}[\cite{LewACD}]\label{thm: LewisT}
For a connection $\nabla$ on $M$, a regular distribution $\Delta\subseteq TM$ is geodesically invariant if and only if
\begin{equation*}
\pg{\Gamma(\Delta),\Gamma(\Delta)}_{s}\subseteq \Gamma(\Delta).
\end{equation*}
\end{theorem}
Note that the assumption of regularity in Theorem \ref{thm: LewisT} cannot be removed.
\begin{example}\label{ex:delta-not-geod-inv}
 On $(\R,\nabla^\emph{Euc})$, consider the singular smooth distribution $\Delta$ given~by
 \begin{align*}
 \Delta_0&=\{ 0\}, & \Delta_x&=T_x\R
 \end{align*}
 for $x\neq 0$. Each vector field $X\in\Gamma(\Delta)$ is of the form $X=f\partial_x$ for some $f\in\cCi(\R)$ such that $f(0)=0$. As, for every $f,h\in\cCi(\R)$, we have that 
 \begin{equation*}
\pg{f\partial_x,h\partial_x}_s=((\partial_xf)h+f(\partial_x h))\partial_x,
 \end{equation*} 
the space $\Gamma(\Delta)$ is preserved by the symmetric bracket. However, the geodesic $\gamma(t):=t$ is tangent to $\Delta$ if and only if $t\neq 0$. So, $\Delta$ is not geodesically invariant.
\end{example}

We have a direct consequence of Lemma \ref{lem: reg-d}, Proposition \ref{prop: sPs-char-d} and Theorem \ref{thm: LewisT}.

\begin{corollary}\label{cor: char-reg-sPs}
 The characteristic distribution of a regular symmetric Poisson structure is geodesically invariant.
\end{corollary}

As we will see later (Theorem \ref{thm: gen-Lewis}), the statement of Corollary \ref{cor: char-reg-sPs} is locally true even for singular symmetric Poisson structures.

The following example illustrates the notion of characteristic distribution and shows that a geodesically invariant distribution may not be involutive. 

\begin{example}\label{ex: SO(3)2}
 Let us return to Example \ref{ex: SO(3)}, the symmetric Poisson structure $(\vartheta:=X_1\otimes X_1+X_2\otimes X_2,\nabla^0)$ on $\SO(3)$, with $\nabla$ is the natural connection (with torsion) on $\SO(3)$. The characteristic distribution at $m\in\SO(3)$ is given by
 \begin{equation*}
 (\im\vartheta)_m=\spann\{\rest{X_1}{m}, \rest{X_2}{m}\}.
 \end{equation*}
 Is is regular, hence, by Corollary \ref{cor: char-reg-sPs}, it is geodesically invariant. However, since
 \begin{equation*}
 [X_1,X_2]=X_3,
 \end{equation*}
 it is not integrable by Frobenius' theorem.
\end{example}

\subsection{The characteristic metric of a symmetric bivector field}\label{sec:char-metric-bivector}
Given a symmetric bivector field $\vartheta\in\symbi$, each vector space $(\im\vartheta)_m\leq T_mM$, for $m\in M$, acquires the canonical linear \mbox{(pseudo-)Rie}mannian metric $g_{\vartheta_m}$ given by
\begin{equation}\label{eq: vartheta-char-metric}
 g_{\vartheta_m}(\vartheta(\alpha),\vartheta(\beta)):=\vartheta(\alpha,\beta), 
\end{equation}
\sloppy for $\alpha,\beta\in T^*_mM$. Clearly, if $(p,q,n-(p+q))$ is the signature of $\vartheta_m$, the metric $g_{\vartheta_m}$ is of signature $(p,q)$. Pointwise also the converse is true.

\begin{proposition}\label{prop: family-metrics}
 Let $(\Delta,g)$ be a \mbox{(pseudo-)Rie}mannian vector subspace of a real $n$-dimensional vector space $V$. There is a unique $\vartheta\in\Sym^2V$ such that $\im\vartheta=\Delta$ and, for $\alpha,\beta\in V^*$,
 \begin{equation*}
 \vartheta(\alpha,\beta)=g(\vartheta(\alpha),\vartheta(\beta)).
 \end{equation*}
 Moreover, if $g$ is of signature $(p,q)$, the bivector $\vartheta$ is of signature $(p,q,n-(p+q))$. 
 \end{proposition}

 \begin{proof}
 For a proof, see Appendix \ref{app: some-proofs}.
\end{proof}

Every symmetric bivector field is thus fully codified by the distribution and the pointwise linear \mbox{(pseudo-)Rie}mannian metric on the subspace given by this distribution. However, not every choice of distribution and pointwise metric, even if it seems to vary smoothly, determines a symmetric bivector field that is smooth, as we show next. 

\begin{example}
 Consider the smooth singular distribution $\Delta$ on $\R$ given by
 \begin{align*}
 \Delta_{0}&=\{ 0\}, & \Delta_{x}=T_x\R
\end{align*}
for $x\neq0$. For every $x\in\R$, consider the linear Riemannian metric on $\Delta_x$ 
\begin{equation*}
 g_x:=x^2\,\rest{\dif x}{x}\otimes \rest{\dif x}{x}.
\end{equation*}
By Proposition \ref{prop: family-metrics}, we get $\vartheta_x\in\Sym^2 T_x\R$ for each $x\in\R$. Explicitly, 
\begin{align*}
 \vartheta_0&=0, & \vartheta_x&=\frac{1}{x^2}\,\rest{\partial_x}{x}\otimes\rest{\partial_x}{x}
\end{align*}
for $x\neq 0$, which does not define a global smooth symmetric bivector field on $\R$.
\end{example}

\begin{definition}\label{def: char-metric}
 For $\vartheta\in\mathfrak{X}^2_\text{sym}(M)$, we call the collection of linear \mbox{(pseudo-)Rie}mannian metrics $\{ g_{\vartheta_m}\}_{m\in M}$, in \eqref{eq: vartheta-char-metric}, the \textbf{characteristic metric} of~$\vartheta$.
\end{definition}

The smoothness of $\vartheta\in\symbi$ can be taken as the definition of smoothness for the characteristic metric. We see what it implies in terms of the map
\begin{equation}\label{eq: char-met-sing}
\begin{split}
 g_\vartheta:&\,\mathcal{F}_\vartheta\times\mathcal{F}_\vartheta\rightarrow\{ \text{maps }M\rightarrow \R\}\\
 &(X,Y)\mapsto ( m\mapsto g_{\vartheta_m}(X_m,Y_m)).
\end{split}
\end{equation}

\begin{lemma}\label{lem: gen-char-met}
 Let $\vartheta\in\mathfrak{X}^2_\emph{sym}(M)$. The map $g_\vartheta$, given by \eqref{eq: char-met-sing}, is symmetric, $\cCi(M)$-bilinear and its codomain is the space $\cCi(M)$.
\end{lemma}

\begin{proof}
 Symmetry and $\cCi(M)$-bilinearity follow from the pointwise definition of $g_\vartheta$. For $X=\vartheta(\alpha), Y=\vartheta(\beta)\in\mathcal{F}_\vartheta$ (for some $\alpha,\beta\in\Omega^1(M)$), we have
 \begin{equation*}
 (g_\vartheta(X,Y))(m)=g_{\vartheta_m}(X_m,Y_m)=\vartheta(\alpha_m,\beta_m)
 \end{equation*}
 for every $m\in M$, that is, $g_\vartheta(X,Y)=\vartheta(\alpha,\beta)$, hence $g_\vartheta(X,Y)\in\cCi(M)$.
\end{proof}

If $\vartheta\in\mathfrak{X}^2_\text{sym}(M)$ is regular, the characteristic distribution $\im\vartheta$ is a subbundle and the characteristic metric corresponds, by Lemmas \ref{lem: reg-d} and \ref{lem: gen-char-met}, to a smooth~section
 \begin{equation*}
M\rightarrow \Sym^2(\im\vartheta)^*.
 \end{equation*} 
In particular, if $\vartheta$ is non-degenerate, it is an element of $\Upsilon^2(M)$: the inverse of $\vartheta$.

\begin{remark}
 We will also use the name `characteristic metric' for the \mbox{$\cCi(M)$-bi}linear map $g_\vartheta:\mathcal{F}_\vartheta\times\mathcal{F}_\vartheta\rightarrow\cCi(M)$.
\end{remark}

We illustrate the notion of characteristic metric with an example, which indicates that this framework might be useful for studying metrics with singularities.

\begin{example}\label{ex: char-metric}
 Consider $\vartheta:=x^3\,\partial_x\otimes\partial_x\in \mathfrak{X}^2_\emph{sym}(\R)$, whose characteristic distribution~is $(\im\vartheta)_{0}=\{ 0\}$, $(\im\vartheta)_{x}=T_x\R$ for $x\neq 0$. Its characteristic metric~is
\begin{equation*}
 g_{\vartheta_{x}}=\frac{1}{x^3}\rest{\dif x}{x}\otimes\rest{\dif x}{x}
\end{equation*}
for $x\neq 0$, and $0$ at the origin. The metric $g_\vartheta$ is thus negative-definite on the half-line $\R^-$ and positive-definite on $\R^+$. Note that, for $Y:=x\partial_x\in\Gamma(\im\vartheta)$, 
\begin{align*}
 (g_\vartheta(Y,Y))(x)=\begin{cases}
 0 & x=0,\\
 \frac{1}x &x\neq 0,
 \end{cases}
\end{align*}
which is clearly not a smooth function. As we proved in Lemma \ref{lem: gen-char-met}, the function above is smooth when restricted to the characteristic module, however, $Y\notin\mathcal{F}_\vartheta$.
\end{example}

Example \ref{ex: char-metric} shows that, despite the fact that the domain of the characteristic metric of $\vartheta\in\symbi$ can be extended to $\Gamma(\im\vartheta)$, the resulting function on $M$ may not be necessarily smooth. It also shows that, as the rank of $\im\vartheta$ jumps, the signature of $g_\vartheta$ may also jump. 

\subsection{The characteristic metric of a symmetric Poisson structure}\label{sec:char-metric} In this section we focus on the case when the symmetric bivector $\vartheta$ is a symmetric Poisson structure for some torsion-free connection.

It is well-known that the square of the speed of a geodesic on a \mbox{(pseudo-)Rie}mannian manifold is constant. We show a generalization of this phenomenon in symmetric Poisson geometry.

\begin{definition}\label{def:theta-admissible}
 Let $\vartheta\in\symbi$. A curve $\gamma:I\rightarrow M$, with $I$ an open interval, is called \textbf{$\vartheta$-admissible} if there is a curve $a:I\rightarrow T^*M$ such that $\vartheta(a(t))=\dot{\gamma}(t)$.
 \end{definition}

Clearly, if $\vartheta\in\symbi$ is non-degenerate, every curve on $M$ is $\vartheta$-admissible. However, it is not generally the case. 

\begin{proposition}\label{prop: const-speed}
 Let $\vartheta\in\mathfrak{X}^2_\emph{sym}(M)$ and $\gamma:I\rightarrow M$ be a $\vartheta$-admissible curve defined on some open interval $I\subseteq \R$. The function $\mathrm{Sq}_\vartheta(\dot{\gamma}):I\rightarrow \R$, given by $\mathrm{Sq}_\vartheta(\dot{\gamma})(t):= g_\vartheta(\dot{\gamma}(t),\dot{\gamma}(t))$ for $t\in I$, is smooth. If, in addition, $(\vartheta,\nabla)$ is a symmetric Poisson structure and $\gamma$ is a geodesic, the function $\mathrm{Sq}_\vartheta(\dot{\gamma})$ is constant.
\end{proposition}

\begin{proof}
 For every $t\in I$, we easily find that $\mathrm{Sq}_\vartheta(\dot{\gamma})(t)=\vartheta(a(t),a(t))$ for some curve $a:I\rightarrow T^*M$, hence $\mathrm{Sq}_\vartheta(\dot{\gamma})$ is smooth. Moreover, for any connection $\nabla$ on $M$,
 \begin{equation*}
 \frac{\dif}{\dif t}\mathrm{Sq}_\vartheta(\dot{\gamma})=\frac{\dif }{\dif t}\vartheta(a,a)=(\nabla_{\vartheta(a)}\vartheta)(a,a)+2\vartheta(\nabla_{\vartheta(a)}a,a).
 \end{equation*}
 If $(\vartheta,\nabla)$ is a symmetric Poisson structure, Propositions \ref{prop: sym-Poisson-bracket} and \ref{prop: sPs-char-d} yield
 \begin{align*}
 (\nabla_{\vartheta(a)}\vartheta)(a,a)&=0, & \vartheta(\nabla_{\vartheta(a)}a,a)&=a(\nabla_{\vartheta(a)}\vartheta(a)),
 \end{align*}
 and hence the result follows from 
 \begin{equation*}
 \frac{\dif}{\dif t}\mathrm{Sq}_\vartheta(\dot{\gamma})=2a(\nabla_{\vartheta(a)}\vartheta(a))=2a(\nabla_{\dot{\gamma}}\dot{\gamma}).\qedhere
 \end{equation*}
\end{proof}

\section{\for{toc}{Geometric interpretation of involutive and strong symmetric Poisson}\except{toc}{Geometric interpretation of involutive\\ and strong symmetric Poisson structures}}\label{sec: geo-int-strong}
We know from its very definition (or by Proposition \ref{prop: sPs-char-d}) that the characteristic module of a strong symmetric Poisson structure is also preserved by the symmetric bracket. Surprisingly, we can prove that it is also involutive (for the Lie bracket on vector fields).

\begin{proposition}\label{prop: ssPs-involutivity}
 The characteristic module of a strong symmetric Poisson structure $(\vartheta,\nabla)$ is involutive.
\end{proposition}

\begin{proof}
 It follows from the torsion-freeness of $\nabla$ that, for $\alpha,\beta\in\Omega^1(M)$, we have
 \begin{align*}
 [\vartheta(\alpha),\vartheta(\beta)]&=\nabla_{\vartheta(\alpha)}\vartheta(\beta)-\nabla_{\vartheta(\beta)}\vartheta(\alpha)\\
 &=(\nabla_{\vartheta(\alpha)}\vartheta)(\beta)+\vartheta(\nabla_{\vartheta(\alpha)}\beta)-(\nabla_{\vartheta(\beta)}\vartheta)(\alpha)-\vartheta(\nabla_{\vartheta(\beta)}\alpha).
 \end{align*}
 By Corollary \ref{cor: ssPs-integrability}, this simplifies to
 \begin{equation*}
 [\vartheta(\alpha),\vartheta(\beta)]=\vartheta(\nabla_{\vartheta(\alpha)}\beta-\nabla_{\vartheta(\beta)}\alpha),
 \end{equation*}
 hence the characteristic module is involutive.
\end{proof}

\begin{remark}
The proof of Proposition \ref{prop: ssPs-involutivity} gives a hint on how to define an \mbox{anti-com}mutative bracket on $\Omega^1(M)$ when a symmetric bivector field and a connection on $M$ are given. It is natural to ask whether this bracket endows $T^*M$ with a Lie algebroid structure. We explore this direction in Appendix~\ref{app: Lie algebroids}. 
\end{remark}

\textit{Frobenius' theorem} establishes that a regular distribution integrates to a regular foliation if and only if its space of global sections is involutive. A generalization of the notion of a regular foliation of $M$ is that of a \textbf{partition} of $M$, that is, a disjoint decomposition of $M$ into immersed submanifolds. Given a smooth partition, every point $m\in M$ is thus contained in only one of these submanifolds, which we denote by $N_m$. A partition of $M$ is called \textbf{smooth} if, for every $u\in T_mN_m$, there is a vector field $X\in\mathfrak{X}(M)$ such that $X_m=u$ and $X_{m'}\in T_{m'}N_{m'}$ for all $m'\in M$. As we will recall next, the involutivity of a $\cCi(M)$-submodule of $\mathfrak{X}(M)$ is closely related to the existence of a partition of the manifold even in the singular case.

\begin{definition}
 We say that a symmetric Poisson structure is \textbf{involutive} if its characteristic module is involutive.
\end{definition}

By Proposition \ref{prop: ssPs-involutivity}, a strong symmetric Poisson structure is always involutive, so we have
\begin{equation*}
 \left\{\begin{array}{c}
\text{strong symmetric}\\
 \text{Poisson structures}
 \end{array}\right\}\subseteq\left\{\begin{array}{c}
\text{involutive symmetric}\\
 \text{Poisson structures}
 \end{array}\right\}\subseteq\left\{\begin{array}{c}
\text{symmetric}\\
 \text{Poisson structures}
 \end{array}\right\}.
\end{equation*}
Example \ref{ex: SO(3)2} shows that the second inclusion is strict. As we will show later that also the first inclusion is strict (Examples \ref{ex: involutive-(s)sPs} and \ref{ex: R5}), we shall thus work with this generality. 

\subsection{The characteristic partition} 
The origin of the theory of singular foliations goes back to the works of Štefan \cite{Stefan} and Sussmann \cite{Sussmann}. The modern approach to the theory is given in \cite{AndHGSF, LauWISF}. Our approach here primarily follows \cite[App. C]{CraLPG} and thus we understand singular foliations as \mbox{$\cCi(M)$-sub}modules of $\mathfrak{X}(M)$. After a brief review, we apply this theory to involutive symmetric Poisson structures.

Every $\cCi(M)$-submodule $\mathcal{F}$ of $\mathfrak{X}(M)$ determines a smooth distribution
\begin{equation*}
\Delta_{\mathcal{F}}:=\bigcup_{m\in M}\{ X_m\in T_mM\,|\,X\in\mathcal{F}\}. 
\end{equation*}

\begin{example}
 For every $\vartheta\in\mathfrak{X}^2_\emph{sym}(M)$, we have that $\im\vartheta=\Delta_{\mathcal{F}_\vartheta}$, see \eqref{eq: char-d-m}.
\end{example}
 
An \textbf{integral submanifold} of a $\cCi(M)$-submodule $\mathcal{F}$ of $\mathfrak{X}(M)$ is an immersed submanifold $N\subseteq M$ satisfying, for $m\in N$,
 \begin{equation*}
 (\Delta_{\mathcal{F}})_m=T_mN\leq T_mM.
 \end{equation*}
 A connected integral submanifold is called a \textbf{leaf} of $\mathcal{F}$ if it is not properly contained in any connected integral submanifold of $\mathcal{F}$.
 
\begin{example}\label{ex: R-part}
 The module $\mathcal{F}_\vartheta$ from Example \ref{ex: sing-d} gives the smooth partition of the real line into three (embedded) leaves $\R=\R^-\cup\{0\}\cup\R^+$.
\end{example}

The involutivity of a $\cCi(M)$-submodule of $\mathfrak{X}(M)$ is generally not enough for the leaves of the module to form a partition of $M$. Let us now recall what the extra needed hypotheses are.
% when it is not the case, counter examples using exp(-1/x).

\begin{definition}\label{def: loc-mod}
 Let $\mathcal{F}$ be a $\cCi(M)$-submodule of $\mathfrak{X}(M)$. A vector field $X\in\mathfrak{X}(M)$ is \textbf{locally in} $\mathcal{F}$ if, for every $m\in M$, there is $Y\in\mathcal{F}$ and a neighbourhood $U$ of $m$ such that $\rest{X}{U}=\rest{Y}{U}$. The submodule $\mathcal{F}$ is called \textbf{local} if it contains all vector fields that are locally in $\mathcal{F}$.
\end{definition}

It is always the case for the characteristic module of a symmetric bivector field.

\begin{lemma}\label{lem: loc}
 For every $\vartheta\in\mathfrak{X}^2_\emph{sym}(M)$, the characteristic module $\mathcal{F}_\vartheta$ is local. 
\end{lemma}

\begin{proof}
 For a proof, see Appendix \ref{app: some-proofs}.
\end{proof}

Let $\mathcal{F}$ be a local $\cCi(M)$-submodule of $\mathfrak{X}(M)$. For an open subset $U\subseteq M$, we introduce the $\cCi(U)$-submodule of $\mathfrak{X}(U)$:
\begin{equation*}
 \Gamma(U,\mathcal{F}):=\{ X\in\mathfrak{X}(U)\,|\,fX\in\mathcal{F}\text{ for all } f\in\cCi_c(U)\},
\end{equation*}
where the subscript `c' denotes compactly supported.

\begin{definition}
A local $\cCi(M)$-submodule $\mathcal{F}\subseteq \mathfrak{X}(M)$ is \textbf{locally finitely generated} if for every $m\in M$, there is a neighbourhood $U\subseteq M$ and a collection $\{ X_i\}_{i=1}^r\subseteq\Gamma(U,\mathcal{F})$, $r\in\N$, that generates $\Gamma_c(U,\mathcal{F}):=\Gamma(U,\mathcal{F})\cap\mathfrak{X}_c(U)$ as a $\cCi_c(U)$-module. If $\mathcal{F}$ is moreover involutive, we call it a \textbf{singular foliation}.
\end{definition}

We have that $\mathcal{F}_\vartheta$ is not only local, but it is also locally finitely generated.

\begin{lemma}\label{lem: loc-fin}
 For every $\vartheta\in\mathfrak{X}^2_\emph{sym}(M)$, the characteristic module $\mathcal{F}_\vartheta$ is locally finitely generated. 
\end{lemma}

\begin{proof}
 For a proof, see Appendix \ref{app: some-proofs}.
\end{proof}

Finally, we state the singular version of Frobenius' theorem.

\begin{theorem}[\cite{AndHGSF,CraLPG,LauWISF}]\label{thm: gen-Frob}
The leaves of a singular foliation form a smooth partition.
\end{theorem}

By Lemmas \ref{lem: loc}, \ref{lem: loc-fin} and Theorem \ref{thm: gen-Frob}, we have the following:

\begin{corollary}\label{cor: char-part}
 An involutive symmetric Poisson structure $(\vartheta,\nabla)$ on $M$ gives a smooth partition of $M$ into leaves of the characteristic module, that is, for every leaf $N\subseteq M$ of $\mathcal{F}_\vartheta$, we have $\rest{\im\vartheta}{N}=TN$. 
\end{corollary}

Note that, by Proposition \ref{prop: ssPs-involutivity}, Corollary \ref{cor: char-part} applies, in particular, to strong symmetric Poisson structures.

\begin{definition}\label{def: inv-sPs}
 Let $(\vartheta,\nabla)$ be an involutive symmetric Poisson structure on $M$, we call the associated partition of $M$ the \textbf{characteristic partition} of $\vartheta$.
\end{definition}

\subsection{Totally geodesic regular foliations}

For a connection $\nabla$, a submanifold $N\subseteq M$ is called \textbf{totally} \textbf{geodesic} if every geodesic that tangentially intersects $N$ is locally contained in $N$, that is, given a geodesic $\gamma:I\rightarrow M$, with $I$ an open interval, such that $\dot{\gamma}(t_0)\in T_{\gamma(t_0)}N$ for some $t_0\in I$, there is an open subinterval $I'$ of $I$ containing $t_0$ such that $\dot{\gamma}(t)\in T_{\gamma(t)}N$ for every $t\in I'$. A partition is called \textbf{totally geodesic} if all of its leaves are totally geodesic submanifolds.

\begin{lemma}\label{lem: reg-tot-geo}
 The regular foliation given by a regular involutive geodesically invariant distribution $\Delta$ is totally geodesic.
\end{lemma}

\begin{proof}
 Consider a leaf $N_0\subseteq M$ of the regular foliation and a geodesic $\gamma:I\rightarrow M$ such that $\dot{\gamma}(t_0)\in T_{\gamma(t_0)}N_0=\Delta_{\gamma(t_0)}$ for some $t_0\in I$. By geodesic invariance, we have that $\dot{\gamma}(t)\in\Delta_{\gamma(t)}$ for every $t\in I$. Choose a normal local coordinate chart $(U,\{ x^i\})$, $U\subseteq M$, for $\Delta$ such that $\gamma(t_0)\in U$. Namely, for some subinterval $I'\subset I$ and $\gamma^i:=x^i\circ\gamma:I'\rightarrow \R$, we have that $\gamma^i(t_0)=0$ for every $i\in\{ r+1\varlist n\}$, where $r:=\rk\Delta=\dim N_0$. If there is $t_1\in I'$ such that $\gamma(t_1)$ lies in a leaf $N\neq N_0$, we find $i_0\in\{ r+1\varlist n\}$ such that $\gamma^{i_0}(t_1)\neq 0$. Therefore, by \textit{Lagrange's mean value theorem}, we find $t_2\in (t_0,t_1)$, or $(t_1,t_0)$, such that $\dot{\gamma}^{i_0}(t_2)\neq 0$, so $\dot{\gamma}(t_2)\notin\spann\{ \rest{\partial_{x^1}}{\gamma(t_2)}\varlist\rest{\partial_{x^r}}{\gamma(t_2)}\}=\Delta_{\gamma(t_2)}$, which contradicts geodesic invariance. Hence, $\gamma(t)\in N_0$ for every $t\in I'$, that is, $N_0$ is totally geodesic.
\end{proof}

The following result, which, in particular applies to strong symmetric Poisson structures, is now an easy consequence of Theorem \ref{thm: LewisT} and Lemma \ref{lem: reg-tot-geo}.

\begin{proposition}\label{prop: reg-fol-ssPs}
 The characteristic regular foliation of a regular involutive symmetric Poisson structure is totally geodesic.
\end{proposition}

It is also true that even the characteristic partition of a singular involutive symmetric structure is totally geodesic (Theorem \ref{thm:nabla-geodesic-part-of-inv-sym-Poisson}), but its proof is much more subtle and requires the theory we develop in Sections \ref{sec: leaf-met} and \ref{sec: leaf-con}.

\begin{example}\label{ex: 3-fol} The punctured plane $M:=\R^2\setminus\lbrace 0\rbrace$ admits nowhere vanishing vector fields:
\begin{align*}
 R&:=x\,\partial_x+y\,\partial_y, & S&:=-y\,\partial_x+x\,\partial_y, & H&:=x\,\partial_x-y\,\partial_y.
\end{align*}
Each of them defines the symmetric bivector field
\begin{align*}
 \vartheta_R&:=R\otimes R, & \vartheta_S&:=S\otimes S, & \vartheta_H&:=H\otimes H,
\end{align*}
and the regular foliation of $M$, see Figure \ref{fig: partitions}. A torsion-free connection $\nabla$ makes $(\vartheta_X,\nabla)$, where $X\in\lbrace R,S,H\rbrace$, a strong symmetric Poisson structure if and only if $\nabla_XX=0$. Respectively, we get the following equations for the unknown $\nabla$:
\begin{align*}
 0&=\nabla_RR=x\,\partial_x+y\,\partial_y+x^2\nabla_{\partial_x}\partial_x+2xy\nabla_{\partial_x}\partial_y+y^2\nabla_{\partial_y}\partial_y,\\
 0&=\nabla_SS=-x\,\partial_x-y\,\partial_y+y^2\nabla_{\partial_x}\partial_x-2xy\nabla_{\partial_x}\partial_y+x^2\nabla_{\partial_y}\partial_y,\\
 0&=\nabla_HH=x\,\partial_x+y\,\partial_y+x^2\nabla_{\partial_x}\partial_x-2xy\nabla_{\partial_x}\partial_y+y^2\nabla_{\partial_y}\partial_y.
\end{align*}
Each of these equations is a system of two algebraic equations for six unknown functions. Their solutions are not unique, but a solution for each one is given by
 \begin{align}\label{eq: 3-reg-tot-geo}
\nabla_{\partial_x}\partial_x&=\nabla_{\partial_y}\partial_y:=\pm\frac{1}{x^2+y^2}R, & \nabla_{\partial_x}\partial_y&:=0,
 \end{align} 
where we choose the $+$ sign for $X=S$ and the $-$ sign for $X=R, H$. All three regular foliations are thus coming from a (non-parallel) strong symmetric Poisson structure, hence, by Proposition \ref{prop: reg-fol-ssPs}, they are totally geodesic for $\nabla$ as in \eqref{eq: 3-reg-tot-geo}.
\end{example}

\begin{figure}[ht]
\begin{center}
\includegraphics[height=4cm,keepaspectratio]{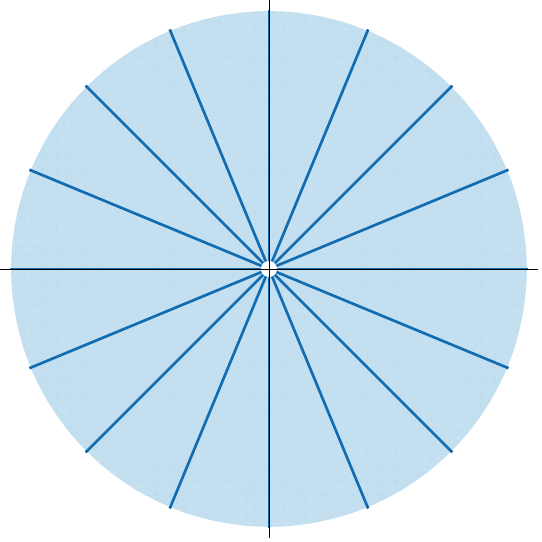}
\hspace{15pt}
\includegraphics[height=4cm,keepaspectratio]{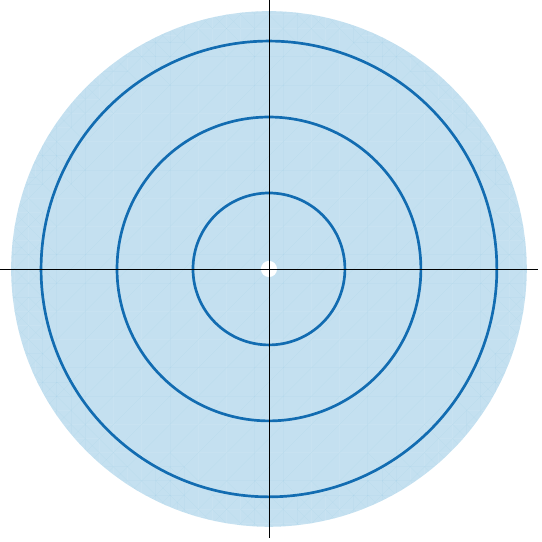}
\hspace{15pt}
\includegraphics[height=4cm,keepaspectratio]{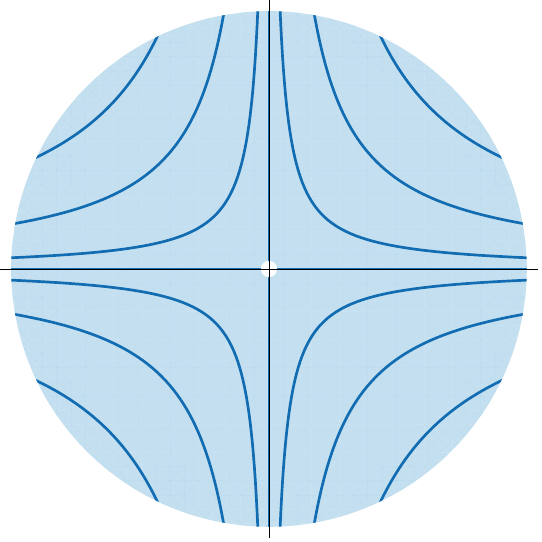}
\end{center}
\caption{Regular foliations of the punctured plane.}\label{fig: partitions}
\end{figure}

\subsection{The leaf metric}\label{sec: leaf-met}
We say that a symmetric bivector field $\vartheta\in\symbi$ is \textbf{tangent to a submanifold} $N\subseteq M$ if, for all $m\in N$,
\begin{equation*}
 \vartheta_m\in\Sym^2T_mN\leq\Sym^2T_mM.
\end{equation*}
When $\vartheta\in\mathfrak{X}^2_\text{sym}(M)$ is tangent to a submanifold $N\subseteq M$, we obtain the symmetric bivector field $\vartheta_N\in\mathfrak{X}^2_\text{sym}(N)$ by restriction of $\vartheta$.

\begin{proposition}\label{prop: symbi-rest}
 Let $(\vartheta,\nabla)$ be an involutive symmetric Poisson structure on $M$. The bivector field $\vartheta$ is tangent to every leaf $N$ of the characteristic partition, and moreover, the restricted symmetric bivector field $\vartheta_N$ is non-degenerate. 
\end{proposition}

\begin{proof}
We recall that $(\im\vartheta)_m=T_mN$ for any leaf $N\subseteq M$ and every $m\in N$. Therefore, given a neighbourhood $U'\subseteq M$ of $m\in M$, we have, for every $\zeta\in T^*_mM$ and $f\in\cCi(U')$ such that $\rest{f}{N\cap U'}=0$, that
\begin{equation*}
 0=\vartheta(\zeta)f=\vartheta(\zeta,\rest{\dif f}{m}).
\end{equation*}
In a normal local coordinate chart $(U,\{ x^i\})$, $U\subseteq M$, for the submanifold $N$,
\begin{equation*}
 \vartheta(\rest{\dif x^i}{m},\rest{\dif x^j}{m})=0
 \end{equation*}
for $i\in\{ 1\varlist n\}$ and $j\in\{ r+1\varlist n\}$, where $r:=\dim N$. It follows that
\begin{equation*}
 \vartheta_m=\frac{1}{2}\sum_{i,j=1}^r\vartheta(\rest{\dif x^i}{m},\rest{\dif x^j}{m})\,\rest{\partial_{x^i}}{m}\odot \rest{\partial_{x^j}}{m}.
\end{equation*}
As $\rest{\partial_{x^i}}{m}\in T_mN$ for every $i\in\{ 1\varlist r\}$, we have that $\vartheta_m\in\Sym^2T_mN$.

As $\vartheta_m(T^*_mN)= T_mN$ is clearly satisfied, the bivector $(\vartheta_N)_m$ is non-degenerate.
\end{proof}

\sloppy It follows directly from Proposition \ref{prop: symbi-rest} that there is an induced \mbox{(pseudo-)Rie}mannian metric on every leaf of the characteristic partition of an involutive symmetric Poisson structure. 

\begin{definition}
 Let $(\vartheta,\nabla)$ be an involutive symmetric Poisson structure and $N$ be a leaf of the characteristic partition. We call the \mbox{(pseudo-)Rie}mannian metric $g_N:=\vartheta^{-1}_N$ the \textbf{leaf metric} of $N$. 
\end{definition}

 The more general concept of the characteristic metric $g_\vartheta$ (Definition \ref{def: char-metric}), which makes sense even for non-involutive symmetric Poisson structures, is obviously related to that of the leaf metric. For every leaf $N\subseteq M$ and $m\in N$, we have
 \begin{equation*}
 g_{\vartheta_m}=(g_N)_m.
 \end{equation*}

\subsection{The leaf connection}\label{sec: leaf-con}
 In this section we show that, in addition to the \mbox{(pseudo-)Riem}annian metric, we also have an induced connection on every leaf of the characteristic partition of an involutive symmetric Poisson structure.

For a torsion-free connection $\nabla$ on $M$ and $X,Y\in\mathfrak{X}(M)$, we have
\begin{equation*}
 \nabla_XY=\frac{1}{2}(\pg{X,Y}_s+[X,Y]).
\end{equation*} 
Given an involutive symmetric Poisson structure $(\vartheta,\nabla)$, it follows from Propositions \ref{prop: sPs-char-d} and \ref{prop: ssPs-involutivity} that, for $X,Y\in\mathcal{F}_\vartheta$, we have
\begin{equation*}
 \nabla_XY\in\mathcal{F}_\vartheta.
\end{equation*}
Namely, for an arbitrary leaf $N\subseteq M$ and $u\in T_mN=(\im\vartheta)_m$, we find 
\begin{equation}\label{eq: leaf-con-1}
 \nabla_u Y\in T_mN.
\end{equation}

\begin{lemma}\label{lem: leaf-con}
 Let $(\vartheta,\nabla)$ be an involutive symmetric Poisson structure on $M$ and $N$ be a leaf of the characteristic partition. Given $u\in T_mN$ and $\beta\in\Omega^1(N)$, consider $\hat{\beta}\in\Omega^1(M)$ an arbitrary extension of $\beta$ around $m$ (that is, $\beta$ and $\hat{\beta}$ coincide on a neighbourhood of $m$ in $N$). Then, the vector $\nabla_u\vartheta(\hat{\beta})$ is tangent to $N$ and independent of $\hat{\beta}$.
\end{lemma}

\begin{proof}
 It follows from \eqref{eq: leaf-con-1} that the vector $\nabla_u\vartheta(\hat{\beta})\in T_mM$ is tangent to $N$. To show independence from the choice of $\hat{\beta}$, it is enough, by linearity, to show that if $\hat{\beta}\in\Omega^1(M)$ satisfies $\hat{\beta}\hspace{-3pt}\rest{}{U}=0$ for some neighbourhood $U\subseteq N$ of $m$, then $\nabla_u\vartheta(\hat{\beta})=0$. Since $u\in T_mN$, there is a neighbourhood $U'\subseteq U$ of $m$ and a curve $\sigma:I\rightarrow U'$ such that $\dot{\sigma}(0)=u$. As $\hat{\beta}_{\sigma(t)}=0$ for every $t\in I$, we have that
 \begin{equation*}
 \nabla_u\vartheta(\hat{\beta})=\lim_{t\to 0}\frac{1}{t}(P^\sigma_{t,0}\vartheta(\hat{\beta}_{\sigma(t)})-\vartheta(\hat{\beta}_{m}))=0,
 \end{equation*}
 where $P^\sigma$ denotes the parallel transport along $\sigma$.
\end{proof}

By Proposition \ref{prop: symbi-rest}, we have that $\vartheta_N$ is non-degenerate, hence there is a unique $\beta\in\Omega^1(N)$ for every $Y\in\mathfrak{X}(N)$ satisfying $Y=\vartheta_N(\beta)$. Thanks to Lemma \ref{lem: leaf-con}, we can introduce the map $\nabla^N:\mathfrak{X}(N)\rightarrow\Gamma (\en TN)$, for every $u\in T_mN$, by
\begin{equation}\label{eq: leaf-con-2}
 (\nabla^N\vartheta_N(\beta))(u):=\nabla_u\vartheta(\hat{\beta}),
\end{equation}
where $\hat{\beta}\in\Omega^1(M)$ is an arbitrary extension of $\beta$ around $m$.

\begin{lemma}
 Let $(\vartheta,\nabla)$ be an involutive symmetric Poisson structure on $M$ and $N$ be a leaf of the characteristic partition. The map $\nabla^N$, given by \eqref{eq: leaf-con-2}, defines a torsion-free connection on $N$.
\end{lemma}

\begin{proof}
For every $u\in T_mN$, $f\in\cCi(N)$, a neighborhood $U\subseteq N$ of $m$ and $\hat{f}\in\cCi(M)$ such that $\hat{f}\hspace{-3pt}\rest{}{U}=\rest{f}{U}$, we find a curve $\sigma:I\rightarrow U$ such that $\dot{\sigma}(0)=u$, hence,
 \begin{equation}\label{eq: leaf-con-3}
 u\hat{f}=\lim_{t\to 0}\frac{1}{t}(\hat{f}(\sigma(t))-\hat{f}(m))=\lim_{t\to 0}\frac{1}{t}(f(\sigma(t))-f(m))=uf.
 \end{equation}
Clearly, we can choose $U$, $\hat{f}$ and $\hat{\beta}\in\Omega^1(M)$ such that $\hat{f}\hspace{-3pt}\rest{}{U}=\rest{f}{U}$ and at the same time $\hat{\beta}\hspace{-3pt}\rest{}{U}=\rest{\beta}{U}$. By Lemma \ref{lem: leaf-con} and \eqref{eq: leaf-con-3}, we get
 \begin{equation*}
 \nabla^N_u(f\vartheta_N(\beta))=\nabla_u(\hat{f}
\vartheta(\hat{\beta}))=u\hat{f}+\hat{f}(m)\nabla_u\vartheta(\hat{\beta})=uf+f(m)\nabla_u^N\vartheta_N(\beta).
\end{equation*}
As $\nabla^N$ is moreover $\R$-linear, it is a connection on $N$.

We choose a normal local coordinate chart $(U,\{ x^i\})$, $U\subseteq M$, for the submanifold $N$ centred at $m$, and $\hat{\alpha},\hat{\beta}\in\Omega^1(M)$ such that $\hat{\alpha}\hspace{-3pt}\rest{}{U}=\rest{\alpha}{U}$ and $\hat{\beta}\hspace{-3pt}\rest{}{U}=\rest{\beta}{U}$. Then,
\begin{align*}
[\vartheta_N(\alpha),&\vartheta_N(\beta)]_m\\
&=\sum_{i=1}^{r}(\vartheta_N(\alpha_m) \dif x^i(\vartheta_N(\beta))-\vartheta_N(\beta_m)\dif x^i(\vartheta_N(\alpha)))\rest{\partial_{x^i}}{m},
\end{align*}
where $r:=\dim N$. By \eqref{eq: leaf-con-3}, the torsion-freeness of $\nabla$ and Lemma \ref{lem: leaf-con}, we get
\begin{align*}
 [\vartheta_N(\alpha),\vartheta_N(\beta)]_m&=\sum_{i=1}^{n}(\vartheta(\hat{\alpha}_m) \dif x^i(\vartheta(\hat{\beta}))-\vartheta(\hat{\beta}_m) \dif x^i(\vartheta(\hat{\alpha})))\rest{\partial_{x^i}}{m}\\
 &=[\vartheta(\hat{\alpha}),\vartheta(\hat{\beta})]_m=\nabla_{\vartheta(\alpha_m)}\vartheta(\hat{\beta})-\nabla_{\vartheta(\beta_m)}\vartheta(\hat{\alpha})\\
 &=\nabla^N_{\vartheta_N(\alpha_m)}\vartheta_N(\beta)-\nabla^N_{\vartheta_N(\beta_m)}\vartheta_N(\alpha).
\end{align*}
\end{proof}

So, every leaf of the characteristic partition of an involutive symmetric Poisson structure is not only naturally equipped with a \mbox{(pseudo-)Rie}mannian structure, it also carries a torsion-free connection.

\begin{definition}
 Let $(\vartheta,\nabla)$ be an involutive symmetric Poisson structure and $N$ be a leaf of the characteristic partition. The torsion-free connection $\nabla^N$ on $N$, defined by \eqref{eq: leaf-con-2}, is called the \textbf{leaf connection} of $N$.
\end{definition}

\subsection{Totally geodesic partitions}
Finally, we are in the position to generalize the result of Proposition \ref{prop: reg-fol-ssPs} to all involutive symmetric Poisson structures. The following lemma establishes a relation between the $\nabla$-parallel transport $P^{\sigma}$ and the $\nabla^N$-parallel transport $P^{\sigma,N}$ along a curve $\sigma:I\rightarrow N\subseteq M$ and will come in handy.

\begin{lemma}\label{lem: parallel}
 Let $(\vartheta,\nabla)$ be an involutive symmetric Poisson structure on $M$ and $N$ be a leaf of the characteristic partition. For a curve $\sigma:I\rightarrow N$, $t_0, t_1\in I$ and $u\in T_{\sigma(t_0)}N$, we have that
 \begin{equation*}
 P^\sigma_{t_0,t_1}u=P^{\sigma,N}_{t_0, t_1}u.
 \end{equation*}
\end{lemma}

\begin{proof}
 We find the unique vector field $X$ along the curve $\sigma$ such that $X_{t_0}=u$~and $\nabla^N_{\dot{\sigma}}X=0$. By the definition of $\nabla^N$-parallel transport, we have $
 P^{\sigma,N}_{t_0, t_1}u=X_{t_1}$. It follows from the definition of the leaf connection that $\nabla_{\dot{\sigma}}X=\nabla^N_{\dot{\sigma}}X=0$. Therefore, by the definition of $\nabla$-parallel transport, we get $ P^{\sigma}_{t_0, t_1}u=X_{t_1}=P^{\sigma,N}_{t_0, t_1}u.$
\end{proof}

\begin{theorem}\label{thm:nabla-geodesic-part-of-inv-sym-Poisson}
 The characteristic partition of an involutive symmetric Poisson structure $(\vartheta,\nabla)$ is totally geodesic. Moreover, on any leaf $N$, the pair $(g^{-1}_N,\nabla^N)$ is a non-degenerate symmetric Poisson structure. In addition, if $(\vartheta,\nabla)$ is strong, $\nabla^N$ is the Levi-Civita connection of $g_N$, that is, $(g^{-1}_N,\nabla^N)$ is also strong.
\end{theorem}

\begin{proof}
Consider a leaf $N\subseteq M$ and a $\nabla$-geodesic $\gamma:I\rightarrow M$ such that $\dot{\gamma}(t_0)\in T_{\gamma(t_0)}N=(\im\vartheta)_{\gamma(t_0)}$ for some $t_0\in I$. We can now consider a $\nabla^N$-geodesic $\sigma:I'\rightarrow N$, $t_0\in I'$, given by the initial condition $\dot{\sigma}(t_0)=\dot{\gamma}(t_0)$. Since $\dot{\sigma}$ is $\nabla^N$-parallel, it is, by Lemma \ref{lem: parallel}, $\nabla$-parallel, that is, $\sigma$ is also $\nabla$-geodesic. It follows from the uniqueness of geodesisc that there is an open interval $I_0\subseteq I\cap I'$ containing $t_0$ such that $\rest{\gamma}{I_0}=\rest{\sigma}{I_0}$, hence $N$ is a totally $\nabla$-geodesic submanifold.

By Proposition \ref{prop: symbi-rest}, we have that $\vartheta_N\in\mathfrak{X}^2_\text{sym}(N)$ is non-degenerate. For every $\alpha,\beta\in\Omega^1(N)$ and $m\in N$, we find
 \begin{align*}
 (\nabla^N_{\vartheta_N(\alpha_m)}\vartheta_N)(\beta,\beta)&=\vartheta_N(\alpha_m)\vartheta_N(\beta,\beta)-2\vartheta_N(\nabla^N_{\vartheta_N(\alpha_m)}\beta,\beta_m)\\
 &=-\vartheta_N(\alpha_m)\vartheta_N(\beta,\beta)+2\beta(\nabla^N_{\vartheta_N(\alpha_m)}\vartheta_N(\beta)).
 \end{align*}
 For an arbitrary extension $\hat{\beta}\in\Omega^1(M)$ of $\beta$ around $m$, the function $\vartheta(\hat{\beta},\hat{\beta})\in\cCi(M)$ is an extension of $\vartheta_N(\beta,\beta)\in\cCi(N)$ around $m$, hence
 \begin{equation}\label{eq: leaf-metric-sPs}
 \begin{split}
 (\nabla^N_{\vartheta_N(\alpha_m)}\vartheta_N)(\beta,\beta)&=-\vartheta(\alpha_m)\vartheta(\hat{\beta},\hat{\beta})+2\hat{\beta}(\nabla_{\vartheta(\alpha_m)}\vartheta(\hat{\beta}))\\
 &=(\nabla_{\vartheta(\alpha_m)}\vartheta)(\hat{\beta},\hat{\beta}).
 \end{split}
 \end{equation}
 By Proposition \ref{prop: sym-Poisson-bracket}, we get 
 \begin{equation*}
 0=(\nabla^N_{\vartheta_N(\alpha)}\vartheta_N)(\alpha,\alpha)=\frac{1}{6}[\vartheta_N,\vartheta_N]_s(\alpha,\alpha,\alpha)
 \end{equation*}
 and thus, by polarization, $(\vartheta_N,\nabla^N)$ is a symmetric Poisson structure on $N$. Analogously, if $(\vartheta,\nabla)$ is strong, it follows from polarization and \eqref{eq: leaf-metric-sPs} that
 \begin{equation*}
 \nabla^N_{\vartheta_N(\alpha)}\vartheta_N=0,
 \end{equation*}
 that is, by Corollary \ref{cor: ssPs-integrability}, the pair $(\vartheta_N,\nabla^N)$ is a strong symmetric Poisson structure on $N$. Therefore, by Proposition \ref{prop: ssPs-nondeg}, we have that $\nabla^N$ is indeed the Levi-Civita connection of $g_N=\vartheta_N^{-1}$.
\end{proof}

\section{Patterson-Walker dynamics and applications}\label{sec: PW} 
Every torsion-free connection on $M$ determines a split-signature metric on $T^*M$, the \textit{Patterson-Walker metric}, first introduced in \cite{PatRE}. We briefly recall its modern definition. First, a torsion-free connection $\nabla$ on $M$ induces a connection on the cotangent bundle $\pr:T^*M\rightarrow M$, which gives the decomposition 
\begin{equation*}
T(T^*M)=\mathcal{H}_\nabla\oplus \mathcal{V}\cong \pr^*TM\oplus \pr^*T^*M = \pr^*(TM\oplus T^*M).
\end{equation*}
Moreover, we get a $\cCi(M)$-module morphism
\begin{equation*}
\phi_\nabla:\Gamma(TM\oplus T^*M)\rightarrow \mathfrak{X}(T^*M), 
\end{equation*}
whose image locally generates the space $\mathfrak{X}(T^*M)$. Therefore, we can bring the canonical symmetric pairing, defined for $X+\al,Y+\be \in \Gamma(TM\oplus T^*M)$ as 
\begin{equation*}
 \langle X+\alpha, Y+\beta\rangle_+:=\alpha(Y)+\beta(X),
\end{equation*}
from the bundle $TM\oplus T^*M$ to the tangent bundle of $T^*M$.

\begin{definition} \label{def:PW-metric} Let $\nabla$ be a torsion-free connection on $M$. \textbf{The Patterson-Walker metric} $g_\nabla$ is the split-signature metric on $T^*M$ determined by
 \begin{equation*}
 g_\nabla(\phi_\nabla a,\phi_\nabla b):=\pr^*\langle a,b\rangle_+.
\end{equation*}
\end{definition}

See \cite{SymCartan} for a discussion of basic properties of the Patteson-Walker metric and its relation to symmetric Cartan calculus.

\begin{remark}
 In natural local coordinates $(T^*U,\{ x^i\}\cup\{ p_j\})$ on $T^*M$, the Patterson-Walker metric reads:
\begin{equation}\label{eq: PW-coordinates}
 \rest{g_\nabla}{T^*U}=\dif p_j\odot\dif x^j-p_k(\pr^*\Gamma^k_{ij})\,\dif x^i\odot\dif x^j,
\end{equation}
where $\{\Gamma^k_{ij}\}\subseteq\cCi(U)$ are the Christoffel symbols of $\nabla$ in the chart $(U,\{ x^i\})$. 
Formula \eqref{eq: PW-coordinates} recovers the original definition in \cite{PatRE} and also
shows the similarity between $g_\nabla$ and the canonical symplectic form $\omega_\text{can}\in\Omega^2(T^*M)$,
\begin{equation*}
 \rest{\omega_\text{can}}{T^*U}=\dif p_i\wedge\dif x^i.
\end{equation*}
For more details about this analogy see \cite{SymCartan}.
\end{remark}

\subsection{The Patterson-Walker bracket} 
By Proposition \ref{prop: ssPs-nondeg}, every torsion-free connection on $M$ defines, through the Patterson-Walker metric, the non-degenerate strong symmetric Poisson structure $(g^{-1}_\nabla,\bar{\nabla})$ on $T^*M$, where $\bar{\nabla}$ denotes the \mbox{Levi-Civ}ita connection of $g_\nabla$. In a natural coordinate chart, the corresponding symmetric Poisson bracket $\lbrace\,,\rbrace_\nabla$ takes the form, for $F$, $G\in\cCi(T^*M)$:
\begin{equation}\label{eq: PW-bracket}
\rest{\po{F,G}_\nabla}{T^*U}=\frac{\partial F}{\partial x^i}\frac{\partial G}{\partial p_i}+\frac{\partial F}{\partial p_i}\frac{\partial G}{\partial x^i}+2p_k\,(\pr^*\Gamma^k_{ij})\,\frac{\partial F}{\partial p_i}\frac{\partial G}{\partial p_j}.
\end{equation}

We show that there is a relation between the Patterson-Walker metric and the symmetric Schouten bracket. First, we introduce the notion of \textit{polynomial in momenta}, which is analogous to that of \textit{polynomial in velocities}, \cite{SymCartan}.

\begin{definition}\label{polvel}
 A \textit{homogeneous function} $F\in\cCi(T^*M)$ \textit{of degree $r$}, that is, for every $\lambda\in\R$ and $\zeta\in T^*M$,
 \begin{equation*}
 F(\lambda\,\zeta)=\lambda^r F(\zeta),
 \end{equation*}
 is called a \textbf{degree}-$r$ \textbf{polynomial in momenta} on $M$. 
\end{definition}

When we say that $F\in\cCi(T^*M)$ is a \textbf{polynomial in momenta}, we mean that it is an element of the direct sum of the spaces of degree-$r$ polynomials in momenta for $r\in \N$. This direct sum is a unital subalgebra of $\cCi(T^*M)$ that is graded and we denote it by
 \begin{equation*}
 \mathcal{P}ol(T^*M).
 \end{equation*}
We explain the notation by the next lemma.

 \begin{lemma}\label{lem: polynomials}
 Let $F\in\cCi(T^*M)$ be a degree-$r$ polynomial in momenta on $M$. If $r=0$, we have that $F=\pr^*f$ for a unique $f\in\cCi(M)$. Whereas, if $r>0$ and a natural coordinate chart $(T^*U,\{ x^i\}\cup\{ p_j\})$ is given, we have that $\rest{F}{T^*U}=(\pr^*F^{j_1\ldots j_r})\,p_{j_1}\ldots p_{j_r}$ for a unique set of functions $\{ F^{j_1\ldots j_r}\}\subseteq \cCi(U)$. 
 \end{lemma}

 \begin{proof}
 The proof is analogous to \cite[Lem. 2.12]{SymCartan}.
 \end{proof}

We introduce the \textbf{vertical lift of a symmetric multivector field}, it is the map $(\,\,)^v:\symall\rightarrow \mathcal{P}ol(T^*M)$ given, for $\zeta\in T^*M$, by
\begin{align*}
\mathcal{X}^v(\zeta)&:=\frac{1}{r!}\,\mathcal{X}(\zeta\varlist \zeta) & &\text{if $r>0$},\\
f^v&:=\pr^*f & &\text{if $r=0$},
\end{align*}
which is a natural extension of the \textit{vertical lift of a vector field}, see \cite[App. B]{SymCartan}. It follows from Lemma \ref{lem: polynomials} and polarization that the map $(\,\,)^v$ is an isomorphism of unital graded algebras
\begin{equation}\label{eq: symall-polynomials}
 (\symall,\odot)\overset{\sim}{\longleftrightarrow} \mathcal{P}ol(T^*M)
\end{equation}
where the pointwise product of functions is considered on $\mathcal{P}ol(T^*M)$. The isomorphism \eqref{eq: symall-polynomials} recovers a well-known interpretation of symmetric multivector fields. We see that we can upgrade this isomorphism by considering the symmetric Schouten bracket and the symmetric Poisson bracket $\{\, ,\, \}_\nabla$.

\begin{proposition}\label{prop: sym-Schouten-PW}
Let $\nabla$ be a torsion-free connection on $M$. The vertical lift of a symmetric multivector field is an algebra isomorphism between the commutative algebras $(\mathfrak{X}^\bullet_\emph{sym}(M),[\,\,,\,]_s)$ and $(\mathcal{P}ol(T^*M),\lbrace\,,\rbrace_\nabla)$.
\end{proposition}

\begin{proof}
 The relation \eqref{eq: symall-polynomials} gives that the vertical lift is a vector space isomorphism so it remains to show that it also intertwines the brackets. As, again by \eqref{eq: symall-polynomials}, the vertical lift is compatible with $\odot$ and the pointwise product on $\cCi(T^*M)$, and moreover, both $[\,\,,\,]_s$ and $\lbrace\,,\rbrace_\nabla$ are commutative and posses the derivation properties (see Definition \ref{def: s-Schouten} and Definition \ref{def: ssPs}), it is enough to check that
 \begin{align}\label{eq:three-eq-prop-alg-morphism}
 [f,\cg]_s^v&=\po{f^v,\cg^v}_\nabla, & [f,X]_s^v&=\po{f^v,X^v}_\nabla, & [X,Y]_s^v&=\po{X^v,Y^v}_\nabla.
 \end{align}
Rewriting the left-hand sides by using the definition of $[\,\,,\,]_s$ yields
\begin{align*}
 0&=\po{f^v,\cg^v}_\nabla, & (Xf)^v&=\po{f^v,X^v}_\nabla, & (\nabla_XY+\nabla_YX)^v&=\po{X^v,Y^v}_\nabla.
\end{align*}
For the right-hand sides, we use local coordinates as in \eqref{eq: PW-bracket} and find
\begin{align*}
 \rest{\po{f^v,\cg^v}_\nabla}{T^*U}&=0,\\
 \rest{\po{f^v,X^v}_\nabla}{T^*U}&=\pr^*\Big(\frac{\partial f}{\partial x^i}X^i\Big)=\pr^*(Xf)=(Xf)^v,\\
 \rest{\po{X^v,Y^v}_\nabla}{T^*U}&=\pr^*\Big(\frac{\partial X^k}{\partial x^i}Y^i+X^i\frac{\partial Y^k}{\partial x^i}+2\Gamma^k_{ij} X^iY^j\Big)p_k=(\nabla_XY+\nabla_YX)^v,
\end{align*}
where $\rest{X}{U}=X^i\partial_{x^i}$ and $\rest{Y}{U}=Y^i\partial_{x^i}$, so \eqref{eq:three-eq-prop-alg-morphism} is satisfied and the result follows. 
\end{proof}
%%% By using the lifts from sym-Cartan, we can avoid going to local coordinates in this proof. But it's ok here with coordinates, as they're simple.

\subsection{Analogy with the canonical symplectic form}\label{sec: analog-can-2-form}

 Similarly as the symmetric Schouten bracket is related to the Patterson-Walker metric/bracket, the \textit{anti-commutative Schouten bracket} on $\symall$ is related to the \textit{canonical symplectic form/bracket} on $T^*M$.
 
 The \textbf{anti-commutative Schouten bracket}, introduced in \cite{SchSSB}, is the unique $\R$-bilinear map $[\,\,,\,]:\symall\times\symall\rightarrow\symall$ such that
 \begin{itemize}
 \item $[X,f]=Xf$ and $[X,Y]=X\circ Y-Y\circ X$ for $X,Y\in\mathfrak{X}(M)$,
 \item $[\mathcal{X},\,\,]$ is a degree-$(r-1)$ derivation of $(\symall,\odot)$ for $\mathcal{X}\in\symr$,
 \item $[\mathcal{X},\mathcal{Y}]=-[\mathcal{Y},\mathcal{X}]$.
 \end{itemize}
Comparing it with the symmetric Schouten bracket (Definition \ref{def: s-Schouten}), we see that the only differences are that the symmetric bracket is replaced with the Lie bracket on vector fields and the commutativity of $[\,\,,\,]_s$ with the anti-commutativity of $[\,\,,\,]$.

On the other hand, the canonical Poisson bracket on $\cCi(T^*M)$ is defined in terms of the canonical symplectic form by
\begin{equation*}
 \po{F,G}_\text{can}=-\omega_\text{can}(\Ham F,\Ham G).
\end{equation*}
In natural coordinates, it reads as
\begin{equation*}
 \rest{\po{F,G}_\text{can}}{T^*U}=\frac{\partial F}{\partial x^i}\frac{\partial G}{\partial p_i}-\frac{\partial G}{\partial x^i}\frac{\partial F}{\partial p_i}.
\end{equation*}
By a straightforward calculation, we recover the fact that the vertical lift of a symmetric multivector field is a Lie algebra isomorphism between $(\symall,[\,\,,\,])$ and $(\mathcal{P}ol(T^*M), -\lbrace\,,\rbrace_\text{can})$ (see, e.g., \cite{Cartier:1994}).

\subsection{Patterson-Walker dynamics}
The analogy between the Patterson-Walker bracket and the canonical Poisson bracket raises a natural question:\newline \textit{Hamiltonian dynamics is driven by $\{\,,\}_\mathrm{can}$. What dynamics is driven by $\{\,,\}_{\nabla}$?}

By the choice of a torsion-free connection $\nabla$ on $M$ and a function $H\in\cCi(T^*M)$, we obtain the vector field on $T^*M$:
\begin{equation*}
 \grad_\nabla H:=g_\nabla^{-1}(\dif H)=\po{H,\,\,}_\nabla.
\end{equation*}
The arising dynamics, which we call \textbf{Patterson-Walker dynamics}, is the study of the integral curves of such vector fields. A crucial difference from Hamiltonian dynamics is, due to
\begin{equation*}
 (\grad_\nabla H)H=g_\nabla(\grad_\nabla H,\grad_\nabla H),
\end{equation*}
the fact that the `Hamiltonian function' $H$ is not always constant along the integral curves of $\grad_\nabla H$. Actually, the function $H$ is conserved if and only if $\grad_\nabla H$ is isotropic with respect to the Patterson-Walker metric $g_\nabla$.

In natural coordinates, from \eqref{eq: PW-bracket}, we have
\begin{equation}\label{eq: grad-PW}
 \rest{\grad_\nabla H}{T^*U}=\frac{\partial H}{\partial p_i}\partial_{x^i}+\left(\frac{\partial H}{\partial x^j}+2p_k(\pr^*\Gamma^k_{ij})\frac{\partial H}{\partial p_i}\right)\partial_{p_j}.
\end{equation}
Therefore, Patterson-Walker dynamics is locally governed by the system of ODEs for the unknown curve $a: I\rightarrow T^*M$:
\begin{align}\label{eq: EoM}
\dot{\gamma}_a^i&=\frac{\partial H}{\partial p_i}\circ a, & \dot{a}_j&=\frac{\partial H}{\partial x^j}\circ a+2a_k\,(\Gamma^k_{ij}\circ\gamma_a)\,\frac{\partial H}{\partial p_i}\circ a,
\end{align}
where we denote $\gamma_a:=\pr\circ a$, $a_j:=p_j\circ a$ and $\gamma_a^i:=x^i\circ\gamma_a=x^i\circ a$. This should be compared with Hamilton's equations (the integral curve equation for the vector field $-\Ham H$):
\begin{align*}
 \dot{\gamma}_a^i&=\frac{\partial H}{\partial p_i}\circ a, & \dot{a}_j&=-\frac{\partial H}{\partial x^j}\circ a.
\end{align*}
The first equation in \eqref{eq: grad-PW} thus coincides with one of Hamilton's equations, whereas the second one differs in general. On the other hand, there is a close relation between $\grad_\nabla H$ and $\Ham H$.

\begin{lemma}\label{lem: hor-ver}
Let $\nabla$ be a torsion-free connection on $M$, $H\in\cCi(T^*M)$, and denote the projections onto the vertical and horizontal subbundle by $\pr_\mathcal{V}$ and $\pr_{\mathcal{H}_\nabla}$ respectively. Then,
 \begin{equation*}
 \grad_\nabla H=\pr_\mathcal{V}\Ham H-\pr_{\mathcal{H}_\nabla}\Ham H.
 \end{equation*}
\end{lemma}

\begin{proof}
 The result follows by a straightforward calculation in natural coordinates.
\end{proof}
% FM: I can prove it in an elegant way if the relation between the canonical 2-form and skew-symmetric pairing on T+T^*. However, it makes the paper longer... 

Let us examine some particular examples to show what mathematical and physical systems can be described by Patterson-Walker dynamics.

\begin{example}[The parallel transport equation]
 Consider a torsion-free connection $\nabla$ on $M$ and $H\in\cCi(T^*M)$ such that $\grad_\nabla H$ is horizontal at each point. By Lemma \ref{lem: hor-ver}, we know that Patterson-Walker dynamics coincides with Hamiltonian dynamics. It follows from \eqref{eq: grad-PW} that locally we have
 \begin{equation*}
 \frac{\partial H}{\partial x^j}+p_k(\pr^*\Gamma^k_{ij})\frac{\partial H}{\partial p_i}=0,
 \end{equation*}
 hence the ODE system \eqref{eq: EoM} becomes
 \begin{align*}
\dot{\gamma}_a^i&=\frac{\partial H}{\partial p_i}\circ a, & \dot{a}_j&=a_k\,(\Gamma^k_{ij}\circ\gamma_a)\,\frac{\partial H}{\partial p_i}\circ a.
\end{align*}
Using the first equation on the second one gives
\begin{equation*}
\dot{a}_j-a_k\,(\Gamma^k_{ij}\circ\gamma_a)\,\dot{\gamma}_a^i=0,
\end{equation*}
which can be expressed globally by $\nabla_{\dot{\gamma}_a}a=0$, that is, the parallel transport equation along the Hamiltonian trajectory in the configuration space.
\end{example}

\begin{example}[Hamiltonians linear in momenta]\label{ex: lin-ham}
 By the correspondence \eqref{eq: symall-polynomials}, a linear polynomial in momenta $H\in\cCi(T^*M)$ is given by a unique vector field $X\in\mathfrak{X}(M)$ as $H=X^v$. In natural coordinates,
 \begin{align*}
 \rest{X}{U}&=X^i\partial_{x^i}, & \rest{H}{T^*U}&=(\pr^*X^i)p_i.
 \end{align*}
 The system \eqref{eq: EoM} takes the form:
 \begin{align*}
 \dot{\gamma}_a^i&=X^i\circ\gamma_a, & \dot{a}_j&=a_i\Big(\frac{\partial X^i}{\partial x^j}\circ\gamma_a\Big)+2 a_k((\Gamma^k_{ij}X^i
)\circ\gamma_a).
 \end{align*}
 The first equation is the integral curve equation for the vector field $X$, whereas the second equation, by employing the first one, can be written in global form as
 \begin{equation}\label{eq: Ham-lin}
 L^s_X a=0. % in Ham dynamics you get the same for Lie derivative
 \end{equation}
 Since the symmetric Lie derivative can be interpreted in terms of the $\nabla^s$-parallel transport, see \cite[Sec. 2.5]{SymCartan}, the left-hand side of this equation is well defined as $a$ is a $1$-form along an integral curve of $X$. The equation \eqref{eq: Ham-lin} is equivalent to having, for every vector field $V$ along $\gamma_a$, the following:
 \begin{equation*}
\frac{\dif }{\dif t}a(V)=a(\pg{X,V}_s).
 \end{equation*}
 In particular, we get that $\frac{\dif}{\dif t}a(\dot{\gamma}_a)=2a(\nabla_{\dot{\gamma}_a}\dot{\gamma}_a)$, hence, if $X$ is autoparallel, that is, $\nabla_XX=0$, the function $a(\dot{\gamma}_a)$ is constant. % in Hamiltonian dynamics it is always constant with no extra assumption
\end{example}

The system we have recovered in Example \ref{ex: lin-ham} is called the \textit{gradient extension of a dynamical system}, see \cite{CorCGCS} for a relation to control theory.

\begin{remark}
Using the language of \cite{TaoUNI}, Example \ref{ex: lin-ham} shows that gradient vector fields are so-called \textit{universal}, that is, given a manifold $M$ and $X\in\mathfrak{X}(M)$, there is an embedding $\phi:M\rightarrow M'$ and a gradient vector field $Y\in\mathfrak{X}(M')$ such that $X$ and $Y$ are $\phi$-related.
\end{remark}

\begin{example}[Conservative systems]
The dynamics of a conservative system with $n$ degrees of freedom is described by the Newtonian equation:
\begin{equation}\label{eq: Newtonian eq}
 \lcn{g}_{\dot{\gamma}}\dot{\gamma}=-\grad_gf,
\end{equation}
where $g$ is a Riemannian metric on an $n$-dimensional manifold $M$ representing the kinetic energy and $f\in\cCi(M)$ represents the potential energy of the system. In Hamiltonian dynamics, equation \eqref{eq: Newtonian eq} is recovered for the Hamiltonian function $H\in\cCi(T^*M)$ given, for $\zeta\in T^*M$, by
\begin{equation*}
 H(\zeta):=\frac{1}{2}g^{-1}(\zeta,\zeta)+f(\pr(\zeta)),
\end{equation*}
or equivalently $H=(g^{-1}+f)^v$. If we consider $H_{PW}:=(g^{-1}-f)^v$ and choose the Levi-Civita connection of $g$, we obtain the Newtonian equation \eqref{eq: Newtonian eq} as arising from Patterson–Walker dynamics.
\end{example}

\subsection{Dynamical interpretation of symmetric Poisson structures} We finish this section with a dynamical interpretation of symmetric Poisson structures. This will allow us to extend the fact that the characteristic distribution is geodesically invariant from regular (Corollary \ref{cor: char-reg-sPs}) to singular structures, although only locally.

An immediate consequence of isomorphism \eqref{eq: symall-polynomials} and Proposition \ref{prop: sym-Schouten-PW} is a characterization of symmetric Poisson structures in terms of Patterson-Walker dynamics.

\begin{corollary}\label{cor: sPs-PW dynamics}
 Let $\vartheta\in\mathfrak{X}^2_\emph{sym}(M)$ and $\nabla$ be a torsion-free connection on $M$. The pair $(\vartheta,\nabla)$ is symmetric Poisson if and only if $\grad_\nabla\vartheta^v$ is isotropic with respect to $g_\nabla$, that is, $\vartheta^v\in\cCi(T^*M)$ is constant along the integral curves of $\grad_\nabla\tilde{\vartheta}$.
\end{corollary}

By \eqref{eq: symall-polynomials}, symmetric bivector fields can be identified with quadratic polynomials in momenta. We explore the corresponding Patterson-Walker dynamics.

\begin{lemma}\label{lem: quadrtaic-PW-dynamics}
 Let $\vartheta\in\mathfrak{X}^2_\emph{sym}(M)$ and $\nabla$ be a torsion-free connection on $M$. A curve $a:I\rightarrow T^*M$ is an integral curve of $\grad_\nabla\vartheta^v$ if and only if
 \begin{align}\label{eq: PW-quadratic}
\dot{\gamma}_a&=\vartheta(a),& \nabla_{\dot{\gamma}_a}a&=\frac{1}{2}(\nabla\vartheta)(a,a).
 \end{align}
\end{lemma}

\begin{proof}
 In natural coordinates, we have that
\begin{align*}
 \rest{\vartheta}{U}&=\frac{1}{2}\vartheta^{ij}\partial_{x^i}\odot\partial_{x^j}, & &\rest{\vartheta^v}{T^*U}=\frac{1}{2}(\pr^*\vartheta^{ij})p_ip_j.
\end{align*}
The integral curves of $\grad_\nabla\vartheta^v$ are determined by the system of ODEs \eqref{eq: EoM}, which in this case takes the form:
 \begin{align*}
 \dot{\gamma}^i_a&=a_j(\vartheta^{ji}\circ\gamma_a), & \dot{a}_j&=a_ia_k\Big(\frac{1}{2}\frac{\partial \vartheta^{ik}}{\partial x^j}+2\Gamma^i_{lj}\vartheta^{lk}\Big)\circ\gamma_a.
 \end{align*}
 The first equation can be expressed globally, simply by $\dot{\gamma}_a=\vartheta(a)$. On the other hand, by using the first equation on the second one, we get
 \begin{equation*}
 \dot{a}_j-(\Gamma^i_{lj}\circ\gamma_a)a_i\dot{\gamma}_a^l=a_ia_k\Big(\frac{1}{2}\frac{\partial \vartheta^{ik}}{\partial x^j}+\Gamma^i_{lj}\vartheta^{lk}\Big)\circ\gamma_a,
 \end{equation*}
 which can be expressed in coordinate-free way by $\nabla_{\dot{\gamma}_a}a=\frac{1}{2}(\nabla\vartheta)(a,a)$.
\end{proof}

\begin{remark}
For a given Poisson structure $\pi$ on $M$, there is the notion of a \textit{cotangent path}, that is, a curve $a:I\rightarrow T^*M$ such that
\begin{equation*}
 \pi(a)=\dot{\gamma}_a.
\end{equation*}
The first equation in \eqref{eq: PW-quadratic} basically tells us $a$ is a ‘cotangent path’ for $\vartheta$. This is a notion related to that of $\vartheta$-admissible curve (Definition \ref{def:theta-admissible}): given a cotangent path $a$ for $\vartheta$, the curve $\gamma_a$ on $M$ is $\vartheta$-admissible.
\end{remark}

We show that there is a necessary condition to be satisfied in order to be an integral curve of $\grad_\nabla\vartheta^v$ and it is related to the symmetric Schouten bracket.

\begin{lemma}\label{lem: quadratic-PW-necessary}
 Consider $\vartheta\in\mathfrak{X}^2_\emph{sym}(M)$, a torsion-free connection $\nabla$ on $M$, and an integral curve $a:I\rightarrow T^*M$ of $\grad_\nabla\vartheta^v$. Then,
 \begin{equation*}
 \nabla_{\dot{\gamma}_a}\dot{\gamma}_a=\frac{1}{4}\iota_a\iota_a[\vartheta,\vartheta]_s.
 \end{equation*}
\end{lemma}

\begin{proof}
 By Lemma \ref{lem: quadrtaic-PW-dynamics}, we have that $\dot{\gamma}_a=\vartheta(a)$, hence
 \begin{equation*}
 \nabla_{\dot{\gamma}_a}\dot{\gamma}_a=\nabla_{\vartheta(a)}\vartheta(a)=(\nabla_{\vartheta(a)}\vartheta)(a)+\vartheta(\nabla_{\vartheta(a)}a).
 \end{equation*}
 Employing the second equation in \eqref{eq: PW-quadratic} yields $\nabla_{\dot{\gamma}_a}\dot{\gamma}_a=(\nabla_{\vartheta(a)}\vartheta)(a)+\frac{1}{2}(\nabla_{\vartheta(\,)}\vartheta)(a,a)$. By Proposition \ref{prop: sym-Poisson-bracket}, we finally obtain
 \begin{equation*}
 \nabla_{\dot{\gamma}_a}\dot{\gamma}_a=\frac{1}{4}\iota_a\iota_a[\vartheta,\vartheta]_s.
 \end{equation*}
\end{proof}

We use Lemma \ref{lem: quadratic-PW-necessary} to find another characterization of symmetric Poisson structures, using the notion of a geodesic.

\begin{proposition}\label{prop: quadratic-PW}
 Let $\vartheta\in\mathfrak{X}^2_\emph{sym}(M)$ and $\nabla$ be a torsion-free connection on $M$. The pair $(\vartheta,\nabla)$ is a symmetric Poisson structure if and only if the curve $\gamma_a$ is a geodesic for every integral curve $a$ of $\grad_\nabla\vartheta^v$.
\end{proposition}

\begin{proof}
 It follows immediately from Lemma \ref{lem: quadratic-PW-necessary} that, if $(\vartheta,\nabla)$ is a symmetric Poisson structure, $\gamma_a$ is a geodesic for every integral curve $a$ of $\grad_\nabla\vartheta^v$. For the converse, consider an arbitrary $\zeta\in T^*M$ and find an integral curve $a$ of $\grad_\nabla\vartheta^v$ such that $a(0)=\zeta$. By Lemma \ref{lem: quadratic-PW-necessary} and the fact that $\gamma_a$ is a geodesic, we get
 \begin{equation*}
 \iota_\zeta\iota_\zeta[\vartheta,\vartheta]_s=0\in T_{\pr(\zeta)}M,
 \end{equation*}
 hence $[\vartheta,\vartheta]_s(\zeta,\zeta,\zeta)=0$. By polarization and the fact that $\zeta\in T^*M$ was chosen arbitrarily, we get that $[\vartheta,\vartheta]_s=0$, that is, $(\vartheta,\nabla)$ is a symmetric Poisson structure.
\end{proof}

We use this dynamical interpretation to describe the characteristic distribution.

\begin{definition}\label{def:loc-geod-inv}
 We call a distribution $\Delta$ on $M$ \textbf{locally geodesically invariant} for a connection $\nabla$ on $M$ if for every geodesic $\gamma:I\rightarrow M$, with $I$ an open interval, such that $\dot{\gamma}(t_0)\in\Delta_{\gamma(t_0)}$ for some $t_0\in I$, there is an open subinterval $I'\subseteq I$ containing $t_0$ such that $\dot{\gamma}(t)\in \Delta_{\gamma(t)}$ for every $t\in I'$.
\end{definition}

The distribution in Example \ref{ex:delta-not-geod-inv} is clearly locally geodesically invariant, but it is not geodesically invariant, so Definition \ref{def:loc-geod-inv} is justified.

\begin{remark}
 It is an open question whether there exists a distribution preserved by the symmetric bracket that is not even locally geodesically invariant. 
\end{remark}

\begin{theorem}\label{thm: gen-Lewis}
 The characteristic distribution of a symmetric Poisson structure is locally geodesically invariant. 
\end{theorem}

\begin{proof}
 Let $(\vartheta,\nabla)$ be a symmetric Poisson structure on $M$. Given a geodesic $\gamma:I\rightarrow M$ such that $\dot{\gamma}(t_0)\in(\im\vartheta)_{\gamma(t_0)}$ for some $t_0\in I$, we find $\zeta\in T^*_{\gamma(t_0)}M$ such that $\vartheta(\zeta)=\dot{\gamma}(t_0)$. Using $\zeta\in T^*M$ as the initial condition $a(t_0):=\zeta$, we find an integral curve $a:I'\rightarrow T^*M$ of $\grad_\nabla \vartheta^v$ defined on an open interval $I'$ containing $t_0$. By Lemma \ref{lem: quadrtaic-PW-dynamics} and Proposition \ref{prop: quadratic-PW}, we have that $\gamma_a:I'\rightarrow M$ is a geodesic with the initial velocity $\dot{\gamma}_a(t_0)=\vartheta(\zeta)$. It follows from the uniqueness of geodesics that there is an open interval $I_0\subseteq I\cap I'$ containing $t_0$ such that $\rest{\gamma}{I_0}=\rest{\gamma_a}{I_0}$. By Lemma \ref{lem: quadrtaic-PW-dynamics}, we have that $\dot{\gamma}(t)=\dot{\gamma}_a(t)=\vartheta(a(t))\in(\im\vartheta)_{\gamma(t)}$ for every $t\in I_0$.
\end{proof}

% \begin{remark}
%  All examples of symmetric Poisson structures we have found so far have geodesically invariant characteristic distribution. It remains an open question whether there exists a symmetric Poisson structure whose characteristic distribution is locally geodesically invariant but not geodesically invariant. 
% \end{remark}
% compare to Example 4.7

The following example shows that locally geodesically invariant distributions that are not geodesically invariant occur for symmetric Poisson structures.

\begin{example}\label{ex: S2}
    Consider the $2$-sphere $\SSS^2$ and the unique extension $X\in\mathfrak{X}(\SSS^2)$ of $\partial_\theta$ for the latitudinal coordinate $\theta$ of a spherical chart. In particular, $X$ vanishes at the poles of the chart. For the round connection $\nabla^\circ$, we have that $\nabla^\circ_{\partial_\theta}\partial_\theta=0$, hence, by continuity, $\nabla^\circ_XX=0$, so
    \begin{equation*}
        [X\otimes X,X\otimes X]_s=2\nabla^\circ_XX\odot X\odot X=0,
    \end{equation*}
that is, $(X\otimes X,\nabla^\circ)$ is a symmetric Poisson structure. As the characteristic distribution vanishes at the poles, it is not geodesically invariant.
\end{example}

 We use the techniques of Patterson-Walker dynamics to generalize the well-known characteristic property of Killing tensors: they induce a conserved quantity along geodesics. Given a non-degenerate symmetric bivector field $\vartheta\in\symbi$, we have the unique function $\xi\in\cCi(TM)$ such that $\vartheta^v=\xi\circ\vartheta$, explicitly, $\xi(u)=\frac{1}{2}\vartheta^{-1}(u,u)$ for $u\in TM$. If $(\vartheta,\nabla)$ is a symmetric Poisson structure, by Proposition \ref{prop: sym-Poisson-nondeg}, we have that $\vartheta^{-1}\in\kil^2_\nabla(M)$, hence $\xi$ is constant along all geodesics on $M$, that is, along all geodesics tangent to the characteristic distribution $\im\vartheta =TM$. On the other hand, if $\vartheta=0$, every single function $\xi\in\cCi(TM)$ satisfies the relation $\vartheta^v=\xi\circ\vartheta$, and moreover, every $\xi\in\cCi(TM)$ is constant along geodesics tangent to the characteristic distribution $\im\vartheta= 0$. The final result of the section gives a generalization of this observation to a symmetric Poisson structure.

 \begin{proposition}
 Let $(\vartheta,\nabla)$ be a symmetric Poisson structure. Then every function $\xi\in\cCi(TM)$ satisfying $\vartheta^v=\xi\circ\vartheta$ is constant along geodesics that are tangent to the characteristic distribution $\im\vartheta\subseteq TM$. 
 \end{proposition}

 \begin{proof}
 Let $\gamma:I\rightarrow M$ be a geodesic that is tangent to $\im\vartheta\subseteq TM$. By repeating the procedure from the proof of Theorem \ref{thm: gen-Lewis} for every $t\in I$, we find an open subinterval $I_{t}\subseteq I$ containing $t$ and an integral curve $a_{t}:I_{t}\rightarrow T^*M$ of $\grad_\nabla\vartheta^v$ such that $\gamma_{a_{t}}=\rest{\gamma}{I_{t}}$. By Corollary \ref{cor: sPs-PW dynamics}, the function $\vartheta^v\in\cCi(T^*M)$ is constant along $a_{t}$, hence, for every $\xi\in\cCi(TM)$ satisfying $\vartheta^v=\xi\circ\vartheta$, we have that
 \begin{equation*}
\xi\circ\vartheta\circ a_{t}
 \end{equation*}
 is a constant function. By Lemma \ref{lem: quadrtaic-PW-dynamics}, we have that $\vartheta\circ a_{t}=\dot{\gamma}_{a_{t}}$ and thus $\xi\circ\rest{\dot{\gamma}}{I_{t}}$ is constant. It follows from the connectedness of $I$ that $\xi\circ\dot{\gamma}$ is constant.
 \end{proof}

\section{Examples of symmetric Poisson structures}
 We present various families of examples of symmetric Poisson structures.

 \subsection{Low-dimensional manifolds}\label{sec: low-dim}
 If $\dim M=1$,
 the conditions (see Proposition \ref{prop: sym-Poisson-bracket} and Corollary \ref{cor: ssPs-integrability})
 \begin{align*}
 \nabla_{\vartheta(\alpha)}\vartheta&=0, & (\nabla_{\vartheta(\alpha)}\vartheta)(\beta,\eta)+\cyc(\alpha,\beta,\eta)&=0
 \end{align*}
 coincide, and thus symmetric Poisson structures
 are strong in dimension $1$.

 On $M=\R$, every connection $\nabla$ is uniquely represented by $ h\in\cCi(\R)$ via $\nabla_{\partial_x}\partial_x=h\,\partial_x$, and every $\vartheta\in\mathfrak{X}^2_\text{sym}(\R)$ is uniquely given by $f\in\cCi(\R)$ via $\vartheta= f\,\partial_x\otimes\partial_x$. The pair $(\vartheta,\nabla)$ is a (strong) symmetric Poisson structure if and only if
\begin{equation}\label{eq: sPs-1-dim}
 0=ff'+2f^2h=\frac{1}{2}(f^2)'+2f^2h.
\end{equation}
This is a first order linear ODE with variable coefficients for the function $f^2$, whose general solution is parametrized by a constant $\lambda\in\R$ as
\begin{equation*}
 f^2=\lambda\,\e^{-4H},
\end{equation*}
where $H$ is a primitive function of $h$. It follows that a symmetric Poisson structure on $\R$ is either trivial $(\vartheta=0)$ or non-degenerate. Since it is also always strong we get, by Proposition \ref{prop: ssPs-nondeg}, that $\nabla$ is necessarily the Levi-Civita connection of $\vartheta^{-1}$.

As all $1$-dimensional manifolds are locally diffeomorphic to $\R$ and the condition for being a symmetric Poisson structure is local, we find that the set of symmetric Poisson structures on a $1$-dimensional manifold decomposes disjointly into \textit{connections}, \textit{Riemannian metrics} and \textit{negative-definite metrics}.

\begin{example}
 The above discussion implies that, although the symmetric bivector field $\vartheta=x\,\partial_x\otimes\partial_x\in\mathfrak{X}^2_\emph{sym}(\R)$ from Example \ref{ex: sing-d} induces the totally geodesic partition of $\R$ with respect to the Euclidean connection (see Example \ref{ex: R-part}), it is not a symmetric Poisson structure for any choice of a connection on $\R$. 
\end{example}

As dimension grows the situation gets much more complicated. The characteristic ODE \eqref{eq: sPs-1-dim} becomes a system of PDEs and symmetric Poisson structures become richer. If $\dim M=2$, Examples \ref{ex: inclusion} and \ref{ex: non-deg-Kill} demonstrate that singular and non-strong symmetric Poisson structures can occur.

\subsection{Geodesic vector fields}\label{sec: geo-vf}
Given a connection $\nabla$ on $M$, recall that a vector field $X\in\mathfrak{X}(M)$ is called \textbf{geodesic} when there is $h\in\cCi(M)$ such that
\begin{equation*}
 \nabla_XX=hX.
\end{equation*}
Equivalently, the integral curves of $X$ can be reparametrized to become geodesics.

\begin{proposition}
 Let $\vartheta=f\,X\otimes X$ for some $X\in\mathfrak{X}(M)$ and $f\in\cCi(M)$ and $\nabla$ be a torsion-free connection on $M$. The pair $(\vartheta,\nabla)$ is a symmetric Poisson structure if and only if $X$ is geodesic, and moreover,
 \begin{equation*} % this generalizes {eq: sPs-1-dim}
 Xf^2 +4f^2h=0,
 \end{equation*}
 where $h\in \cCi(M)$ is uniquely determined by $\nabla_XX=hX$. In addition, every symmetric Poisson structure of this form is strong.
\end{proposition}

\begin{proof}
By the axioms of the symmetric Schouten bracket and $\vartheta=\frac{1}{2}f\,X\odot X$, 
 \begin{align*}
 2[\vartheta,\vartheta]_s&=2f[f,X]_s\odot X\odot X\odot X+2f^2[X,X]_s\odot X\odot X\\
 &=2f(Xf)\,X\odot X\odot X+4f^2\nabla_XX\odot X\odot X.
 \end{align*}
 Therefore, $[\vartheta,\vartheta]_s=0$ if and only if $\nabla_XX=hX$ for some $h\in\cCi(M)$ such that $Xf^2+4f^2h=0$. On the other hand, we have that
 \begin{equation*}
 2\nabla_{\vartheta(\alpha)}\vartheta=f\alpha(X)\,\nabla_X(f\,X\otimes X)=\alpha(X)(f(Xf)\,X\otimes X+2f^2\, \nabla_XX\otimes X),
 \end{equation*}
 so, by Corollary \ref{cor: ssPs-integrability}, $(\vartheta,\nabla)$ is a strong symmetric Poisson structure whenever it is a symmetric Poisson structure.
\end{proof}

As a regular symmetric bivector field $\vartheta$ with $\rk\vartheta=1$ can always be locally written as $\vartheta=f\,X\otimes X$, we can generalize the result in the first paragraph of Section \ref{sec: low-dim} as follows.

\begin{corollary}\label{cor: rank-1}
 A symmetric Poisson structure $(\vartheta,\nabla)$ with $\rk \vartheta=1$ is strong.
\end{corollary}

In particular, a pair $(\vartheta,\nabla)$ such that $\vartheta=\lambda\,X\otimes X$ for some constant $\lambda\in \R$ and $X\in\mathfrak{X}(M)$ is (strong) symmetric Poisson if and only if $X$ is \textbf{autoparallel}, that is, $\nabla_XX=0$, or equivalently, the integral curves of $X$ are geodesics (see Examples \ref{ex: 3-fol} and \ref{ex: S2} for an application of this). In other words, the concept of symmetric Poisson structures can be seen as a generalization of autoparallel vector fields.

\subsection{Left-invariant symmetric Poisson structures}\label{sec: li-sps}
We use the notation $\mathfrak{X}_\text{L}(G)$ and $\mathfrak{X}^2_\text{sym,L}(G)$ for the space of left-invariant vector fields and symmetric bivector fields respectively.

Every Lie group $G$ admits the natural flat left-invariant connection $\wn$ determined by the condition that all left-invariant vector fields are parallel (a special example of the so-called \textit{Weitzenböck connection}). Its torsion is
\begin{equation*}
 T_{\wn}(X,Y)=-[X,Y],
\end{equation*}
for every $X,Y\in\mathfrak{X}_\text{L}(G)$. Therefore, the connection $\wn$ is torsion-free if and only if the Lie group is abelian, and the associated torsion-free connection is given by
\begin{equation*}
 \wnt_XY=\frac{1}{2}[X,Y].
\end{equation*}
Note that $\wnt$ is no longer flat in general. We have, for $X,Y,Z\in\mathfrak{X}_\text{L}(G)$,
\begin{equation*}
 R_{\wnt}(X,Y)Z=-\frac{1}{4}[[X,Y],Z],
\end{equation*}
so, it is flat if and only if the Lie algebra is $2$\textit{-step nilpotent}, equivalently, associative.

\begin{proposition}\label{prop: left-inv-sPs}
 The pair $(\vartheta,\wnt)$ is a symmetric Poisson structure on $G$ for every $\vartheta\in\mathfrak{X}^2_\emph{sym,L}(G)$. If the Lie group is, in addition, abelian, every $\vartheta\in\mathfrak{X}^2_\emph{sym,L}(G)$ is parallel with respect to $\wnt$, hence $(\vartheta,\wnt)$ is a strong symmetric Poisson structure.
\end{proposition}

\begin{proof}
 The space $\mathfrak{X}^2_\text{sym,L}(G)$ is clearly generated by $\R$-linear combinations and the symmetric product $\odot$ from the space $\mathfrak{X}_\text{L}(G)$. The result follows easily from \eqref{eq: sym-Schouten-decomposable} and the fact that for the symmetric bracket corresponding to $\wnt$, we have that $\pg{X,Y}_s=0$ for every $X,Y\in\mathfrak{X}_\text{L}(G)$.
\end{proof}

\begin{example}
 Example \ref{ex: SO(3)} clearly fits into this construction.
\end{example}

 For $V^*$ seen as an abelian Lie group, the condition of being left-invariant for $\vartheta\in\mathfrak{X}^2_\text{sym}(V^*)$ is equivalent to that its component functions $\vartheta(\dif x^i,\dif x^j)$ are constant for any global linear coordinates $\lbrace x^i\rbrace$. The connection $\wn=\wnt$ is torsion-free in this case and we will refer to it as the \textbf{Euclidean connection} $\nabla^\text{Euc}$ on $V^*$. By Proposition \ref{prop: left-inv-sPs}, we get that every $\vartheta\in\mathfrak{X}^2_\text{sym,L}(V^*)$ is parallel with respect to $\nabla^\text{Euc}$.

However, even a non-abelian Lie group admits left-invariant strong symmetric Poisson structures with $\wnt$. What is more, they can be \mbox{non-par}allel.

\begin{example} The \textit{affine group} $\aff(1)\cong\GL(\R)\ltimes\R$ is a \mbox{$2$-dim}ensional \mbox{non-ab}elian Lie group. The standard basis $(X,Y)$ for $\mathfrak{X}_\emph{L}(\aff(1))$ satisfies
 \begin{equation*}
 [X,Y]=Y.
 \end{equation*}
 An arbitrary $\vartheta\in\mathfrak{X}^2_{\emph{sym,L}}(\aff(1))$ is determined by three constants $\lambda_1,\lambda_2,\lambda_3\in\R$ via
 \begin{equation*}
 \vartheta:=\lambda_1X\otimes X+\lambda_2 X\odot Y+\lambda_3 Y\otimes Y.
 \end{equation*}
By Proposition \ref{prop: left-inv-sPs}, we have that $(\vartheta,\wnt)$ is symmetric Poisson. A straightforward calculation shows that $(\vartheta,\wnt)$ is strong if and only if $\lambda_1\lambda_3-\lambda_2^2=0$, that is, $\vartheta$ is
 degenerate. It is, moreover, parallel if and only if $\vartheta=0$.
\end{example}

Note that it is not always the case that $\vartheta\in\mathfrak{X}^2_\text{sym,L}(G)$ has to be degenerate in order for $(\vartheta,\wn^0)$ to be strong symmetric Poisson on a non-abelian Lie group.

% Symmetric Poisson structures on Hopf fibrations
% https://repository.upenn.edu/server/api/core/bitstreams/ebce3945-ef57-4d00-a658-9db5e2c4e468/content
\begin{example}
The round metric $g^\circ$ on the $3$-sphere $\SSS^3$ seen as a Lie group is the left-invariant Riemannian metric induced by the Cartan-Killing form $\mathfrak{c}$ on $\mathfrak{su}(2)$:
 \begin{equation*}
 g^\circ(X,Y)=-\mathfrak{c}(X_{e},Y_{e}),
 \end{equation*}
 for $X,Y\in\mathfrak{X}_\emph{L}(\SSS^3)$. By a straightforward calculation, we find that the connection $\wn^0$ is actually the Levi-Civita connection of $g^\circ$, hence, by Proposition \ref{prop: ssPs-nondeg}, we get that $((g^\circ)^{-1},\wn^0)$ is a strong symmetric Poisson structure on $\SSS^3$.
\end{example}

We show that $\SSS^3$ admits also degenerate left-invariant strong symmetric Poisson structures and demonstrate their relation to the \textit{Hopf fibration}.

\begin{example}
As every $X\in\mathfrak{X}_\emph{L}(\SSS^3)$ is autoparallel with respect to $\wnt$, we have that $(X\otimes X,\wnt)$ is always a strong symmetric Poisson structure on $\SSS^3$. The corresponding totally-geodesic regular foliation is given by individual geodesics that are the integral curves of $X$. Therefore, the leaves of the foliation are great circles (diffeomorphic to $\SSS^1)$ with the round metric.

If we represent $\SSS^3$ as unit quaternions, left-invariant vector fields on $\SSS^3$ can be seen as imaginary quaternions. Evaluating the vector field corresponding to $a i+bj+ck\in\mathbb{H}$ at a point $m\in\SSS^3\subseteq\mathbb{H}$ is given by the quaternionic multiplication as $m(ai+bj+ck)\in\mathbb{H}\cong T_m\SSS^3$. Integral curves $\gamma(t)=x(t)+y(t) i+z(t) j+w(t)k$, of such a vector field are thus given by the system of homogeneous linear ODEs
\begin{equation*}
 \begin{pmatrix}
 \dot{x}\\
 \dot{y}\\
 \dot{z}\\
 \dot{w}
 \end{pmatrix}=\begin{pmatrix}
 0 & -a &-b &-c\\
 a &0 &c &-b\\
 b &-c &0 & a\\
 c &b &a &0
 \end{pmatrix}\begin{pmatrix}
 x\\
 y\\
 z\\
 w
 \end{pmatrix}.
\end{equation*}
In the special case when $X$ is given by $a=1$, $b=c=0$, we easily find the solution:
\begin{align*}
 x(t)&=x(0)\cos t-y(0)\sin t, & z(t)&=z(0)\cos t+w(0)\sin t,\\
 y(t)&=y(0)\cos t+x(0)\sin t, & w(t)&=w(0)\cos t-z(0)\sin t.
\end{align*}
Indeed, the image of the maximal integral curve of $X$ starting at $m\in\SSS^3$ is the fibre through $m$ of the Hopf fibration $\SSS^3\rightarrow \SSS^2$.

\end{example}

 Torsion-free connections on $G$ other than $\wnt$ can also make a left-invariant symmetric bivector field a symmetric Poisson structure.

 \begin{example}\label{ex: involutive-(s)sPs} 
 Consider the $3$-dimensional Lie group $G:=\aff(1)\times\R$. For the standard basis $(X,Y,Z)$ of $\mathfrak{X}_\text{L}(G)$, we have
 \begin{align*}
 [X,Y]&=Y, &[X,Z]&=0, &[Y,Z]&=0.
 \end{align*}
 The left-invariant symmetric bivector field $\vartheta:=X\odot Y$ clearly defines an involutive symmetric Poisson structure $(\vartheta,\wnt)$ on $G$ that is not strong as we have $\wnt_X\vartheta=\frac{1}{2}\vartheta$. However, if we choose a different torsion-free connection $\nabla$, the pair $(\vartheta,\nabla)$ can be made a strong symmetric Poisson structure. For instance, consider the torsion-free connection $\nabla$ determined by
 \begin{align*}
 \nabla_XX&:=-X, & \nabla_YY&:=0, & \nabla_ZZ&:=0,\\
 \nabla_XY&:=Y, & \nabla_XZ&:=0, & \nabla_YZ&:=0.
 \end{align*}
 It is easy to check that $(\vartheta,\nabla)$ is not only strong but $\vartheta$ is parallel with respect to $\nabla$.
\end{example}

Example \ref{ex: involutive-(s)sPs} shows that even non-strong symmetric Poisson structures can be involutive and thus give a smooth partition, which justifies Definition \ref{def: inv-sPs}. A natural question arises: \textit{Is to possible to find an involutive symmetric Poisson structure that cannot be made strong by changing the torsion-free connection?} We address this question in Example \ref{ex: R5} and give the positive answer to it.

\subsection{Types of symmetric Poisson structures}
We show how different types of symmetric Poisson structures are related to each other in Figure \ref{fig: diagram}. We use the notation of Propositions \ref{prop: sym-Poisson-nondeg} and \ref{prop: ssPs-nondeg}, in particular, by \mbox{(pseudo-)Rie}mannian metrics we mean pairs consisting of the inverse of a \mbox{(pseudo-)}Riemannian metric and its Levi-Civita connection.

\begin{figure}[ht]
\begin{center}
\includegraphics[scale=1.3]{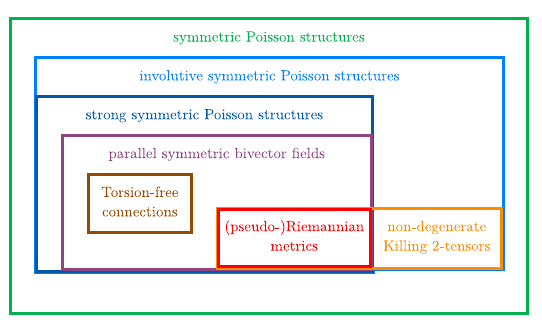}
\end{center}\vspace{-.5cm}
\caption{Types of symmetric Poisson structures.}\label{fig: diagram}
\end{figure}

Examples \ref{ex: inclusion}, \ref{ex: SO(3)}, \ref{ex: non-deg-Kill}, \ref{ex: involutive-(s)sPs} and the fact that we can easily find parallel symmetric bivector fields that are neither zero nor non-degenerate (e.g. a left-invariant degenerate symmetric bivector field on an abelian Lie group with $\wnt$) show that all the inclusions in Figure \ref{fig: diagram} are strict (when regarded as classes of manifolds admitting such structures).

\section{\for{toc}{Linear symmetric Poisson structures and Jacobi-Jordan algebras}\except{toc}{Linear symmetric Poisson structures\\and Jacobi-Jordan algebras}}
Let us recall a well-known characterization of Lie algebras in the language of Poisson geometry:
\begin{equation*}
 \left\{\begin{array}{c}
\text{linear Poisson}\\
 \text{structures on }V^*
 \end{array}\right\}\overset{\sim}{\longleftrightarrow}\left\{\begin{array}{c}
 \text{Lie algebra}\\
 \text{structures on }V
 \end{array}\right\}, 
\end{equation*}
see e.g. \cite{CraLPG}. The aim of this section is to find an analogous result for symmetric Poisson structures on $M=V^*$.

In the case $M=V^*$, a symmetric Poisson bracket $\lbrace\,,\rbrace$ is fully determined by its restriction to the space of linear functions $\mathcal{C}_\text{lin}^\infty(V^*)$. It follows from the fact that
\begin{equation}\label{eq: bracket}
 \po{f,\cg}=\po{x^i,x^j}\frac{\partial f}{\partial x^i}\frac{\partial \cg}{\partial x^j},
\end{equation}
where $\{ x^i\}$ are any global linear coordinates on $V^*$. 

\begin{definition}
A symmetric Poisson bracket $\lbrace\,,\rbrace$ on $\cCi(V^*)$ is called \textbf{linear} if
 \begin{equation*}\label{eq: lin-sPs}
 \po{\mathcal{C}^\infty_\text{lin}(V^*),\mathcal{C}^\infty_\text{lin}(V^*)}\subseteq\mathcal{C}^\infty_\text{lin}(V^*).
 \end{equation*}
\end{definition}

\subsection{Characterization of linear symmetric Poisson structures}
 As there is the canonical identification $\mathcal{C}_\text{lin}^\infty(V^*)\cong V$ and a symmetric Poisson bracket is fully determined by its values on linear functions by \eqref{eq: bracket}, every linear bracket $\lbrace\,,\rbrace$ is equivalent to a commutative algebra structure $\cdot$ on $V$ given, for $u, v \in V$, by
\begin{equation}\label{eq: J-J product}
 \iota_{u\cdot v}:=\po{\iota_u,\iota_v}.
\end{equation}

\begin{lemma}\label{lem: ass-Jac-Jord}
 A pair $(\lbrace\,,\rbrace,\nabla)$ consisting of a linear symmetric Poisson bracket and a torsion-free connection on $V^*$ is a strong symmetric Poisson structure if and only if, for every $u,v,w\in V$, we have
\begin{align*}
 u\cdot (v\cdot w)&=0 & &\text{and} & \nabla_{X_{\iota_u}}X_{\iota_v}&=0.
\end{align*}
\end{lemma}

\begin{proof}
By definition, $(\lbrace\,,\rbrace, \nabla)$ is strong symmetric Poisson if and only if
 \begin{equation}\label{eq: lin-ssPS}
 \po{\iota_w,\po{\iota_u,\iota_v}}=(\dif\iota_w)(\pg{X_{\iota_u},X_{\iota_v}}_s)
 \end{equation}
 for every $u,v,w\in V$. By \eqref{eq: J-J product}, the left hand side of \eqref{eq: lin-ssPS} is equal to
 \begin{equation}\label{eq: lin-ssPs-LHS}
 \iota_{w\cdot (u\cdot v)},
 \end{equation}
 which is a linear function on $V^*$. Choosing a basis $\{ e_i\}$ for $V$, we get global linear coordinates $\{ x^i\}$ on $V^*$ and find that
 \begin{align*}
 \dif \iota_w&=w_i\dif x^i, & X_{\iota_u}&=\iota_{u\cdot e_i}\partial_{x^i}, & X_{\iota_v}&=\iota_{v\cdot e_i}\partial_{x^i},
 \end{align*}
 where $\{ w_i\}, \{ u_j\}, \{ v_k\}$ are the components of the vectors $w,u,v$ with respect to $\{ e_i\}$. It follows from torsion-freeness of $\nabla$ that
 \begin{equation}\label{eq: lin-ssPs-RHS}
 (\dif\iota_w)(\pg{X_{\iota_u},X_{\iota_v}}_s)=\iota_{u\cdot(v\cdot w)+v\cdot(u\cdot w)}+2 w_k \iota_{v\cdot e_i}\iota_{u\cdot e_j}\Gamma^k_{ij},
 \end{equation}
 where $\{ \Gamma^k_{ij} \}\subseteq \cCi(V^*)$ are Christoffel symbols of $\nabla$ corresponding to global coordinates $\{ x^i\}$. That fact that \eqref{eq: lin-ssPs-LHS} is equal to \eqref{eq: lin-ssPs-RHS} is thus equivalent to
 \begin{equation}\label{eq: lin-ssPs-2sides}
 \iota_{w\cdot(u\cdot v)-u\cdot(w\cdot v)-(w\cdot u)\cdot v}=2 w_k \iota_{v\cdot e_i}\iota_{u\cdot e_j}\Gamma^k_{ij}.
 \end{equation}
While the left-hand side is a linear function on $V^*$, the right-hand side is a linear combination of quadratic polynomials in $\{ x^i\}$ with coefficients in $\cCi(V^*)$, hence \eqref{eq: lin-ssPs-2sides} is satisfied if and only if both sides vanish independently. Vanishing of the left-hand side of \eqref{eq: lin-ssPs-2sides} is equivalent to
\begin{equation*}
 w\cdot(u\cdot v)=u\cdot(w\cdot v)+(w\cdot u)\cdot v.
\end{equation*}
Swapping the role of $w$ and $v$ and summing the two equations results into
\begin{equation}\label{eq: lin-ssPs-(1)}
 2\,u\cdot (v\cdot w)=0.
\end{equation}
On the other hand, using \eqref{eq: lin-ssPs-(1)} for the right-hand side of \eqref{eq: lin-ssPs-2sides} we obtain
\begin{equation*}
 0=w_k \iota_{v\cdot e_i}\iota_{u\cdot e_j}\Gamma^k_{ij}=w_k \iota_{v\cdot e_i}\iota_{u\cdot e_j}\Gamma^k_{ij}+\iota_{u\cdot (v\cdot w)}=(\dif \iota_w)(\nabla_{X_{\iota_u}} X_{\iota_v}),
\end{equation*}
 which is equivalent to $\nabla_{X_{\iota_u}}X_{\iota_v}=0$.
\end{proof}

For the next result, we use the Jacobiator for the product $\cdot$, that is,
\begin{equation*}
 \jac(u,v,w)=u\cdot (v\cdot w)+\cyc(u,v,w).
\end{equation*}

\begin{proposition}\label{prop: linear-sPs}
 A pair $(\lbrace\,,\rbrace,\nabla)$ consisting of a linear symmetric Poisson bracket and a torsion-free connection on $V^*$ is a symmetric Poisson structure if and only if, for every $u,v,w\in V$, we have
\begin{align*}
 \jac(u,v,w)&=0 & &\text{and} & (\dif\iota_w)(\nabla_{X_{\iota_u}}X_{\iota_v})+\cyc(u,v,w)&=0.
\end{align*}
\end{proposition}

\begin{proof}
By Proposition \ref{prop: sym-Poisson-bracket}, $(\lbrace\,,\,\rbrace,\nabla)$ is symmetric Poisson if and only if 
\begin{equation*}
 \po{\iota_w,\po{\iota_u,\iota_v}}-(\dif\iota_w)(\pg{X_{\iota_u},X_{\iota_v}}_s)+\cyc (u,v,w)=0.
\end{equation*}
Following the same approach as in the proof of Lemma \ref{lem: ass-Jac-Jord} we get the result.
\end{proof}

The first condition in Proposition \ref{prop: linear-sPs}, $\jac(u,v,w)=0$, means that every linear symmetric Poisson structure on $V^*$ gives a \textit{Jacobi-Jordan algebra structure} on $V$.

\begin{definition}[\cite{JJaBur}]
 A commutative algebra $(\mathcal{J},\cdot)$ over a field $\mathbb{K}$ is called a \textbf{Jacobi-Jordan algebra} if, for every $u,v,w\in\mathcal{J}$, we have
 \begin{equation*}
 \jac(u,v,w)=0.
 \end{equation*}
\end{definition}

\begin{remark}
 As the name suggests, Jacobi-Jordan algebras are Jordan algebras. Indeed, the defining axiom for Jordan algebras is satisfied as for every $u,v\in \mathcal{J}$,
\begin{equation*}
 ((u\cdot u)\cdot v)\cdot u- (u\cdot u)\cdot (v\cdot u)=\jac(u, u\cdot v,u)-u\cdot\jac(u,u,v).
\end{equation*} 
Jacobi-Jordan algebras are also referred to as \textit{Jordan algebras of nil-index $3$}, \textit{Lie-Jordan algebras} or \textit{\mbox{mock-Lie} algebras}, see \cite{ZusSEML} and the references therein.
\end{remark}

The relation between linear symmetric Poisson structures and Jacobi-Jordan algebras becomes one-to-one once we fix the connection to be the Euclidean one.

\begin{theorem}\label{thm: JJA-1-to-1}
 The assignment $\iota_{u\cdot v}:=\lbrace \iota_u,\iota_v\rbrace$ gives a bijection
 \begin{equation*}
 \left\{\begin{array}{c}
\text{linear (strong) symmetric Poisson}\\
 \text{structures $(\lbrace\,,\rbrace,\nabla^\emph{Euc})$ on }V^*
 \end{array}\right\}\overset{\sim}{\longleftrightarrow}\left\{\begin{array}{c}
 \text{(associative) Jacobi-Jordan}\\
 \text{algebra structures $\cdot$ on }V
 \end{array}\right\}.
 \end{equation*}
\end{theorem}
\begin{proof}
Given a Jacobi-Jordan structure $\cdot$ on $V$, we construct a linear symmetric Poisson bracket $\lbrace\,,\rbrace$ on $\cCi(V^*)$ by the relation $\lbrace \iota_u,\iota_v\rbrace:=\iota_{u\cdot v}$. For every $u,v,w\in V$, we have
 \begin{equation}\label{eq: 1-to-1-JJ}
 (\dif\iota_w)(\nabla^\text{Euc}_{X_{\iota_u}}X_{\iota_v})=\iota_{u\cdot(v\cdot w)+v\cdot(u\cdot w)}+2 w_k \iota_{v\cdot e_i}\iota_{u\cdot e_j}{\Gamma}^k_{ij}=\iota_{u\cdot(v\cdot w)+v\cdot(u\cdot w)},
\end{equation}
where we use that the Christoffel symbols of the Euclidean connection $\lbrace\Gamma^k_{ij}\}$ vanish identically. The cyclic permutation of \eqref{eq: 1-to-1-JJ} yields
 \begin{equation*}
 (\dif\iota_w)(\nabla^\text{Euc}_{X_{\iota_u}}X_{\iota_v})+\cyc(u,v,w)=2\iota_{\Jac(u,v,w)}=0,
 \end{equation*}
that is, by Proposition \ref{prop: linear-sPs}, $(\lbrace\,,\,\rbrace,\nabla^\text{Euc})$ is a linear symmetric Poisson structure. 

Conversely, by Proposition \ref{prop: linear-sPs}, the same relation $\iota_{u\cdot v}:=\lbrace \iota_u,\iota_v\rbrace$ maps a linear symmetric Poisson structure $(\lbrace\,,\rbrace,\nabla^\text{Euc})$ on $V^*$ to a Jacobi-Jordan algebra.

The claim about the strong case follows easily from Lemma \ref{lem: ass-Jac-Jord}, \eqref{eq: 1-to-1-JJ} and the fact that
 \begin{equation*}
 u\cdot(v\cdot w)=\frac{1}{3}(\jac(u,v,w)+\assoc(u,v,w)+\assoc(u,w,v)),
 \end{equation*}
 where $\assoc(u,v,w):=u\cdot(v\cdot w)-(u\cdot v)\cdot w$ is the \textit{associator}.
\end{proof}

\subsection{Examples of linear symmetric Poisson structures}
Theorem \ref{thm: JJA-1-to-1} together with Sections \ref{sec: geo-int}, \ref{sec: geo-int-strong} and \ref{sec: PW} give a clear geometrical interpretation of the notion of a real Jacobi-Jordan algebra. In \cite{JJaBur}, the classification of Jacobi-Jordan algebras for $\dim V\leq 6$ is given. Using Theorem \ref{thm: JJA-1-to-1}, we translate some of their results into the language of symmetric Poisson geometry. First we restate \cite[Prop. 3.1]{JJaBur}.

\begin{proposition}\label{prop: JJaBur0}
 All linear symmetric Poisson structures $(\vartheta,\nabla^\emph{Euc})$ on $V^*$ are strong if $\dim V\leq 4$. By a suitable choice of global linear coordinates $\lbrace x,y,z,t\rbrace$ on $V^*$, a non-trivial ($\vartheta\neq 0$) linear symmetric Poisson structure $(\vartheta,\nabla^\emph{Euc})$ on $V^*$ can be written in one of the form listed in Table \ref{tab: lin-sPs}.
\end{proposition}

\begin{table}[ht]
 \begin{center}
  \SetTblrInner{rowsep=0.6ex}
\begin{tblr}{width=\textwidth, colspec={m{0.07\textwidth,c}||m{0.31\textwidth,c}|m{0.14\textwidth,c}|m{0.345\textwidth,c}},cell{4}{1}={r=2}{c},cell{6}{1}={r=5}{c}}
$\dim V$ & $\vartheta$ making $(\vartheta,\nabla^\text{Euc})$ linear\break symmetric Poisson & leaf dimensions & non-trivial leaf\break metric signatures\\
\hline
\hline
$1$ & – &– &–\\
\hline
$2$&$y\,\partial_x\otimes \partial_x$ & $0,1$ &$(1,0),(0,1)$ \\
\hline
$3$ & $z\,\partial_x\otimes\partial_x$ & $0,1$ &$(1,0),(0,1)$\\
&$z\,(\partial_x\otimes\partial_x+\partial_y\otimes\partial_ y)$ &$0,2$ &$(2,0),(0,2)$\\
\hline
$4$ & $t\,\partial_x\otimes\partial_x$ & $0,1$&$(1,0),(0,1)$\\
&$t\,(\partial_x\otimes\partial_x+\partial_y\otimes\partial_y)$ &$0,2$ &$(2,0),(0,2)$\\
&$t\,\partial_x\otimes\partial_x+z\,\partial_y\otimes\partial_y$ & $0,1,2$ &$(1,0),(0,1),(2,0),(0,2),(1,1)$\\
&$t\,\partial_x\otimes\partial_x+z\,\partial_x\odot\partial_y$ & $0,1,2$ & $(1,0),(0,1),(1,1)$\\
&$t\,(\partial_x\otimes\partial_x+\partial_y\odot\partial_z)$ & $0,3$& $(2,1),(1,2)$\\

\end{tblr} 
\end{center}
 \caption{Non-trivial linear symmetric Poisson structures, $\dim V\leq 4$.}\label{tab: lin-sPs} 
 \end{table}

In particular, by Theorem \ref{thm:nabla-geodesic-part-of-inv-sym-Poisson}, every linear symmetric Poisson structure $(\vartheta,\nabla^\text{Euc})$ on $V^*$ with $\dim V\leq 4$ gives a totally geodesic partition, see Figure \ref{fig: JJ}. Note that the corresponding leaf metrics are always Riemannian (or negative definite) for $\dim V\leq 3$. If $\dim V=4$, this is no longer true as the last three symmetric Poisson structures in Table \ref{tab: lin-sPs} additionally possess Lorentzian leaves. 

\begin{figure}[ht]
\begin{center}
\includegraphics[height=4cm,keepaspectratio]{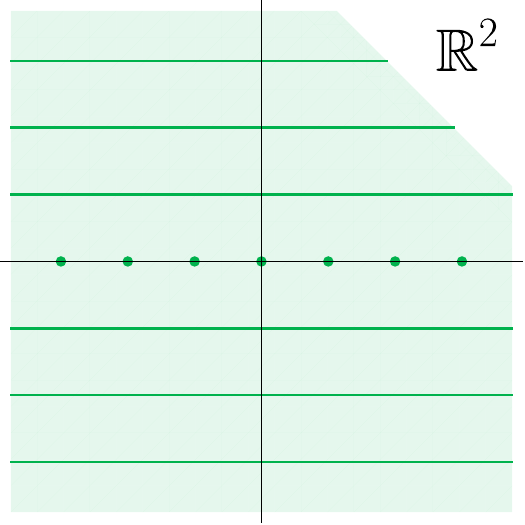}
\hspace{15pt}
\includegraphics[height=4cm,keepaspectratio]{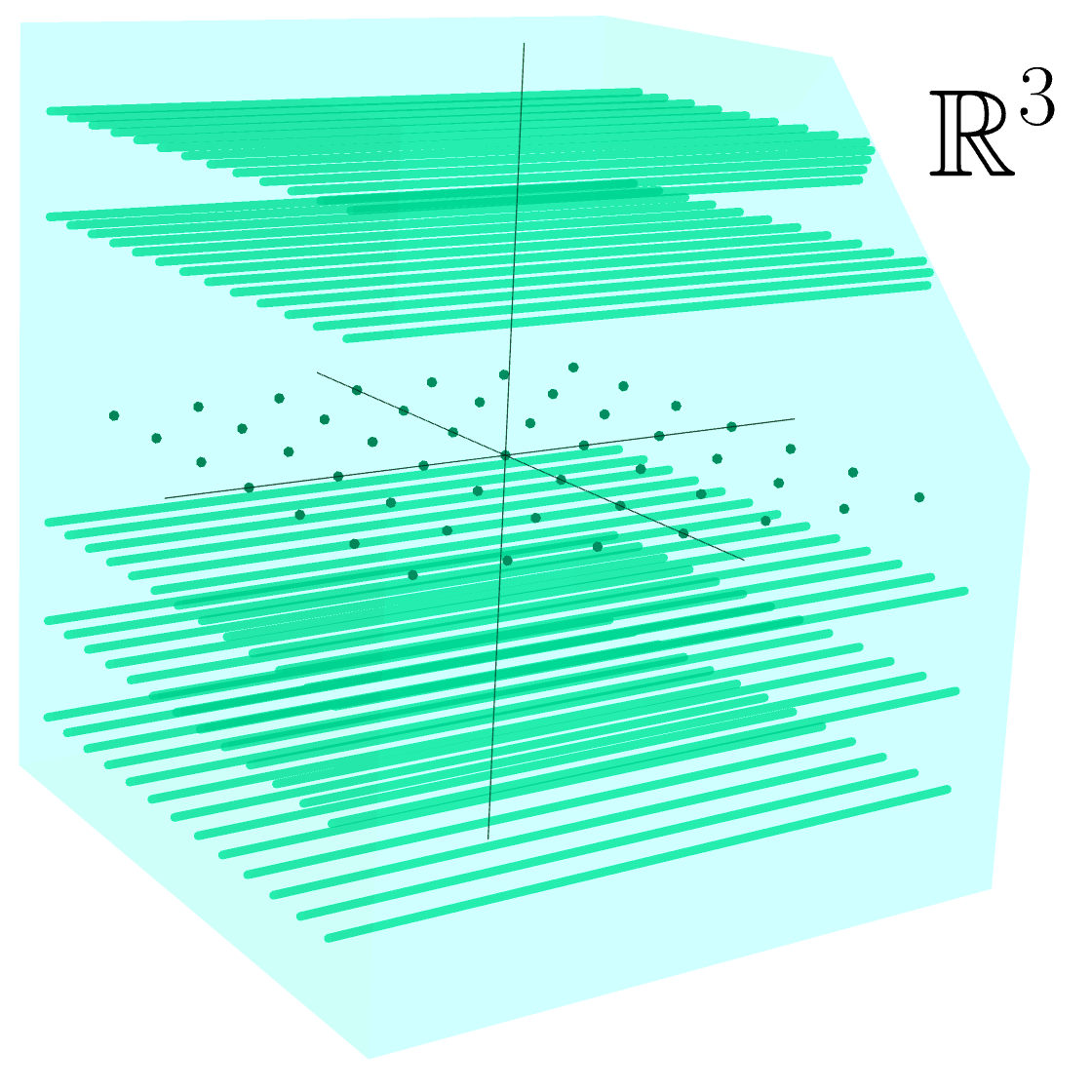}
\hspace{15pt}
\includegraphics[height=4cm,keepaspectratio]{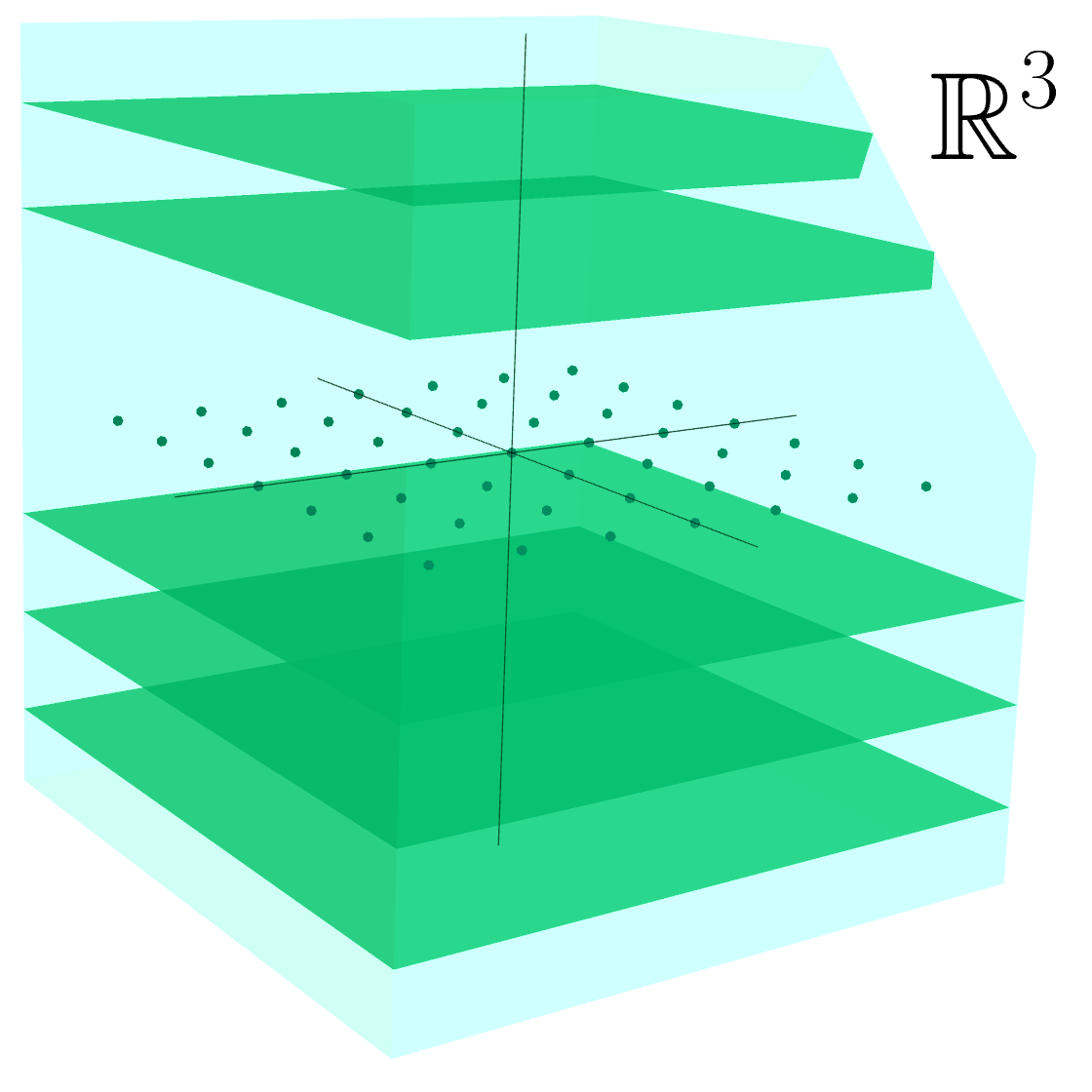}
\end{center}
\caption{Partitions induced by
Jacobi-Jordan algebras up to dimension $3$.}\label{fig: JJ}
\end{figure}

\begin{remark}
 Besides the cases listed in Proposition \ref{prop: JJaBur0}, there is always (for any dimension of $V$) the trivial linear strong symmetric Poisson structure $(0,\nabla^\text{Euc})$ on $V^*$. The corresponding totally geodesic partition is the partition into points. 
\end{remark}

We continue by rephrasing \cite[Prop. 4.1]{JJaBur} using symmetric Poisson geometry.

\begin{proposition}\label{prop: JJaBur}
 If $\dim V=5$, every non-strong linear symmetric Poisson structure $(\vartheta,\nabla^\emph{Euc})$ on $V^*$ can be written, by a suitable choice of linear global coordinates $\{ x^i\}$, as 
 \begin{equation*}\label{eq: R5}
 \vartheta:=x^2\partial_{x^1}\,\otimes\partial_{x^1}+x^5\,\partial_{x^1}\odot\partial_{x^4}-\frac{1}{2}x^3\,\partial_{x^1}\odot\partial_{x^5}+x^3\,\partial_{x^2}\odot\partial_{x^4}.
 \end{equation*}
\end{proposition}

The linear symmetric Poisson structure from Proposition \ref{prop: JJaBur} is not only interesting from the algebraic viewpoint (it corresponds to the lowest-dimensional non-associative Jacobi-Jordan algebra), but also from the geometric one.

\begin{example}\label{ex: R5}
 Consider the symmetric Poisson structure $(\vartheta, \nabla^\emph{Euc})$ on $\R^5$ from Proposition \ref{prop: JJaBur}. The characteristic module $\mathcal{F}_\vartheta$ is generated by
 \begin{align*}
 X_1&:=
x^2\partial_{x^1}+x^5\partial_{x^4}-\frac{1}{2}x^3\partial_{x^5}, & X_2&:=x^3\partial_{x^4},\\
X_3&:=x^5\partial_{x^1}+x^3\partial_{x^2}, & X_4&:=x^3\partial_{x^1}.
 \end{align*}
 The only non-trivial commutator is $ [X_1,X_3]=\frac{1}{2}X_4$, hence, although $(\vartheta,\nabla^\emph{Euc})$ is not strong, it is involutive. By Theorem \ref{thm:nabla-geodesic-part-of-inv-sym-Poisson}, we get a totally-geodesic partition of $\R^5$. By a straightforward calculation, one finds the characteristic distribution:
 \begin{equation*}
 \im\vartheta=\begin{cases}
 \{0\} & x^3=x^5=x^2=0,\\
 \spann\{\partial_{x^1}\} &x^3=x^5=0, \,x^2\neq 0,\\
 \spann\{ \partial_{x^1},\partial_{x^4}\} &x^3=0, \,x^5\neq 0,\\
 \spann\{\partial_{x^1},\partial_{x^2},\partial_{x^4},\partial_{x^5}\} &x^3\neq 0.
 \end{cases}
 \end{equation*}

 The $4$-dimensional leaves $N^{4}_t$ given by $x^3=t$ for a constant $t\in\R$, $t\neq 0$, acquire the leaf metrics:
 \begin{equation*}
 g_{N^{4}_t}=-\frac{2}{t}\dif x^1\odot\dif x^5+\frac{1}{t}\dif x^2\odot\dif x^4 + \frac{2x^5}{t^2}\dif x^2\odot \dif x^5-\frac{4x^2}{t^2}\dif x^5\otimes\dif x^5,
\end{equation*}
all of which have split signature $(2,2)$. The $2$-dimensional leaves that are the planes $N^{2}_{a,b}$: $x^3=0$, $x^5=a$, $x^2=b$ for constants $a,b\in\R$, $a\neq 0$, in $\R^5$ are equipped with the Lorentzian leaf metrics
\begin{equation*}
 g_{N^{2}_{a,b}}=\frac{1}{a}\dif x^1\odot\dif x^4-\frac{b}{a^2}\dif x^4\otimes \dif x^4.
\end{equation*}
Finally, the $1$-dimensional leaves that are the lines $N^{1}_{b,c}$ described by $x^3=x^5=0$ and $x^2=b$, $x^4=c$ for constants $b,c\in \R$, $b\neq0$, inherit the Riemannian (or negative-definite) metrics
\begin{equation*}
 g_{N^{1}_{b,c}}=\frac{1}{b}\dif x^1\otimes\dif x^1.
\end{equation*}
The rest of the leaves are just points in $\R^5$. 

The leaf connection on any $N^{4}_t$, which is the Euclidean one, is clearly not the Levi-Civita connection of $g_{N^{4}_t}$. On the other hand, for any other leaf, the leaf connection coincides with the Levi-Civita connection of the leaf metric. Thus, although all leaves are topologically cartesian spaces, the $4$-dimensional leaves are curved from a geometric viewpoint. Indeed, the Riemann curvature is given by
\begin{align*}
    R_t=\,&-\frac{1}{t^2}\dif x^1\wedge\dif x^2\otimes\Big(\frac{1}{2}\dif x^1\otimes\partial_{x^5}+\dif x^4\otimes\partial_{x^5}\Big)\\
    &-\frac{1}{t^2}\dif x^1\wedge\dif x^4\otimes\Big(\dif x^4\otimes\partial_{x^1}-\big(\frac{x^2}{t}\dif x^1+\frac{x^5}{t}\dif x^4-\frac{1}{2}\dif x^5\big)\otimes \partial_{x^2}+\frac{1}{2}\dif x^1\otimes\partial_{x^4}\\
    &+\big(\frac{x^5}{t}\dif x^1+\dif x^2\big)\otimes \partial_{x^5}\Big)+\frac{1}{t^2}\dif x^4\wedge\dif x^5\otimes\Big(\frac{1}{2}\dif x^1\otimes\partial_{x^2}+\dif x^4\otimes\partial_{x^5}\Big)
\end{align*}
By contractions, we get the Ricci and scalar curvatures:
\begin{align*}
    \ric_t&=-\frac{1}{t^2}(\dif x^1\otimes \dif x^1+2\dif x^4\otimes\dif x^4), & \mathcal{R}_t&=-\frac{x^2}{t^2}.
\end{align*}
%This is perfectly consistent with the fact that $(\vartheta,\nabla^\emph{Euc})$ is involutive but not a strong symmetric Poisson structure, see Theorem \ref{thm:nabla-geodesic-part-of-inv-sym-Poisson}.
\end{example}

Example \ref{ex: R5} is especially significant as it gives the positive answer to the question posed after Example \ref{ex: involutive-(s)sPs}. It is indeed an involutive symmetric Poisson structure that cannot be made strong by any choice of a torsion-free connection because it comes from a non-associative Jacobi-Jordan algebra, see Lemma \ref{lem: ass-Jac-Jord}.

\begin{example}
   By \cite[Prop. 5.1]{JJaBur}, there are exactly five non-associative Jacobi-Jordan algebras up to isomorphism in dimension $6$. The corresponding symmetric bivector fields are
   \begin{align*}
       \vartheta\coloneqq\,& x^2\partial_{x^1}\otimes\partial_{x^1}+x^5\partial_{x^1}\odot\partial_{x^3}+x^6\partial_{x^1}\odot\partial_{x^4}+x^4\partial_{x^3}\otimes\partial_{x^3}-\frac{1}{2}x^6\partial_{x^3}\odot\partial_{x^5}\\
       \vartheta_{\lambda,\mu}\coloneqq\,& x^2\partial_{x^1}\otimes\partial_{x^1}-\lambda x^2\partial_{x^1}\odot\partial_{x^3}+x^5\partial_{x^1}\odot\partial_{x^4}-\frac{1}{2}x^6\partial_{x^1}\odot\partial_{x^5}+x^6\partial_{x^2}\odot\partial_{x^4}\\
       &+\mu x^6\partial_{x^3}\otimes \partial_{x^3}+\lambda x^6\partial_{x^3}\odot\partial_{x^5}
   \end{align*}
   for $\lambda,\mu\in\{0,1\}$. By a direct computation, we find that the characteristic modules of $\vartheta$ and $\vartheta_{0,\mu}$ are involutive, whereas the characteristic modules of $\vartheta_{1,\mu}$ are not involutive. The pairs $(\vartheta_{1,0},\nabla^\emph{Euc})$ and $(\vartheta_{1,1},\nabla^\emph{Euc})$ provide the first  examples of singular non-involutive symmetric Poisson structures.
\end{example}

%Note that all previously given examples of non-involutive symmetric Poisson structures were regular.

Since two Jacobi-Jordan algebras are isomorphic if and only if their corresponding symmetric bivector fields can be brought to the each other by a change of global linear coordinates, we get Jacobi-Jordan algebra invariants from the symmetric Poisson geometry: the rank of the characteristic distribution, the signature of the characteristic metric and many more. The translation of these invariants to the language of Jacobi-Jordan algebras and their possible applications is an interesting question.

\clearpage
\appendix

\section{\for{toc}{Complementary results on the symmetric Schouten bracket}\except{toc}{Complementary results on the\\symmetric Schouten bracket}}\label{app: compl-s-Schouten}
As an application of the symmetric Schouten bracket (Definition \ref{def: s-Schouten}), we show that it can be used to characterize Killing tensors of a \mbox{(pseudo-)Rie}mannian metric $g$ on $M$. Recall that the metric $g$ induces an isomorphism of unital graded algebras $\symall$ and $\Upsilon^\bullet(M)$ explicitly given, for $\mathcal{X}\in\symr$ and $X_j\in\mathfrak{X}(M)$, by 
\begin{equation*}
 g(\mathcal{X})(X_1\varlist X_r)=\mathcal{X}(g(X_1)\varlist g(X_r)).
\end{equation*}

\begin{proposition}
 Let $g$ be a \mbox{(pseudo-)Rie}mannian metric on $M$. A symmetric form $K\in\Upsilon^r(M)$ is a Killing tensor of $g$ if and only if
 \begin{equation*}
 [g^{-1},g^{-1}(K)]_s=0,
 \end{equation*}
 where $[\,\,,\,]_s$ is given by the Levi-Civita connection of $g$.
\end{proposition}

\begin{proof}
 By Proposition \ref{prop: s-Schouten-explicit} and the fact that $g$ is parallel with respect to its \mbox{Levi-Civ}ita connection $\lcn{g}$, we get
 \begin{equation*}
 [g^{-1},g^{-1}(K)]_s=\tr(\iota_\star g^{-1}\odot\lcn{g}_\star\, g^{-1}(K))=(r+1)\sym (\lcn{g}_{g^{-1}(\,\,)}g^{-1}(K)).
 \end{equation*}
 As $g^{-1}$ commutes with a covariant derivative given by its Levi-Civita connection and also with the symmetric projection, we find, by using the definition of the symmetric derivative \eqref{eq: sym-der-def}, that
 \begin{equation*}
 [g^{-1},g^{-1}(K)]_s=g^{-1}(\lcn{g}^s K).
 \end{equation*}
 As $g^{-1}$ is an isomorphism and $\kil^\bullet_{\lcn{g}}(M)=\ker(\lcn{g}^s)$, the result follows.
\end{proof}

\subsection{Derivation of the symmetric Schouten bracket}
We will now show that the symmetric Schouten bracket can be seen as a 'derived bracket'. Strictly speaking it is not a derived bracket in the sense of \cite{KosDB} as
\begin{equation*}
[\nabla^s,\,]:\en(\Upsilon^\bullet(M))\rightarrow\en(\Upsilon^\bullet (M)),
\end{equation*}
where $[\,\,,\,]$ stands for the commutator, does not square to zero, hence the theorems proved for a derived bracket in \cite{KosDB} do not necessarily apply here.

Given a vector space $V$ and $\mathcal{X}\in\Sym^r V$, the \textbf{contraction operator}
\begin{equation*}
 \iota_\mathcal{X}:\Sym^\bullet V^*\rightarrow\Sym^\bullet V^*.
\end{equation*}
 is a degree-$(-r)$ linear map determined by
\begin{equation*}
 \iota_{X_1\odot\ldots\odot X_r}:=\iota_{X_1}\circ\ldots\circ\iota_{X_r}
\end{equation*}
for $X_j\in V$. For $\lambda\in\R$, we define $\iota_\lambda$ as the multiplication by $\lambda$.

\begin{lemma}\label{lem: sym-Schouten-derivation}
 Let $V$ be a vector space. For every $\mathcal{X}\in\Sym^\bullet V$, $\alpha\in V^*$ and $\varphi\in\Sym^\bullet V^*$, we have
 \begin{equation}\label{eq: lem-sym-Schouten-derivation}
 \iota_{\mathcal{X}}(\alpha\odot\varphi)=\iota_{(\iota_\alpha\mathcal{X})}\varphi+\alpha\odot \iota_{\mathcal{X}}\varphi.
 \end{equation}
\end{lemma}

\begin{proof}
 We prove it by induction. For $\mathcal{X}\in\Sym^0 V$ a scalar, the identity is clearly satisfied. Fix $r\in\mathbb{N}$ and assume that \eqref{eq: lem-sym-Schouten-derivation} is true for every $\mathcal{X}\in\Sym^r V$. Using the fact that the contraction operator by a vector is a derivation of $(\Sym^\bullet V^*,\odot)$, we get, for every $X\in V$,
 \begin{align*}
 \iota_{X\odot\mathcal{X}}(\alpha\odot\varphi)
 &=\iota_X\iota_\mathcal{X}(\alpha\odot\varphi)=\iota_X(\iota_{(\iota_\alpha\mathcal{X})}\varphi+\alpha\odot\iota_\mathcal{X}\varphi)\\
 &=\iota_{X\odot(\iota_\alpha\mathcal{X})}\varphi+(\iota_X\alpha)\iota_{\mathcal{X}}\varphi+\alpha\odot\iota_X\iota_\mathcal{X}\varphi\\
 &=\iota_{(X\odot\iota_\alpha\mathcal{X}+\iota_\alpha X\odot\mathcal{X})}\varphi+\alpha\odot\iota_{X\odot\mathcal{X}}\varphi\\
 &=\iota_{\iota_\alpha(X\odot\mathcal{X})}\varphi+\alpha\odot\iota_{X\odot\mathcal{X}}\varphi.
 \end{align*}
 The result follows by linearity and the fact that $\Sym^\bullet V$ is generated by decomposable elements. 
\end{proof}

The contraction operator is naturally extended to $\symall$ and $\Upsilon^\bullet(M)$, which allows us to prove the next result.

\begin{proposition}
 Let $\nabla$ be a connection on $M$. For $\mathcal{X}$, $\mathcal{Y}\in\mathfrak{X}^\bullet_\emph{sym}(M)$, we have
 \begin{equation}\label{eq: sym-Schouten-derived}
 [[\iota_\mathcal{X},\nabla^s],\iota_{\mathcal{Y}}]=\iota_{[\mathcal{X},\mathcal{Y}]_s}.
 \end{equation}
\end{proposition}

\begin{proof}
We start with the case when $\mathcal{X}=f$ is a function. Using the fact that $\nabla^s$ is a derivation of $\Upsilon^\bullet(M)$, for every $\mathcal{Y}\in\symall$, we have
\begin{align*}
 [[\iota_f,\nabla^s],\iota_\mathcal{Y}]\varphi&=f\nabla^s\iota_\mathcal{Y}\varphi-\nabla^s(f\iota_\mathcal{Y}\varphi)-\iota_\mathcal{Y}(f\nabla^s\varphi)+\iota_\mathcal{Y}\nabla^s(f\varphi)\\
 &=-\dif f\odot\iota_\mathcal{Y}\varphi+\iota_\mathcal{Y}(\dif f\odot \mathcal{Y}).
\end{align*}
It follows from Lemma \ref{lem: sym-Schouten-derivation} that
\begin{equation*}
 [[\iota_f,\nabla^s],\iota_\mathcal{Y}]=\iota_{(\iota_{\dif f}\mathcal{Y})},
\end{equation*}
and a straightforward calculation shows that $[f,\mathcal{Y}]_s=\iota_{\dif f}\mathcal{Y}$. 

The case $\mathcal{Y}=\cg$ is a function and $\mathcal{X}$ is arbitrary follows from the previous case and the fact that both sides of \eqref{eq: sym-Schouten-derived} are symmetric on $\mathcal{X}$ and $\mathcal{Y}$. 

 Assume now that, for arbitrary $l\in\mathbb{N}$, $X\in\mathfrak{X}(M)$ and $\mathcal{Y}\in\mathfrak{X}^l_\text{sym}(M)$, the identity
 \begin{equation*}
 [[\iota_X,\nabla^s],\iota_\mathcal{Y}]=\iota_{[X,\mathcal{Y}]_s}
 \end{equation*}
 holds. Using the definition of symmetric Lie derivative \eqref{eq: sym-Lie-der} and relation \eqref{eq: sym-bracket-derived}, for $Y\in\mathfrak{X}(M)$, we get
 \begin{align*}
 [[\iota_X,\nabla^s],\iota_{Y\odot\mathcal{Y}}]&=L^s_X\iota_Y\iota_\mathcal{Y}-\iota_Y\iota_\mathcal{Y}L^s_X=\iota_Y(L^s_X\iota_\mathcal{Y}-\iota_\mathcal{Y}L^s_X)+\iota_{\pg{X,Y}_s}\iota_\mathcal{Y}\\
 &=\iota_{Y\odot[X,\mathcal{Y}]_s+[X,Y]_s\odot\mathcal{Y}}=\iota_{[X,Y\odot\mathcal{Y}]_s}.
 \end{align*}
 It follows from induction, the linearity and the fact that $\symall$ is locally generated by decomposable elements that $[[\iota_X,\nabla^s],\iota_\mathcal{Y}]=\iota_{[X,\mathcal{Y}]_s}$ for every $\mathcal{Y}\in\symall$. 
 
 Finally, consider an arbitrary $r\in\mathbb{N}$ and assume that, for every $\mathcal{X}\in\mathfrak{X}^r_\text{sym}(M)$,
 \begin{equation*}
 [[\iota_\mathcal{X},\nabla^s],\iota_\mathcal{Y}]=\iota_{[\mathcal{X},\mathcal{Y}]_s}.
 \end{equation*}
 For every $X\in\mathfrak{X}(M)$, we have
 \begin{align*}
 [[\iota_{X\odot\mathcal{X}},\nabla^s],\iota_\mathcal{Y}]&=[\iota_X\iota_\mathcal{X}\nabla^s-\nabla^s\iota_X\iota_\mathcal{X},\iota_\mathcal{Y}]=[\iota_X[\iota_\mathcal{X},\nabla^s]+L^s_X\iota_\mathcal{X},\iota_\mathcal{Y}]\\
 &=\iota_X[\iota_{\mathcal{X}},\nabla^s]\iota_\mathcal{Y}-\iota_X\iota_\mathcal{Y}[\iota_\mathcal{X},\nabla^s]+L^s_X\iota_\mathcal{X}\iota_\mathcal{Y}-\iota_\mathcal{Y}L^s_X\iota_\mathcal{X}\\
 &=\iota_{X}\iota_{[\mathcal{X},\mathcal{Y}]_s}+(L^s_X\iota_\mathcal{Y}-\iota_\mathcal{Y}L^s_X)\iota_\mathcal{X}=\iota_{X\odot[\mathcal{X},\mathcal{Y}]_s}+\iota_{[X,\mathcal{Y}]_s\odot\mathcal{X}}\\
 &=\iota_{[X\odot\mathcal{X},\mathcal{Y}]_s}
 \end{align*}
 and the result follows by induction on $r$.
\end{proof}

\section{Some proofs of Sections \ref{sec: geo-int} and \ref{sec: geo-int-strong}}\label{app: some-proofs}
For the sake of completeness, we present several technical proofs for statements whose analogues are also valid in classical Poisson geometry. However, we were unable to find these proofs in the existing literature.

\begin{customlem}{\ref{lem: loc}}
 For every $\vartheta\in\mathfrak{X}^2_\emph{sym}(M)$, the characteristic module $\mathcal{F}_\vartheta$ is local.
\end{customlem}
\begin{proof}
Consider $X\in\mathfrak{X}(M)$ such that is locally in $\mathcal{F}_\vartheta$. We thus have an open cover $\{ U_{m'}\}_{m'\in M}$ of $M$ and a collection $\{ \alpha^{(m')}\}_{m'\in M}\subseteq \Omega^1(M)$ such that
 \begin{equation*}
 \rest{X}{U_{m'}}=\vartheta(\alpha^{(m')}\hspace{-3pt}\rest{}{U_{m'}}).
 \end{equation*}
 We choose an arbitrary partition of unity $\{ f^{(m')}\}_{m'\in M}\subseteq \cCi(M)$ for the cover $\{ U_{m'}\}_{m'\in M}$ and define $\alpha\in\Omega^1(M)$ by the formula
 \begin{equation*}
 \alpha:=\sum_{m'\in M}f^{(m')}\alpha^{(m')}.
 \end{equation*}
 For every $m\in M$, we find
 \begin{align*}
 \vartheta(\alpha_m)&=\vartheta\Big(\sum_{m'\in M}f^{(m')}(m)\alpha^{(m')}_m\Big)=\sum_{m'\in M,\, m\in U_{m'}}f^{(m')}(m)\vartheta(\alpha^{(m')}_m)\\
 &=\sum_{m'\in M,\, m\in U_{m'}}f^{(m')}(m)X_m=X_m.
 \end{align*}
\end{proof}

\begin{customlem}{\ref{lem: reg-d}}
 If $\vartheta\in\mathfrak{X}^2_\emph{sym}(M)$ is regular, we have $\Gamma(\im\vartheta)=\mathcal{F}_\vartheta$. 
\end{customlem}

\begin{proof}
 Since $\mathcal{F}_\vartheta\subseteq\Gamma(\im\vartheta)$, it is enough to show that every $Y\in\Gamma(\im\vartheta)$ is in $\mathcal{F}_\vartheta$. For a given point $m\in M$, we find a coordinate chart $(U,\{ x^i\})$ around it. It is clear that the gradient vector fields $\{ X_{x^i}\}$ generate the space of local sections $\Gamma(U,\im\vartheta)$ as a $\cCi(U)$-module. Therefore, we have a (not unique) collection $\{ f^i\}\subseteq\cCi(U)$ such that
 \begin{equation*}
 \rest{Y}{U}=f^iX_{x^i}=\vartheta( f^i\dif x^i).
 \end{equation*}
 We define $\alpha\in\Omega^1(U)$ by $\alpha:=f^i\dif x^i$. Namely, we have $\rest{Y}{U}=\vartheta(\alpha)$. As we can find a neighbourhood of $m$ (possibly smaller than $U$) such that there is $\hat{\alpha}\in\Omega^1(M)$ that agrees with $\alpha$ on that neighborhood, the result follows from the fact that $\mathcal{F}_\vartheta$ is, by Lemma \ref{lem: loc}, local.
\end{proof}

\begin{remark}
 We are presenting the results in the logical order for this appendix. Note that Lemma \ref{lem: loc} has a stand-alone proof.
\end{remark}

\begin{customprop}{\ref{prop: family-metrics}}
 Let $(\Delta,g)$ be a \mbox{(pseudo-)Rie}mannian vector subspace of a real $n$-dimensional vector space $V$. There is a unique $\vartheta\in\Sym^2V$ such that $\im\vartheta=\Delta$ and, for $\alpha,\beta\in V^*$,
 \begin{equation*}
 \vartheta(\alpha,\beta)=g(\vartheta(\alpha),\vartheta(\beta)).
 \end{equation*}
 Moreover, if $g$ is of signature $(r,s)$, the symmetric bivector $\vartheta$ is of signature $(r,s,n-(r+s))$. 
\end{customprop}

\begin{proof}
 We have the canonical vector space isomorphism $V^*/\Ann \Delta\cong \Delta^*$. Using the quotient map $q:V^*\rightarrow\Delta^*$, we define
 \begin{equation*}
 \vartheta(\alpha,\beta):=g^{-1}(q(\alpha),q(\beta))
 \end{equation*}
 for any $\alpha,\beta\in V^*$. Clearly, we have $\vartheta\in\Sym^2V$ and $\vartheta=q^t\circ g^{-1}\circ q$ as a map $V^*\rightarrow V$. Explicitly, $q^t$ is the inclusion and $q=\rest{ }{\Delta}$ is the restriction map. Therefore,
 \begin{equation*}
 \im\vartheta=\im(q^t\circ g^{-1}\circ q)=\Delta
 \end{equation*}
 and, for $\alpha,\beta\in V^*$,
 \begin{equation*}
 \vartheta(\alpha,\beta)=g^{-1}(q(\alpha),q(\beta))=g(g^{-1}(q(\alpha)),g^{-1}(q(\beta)))=g(\vartheta(\alpha),\vartheta(\beta)).
 \end{equation*}
 Consider now $\varkappa\in\Sym^2V$ such that $\vartheta+\varkappa$ also satisfies the two properties. Since $\im\vartheta+\im\varkappa=\Delta$, we have that $\im\varkappa\leq\Delta=\im\vartheta$. On the other hand,
 \begin{align*}
 (\vartheta+\varkappa)(\alpha,\beta)&=g((\vartheta+\varkappa)(\alpha),(\vartheta+\varkappa)(\beta))\\
 &=g((\vartheta+\varkappa)(\alpha),\varkappa(\beta))+g(\vartheta(\alpha),\vartheta(\beta))+g(\varkappa(\alpha),\vartheta(\beta)).
 \end{align*}
 As there is $\alpha'\in V^*$ such that $\varkappa(\alpha)=\vartheta(\alpha')$, we get
 \begin{equation*}
 g(\vartheta(\alpha)+\varkappa(\alpha),\varkappa(\beta))=0,
 \end{equation*}
 that is, $\vartheta(\alpha)+\varkappa(\alpha)\in\im\varkappa^\perp\leq \Delta$ for every $\alpha\in V^*$, which means $\Delta=\im(\vartheta+\varkappa)=\im\varkappa^\perp$. So $\im\varkappa=\lbrace 0\rbrace$, that is $\varkappa=0$, and the uniqueness of $\vartheta$ follows.
\end{proof}

\begin{customlem}{\ref{lem: loc-fin}}
 For every $\vartheta\in\mathfrak{X}^2_\emph{sym}(M)$, the characteristic module $\mathcal{F}_\vartheta$ is locally finitely generated. 
\end{customlem}

\begin{proof}
 For a point $m\in M$, we choose a coordinate chart $(U,\{ x^i\})$ for $M$. Given $Y\in\Gamma_c(U,\mathcal{F}_\vartheta)$, we take a bump $f_Y\in\cCi(M)$ such that $\supp f_Y\subseteq U$ and $\rest{f_Y}{\supp Y}=1$, hence $Y=f_YY\in\mathcal{F}_\vartheta\subseteq\mathfrak{X}(M)$ and we find $\alpha\in\Omega^1(M)$ such that $Y=\vartheta(\alpha)$. As there is the unique collection $\{ \alpha_i\}\subseteq\cCi(U)$ such that $\rest{\alpha}{U}=\alpha_i\dif x^i$, we have
 \begin{equation*}
 Y=f_YY=f_Y\alpha_i\vartheta(\dif x^i)=(f_Y\alpha_i)X_{x^i}.
 \end{equation*}
 Therefore, the gradient vector fields $\{ X_{x^i}\}_{i=1}^n\subseteq \Gamma(U,\mathcal{F}_\vartheta)$ generate the space $\Gamma_c(U,\mathcal{F}_\vartheta)$ as a $\cCi_c(U)$-module.
\end{proof}

\section{Lie algebroids and symmetric bivector fields}\label{app: Lie algebroids}
For a symmetric bivector field $\vartheta\in\symbi$ and a connection $\nabla$ on $M$, we have the bracket $[\,\,,\,]:\Omega^1(M)\times\Omega^1(M)\rightarrow\Omega^1(M)$ defined, for $\alpha,\beta\in\Omega^1(M)$, by
\begin{equation}\label{eq: ssPs-Lie}
 [\alpha,\beta]:=\nabla_{\vartheta(\alpha)}\beta-\nabla_{\vartheta(\beta)}\alpha.
\end{equation}
We explore its properties of in this section. First, we recall some terminology.

\begin{definition}
 An \textbf{almost Lie algebroid} over $M$ is a triple $(E,\rho,[\,\,,\,])$, consisting of a vector bundle $E\rightarrow M$, a vector bundle morphism $\rho:E\rightarrow TM$ over the identity on $M$ and an $\R$-bilinear map $[\,\,,\,]:\Gamma(E)\times\Gamma(E)\rightarrow\Gamma(E)$, such that, for $a,b\in\Gamma(E)$ and $f\in\cCi(M)$,
 \begin{align*}
 [a,b]&=-[b,a], & [a,fb]&=(\rho(a)f)b+f[a,b].
 \end{align*}
 If in addition $\rho:(E,[\,\,,\,])\rightarrow(TM,[\,\,,\,])$ is an algebra morphism, $(E,\rho,[\,\,,\,])$ is called a \textbf{pre-Lie algebroid}. If $(\Gamma(E),[\,\,,\,])$ is, in addition, a Lie algebra, $(E,\rho,[\,\,,\,])$ is called a \textbf{Lie algebroid}.
\end{definition}

It is easy to see that the bracket $[\,\,,\,]$, defined by \eqref{eq: ssPs-Lie}, makes $(T^*M,\vartheta,[\,\,,\,])$ an almost Lie algebroid. By Proposition \ref{prop: ssPs-involutivity}, we have that $(T^*M,\vartheta,[\,\,,\,])$ is a pre-Lie algebroid if $(\vartheta,\nabla)$ is a strong symmetric Poisson structure.

\begin{proposition}\label{prop: Lie-algebroid}
 Let $(\vartheta,\nabla)$ be a strong symmetric Poisson structure on $M$. The triple $(T^*M,\vartheta,[\,\,,\,])$ is a Lie algebroid if and only if, for $\alpha,\beta,\eta\in\Omega^1(M)$,
 \begin{equation}\label{eq: Bianchi}
 R_\nabla(\vartheta(\alpha),\vartheta(\beta))\eta+\cyc(\alpha,\beta,\eta)=0,
 \end{equation}
 where $R_\nabla\in\Omega^2(M,\en TM)$ is the Riemann curvature. Namely, a strong symmetric Poisson structure $(\vartheta,\nabla)$ gives a Lie algebroid if $\nabla$ is flat.
\end{proposition}

\begin{proof}
 As we already know that for a strong symmetric Poisson structure $(\vartheta,\nabla)$, the triple $(T^*M,\vartheta,[\,\,,\,])$ is a pre-Lie algebroid. It becomes a Lie algebroid if and only if the Jacobi identity is satisfied. For $\alpha,\beta,\eta\in\Omega^1(M)$, we get
 \begin{equation*}
 [\alpha,[\beta,\eta]]=\nabla_{\vartheta(\alpha)}\nabla_{\vartheta(\beta)}\eta-\nabla_{\vartheta(\beta)}\nabla_{\vartheta(\alpha)}\eta-\nabla_{\vartheta([\beta,\eta])}\alpha.
 \end{equation*}
 As $(\vartheta,\nabla)$ is a strong symmetric Poisson structure, we get by Proposition \ref{prop: ssPs-involutivity} that
 \begin{equation*}
 [\alpha,[\beta,\eta]]=\nabla_{\vartheta(\alpha)}\nabla_{\vartheta(\beta)}\eta-\nabla_{\vartheta(\beta)}\nabla_{\vartheta(\alpha)}\eta-\nabla_{[\vartheta(\beta),\vartheta(\eta)]}\alpha.
 \end{equation*}
 The cyclic permutation of the above equation gives the result. 
\end{proof}

If $g$ is a (pseudo-)Riemannian metric on $M$, the condition \eqref{eq: Bianchi} is equivalent to
\begin{equation*}
 R_\nabla(X,Y)g(Z)+\cyc(X,Y,Z)=0,
\end{equation*}
for $X,Y,Z\in\mathfrak{X}(M)$. If $\nabla=\lcn{g}$, that is, $\nabla$ is the Levi-Civita connection of $g$, we have that $g$ and $\lcn{g}$ commute, hence \eqref{eq: Bianchi} becomes
\begin{equation*}
 g(R_{\lcn{g}}(X,Y)Z+\cyc(X,Y,Z))=0.
\end{equation*}
The algebraic Bianchi identity, $R_\nabla(X,Y)Z+\cyc(X,Y,Z)=0$, is satisfied for any torsion-free connection $\nabla$ on $M$ (by the Jacobi identity of the usual Lie bracket), so
$(T^*M, g^{-1},[\,\,,\,])$ is a Lie algebroid. Moreover, $g^{-1}:T^*M\rightarrow TM$ is a Lie algebroid isomorphism between $(T^*M, g^{-1},[\,\,,\,])$ and the tangent Lie algebroid.

% https://math.umd.edu/~wmg/gstom.pdf. pages 113, 
%https://mathoverflow.net/questions/149492/can-a-manifold-have-a-curvature-free-connection-that-is-not-torsion-free
\begin{remark}
Note that flatness is a very strong constraint, for instance, for closed surfaces is equivalent to the vanishing of the Euler class. After we had developed symmetric Poisson geometry, we found that pairs consisting of a flat torsion-free connection $\nabla$ and a symmetric bivector $\vartheta$ had been considered before in \cite{pseudo-Hessian-0} to describe degenerations of so-called pseudo-Hessian structures (pairs $(g,\nabla)$ consisting of a (pseudo-)Riemannian metric and a flat torsion-free connection satisfying  
$(\nabla_Xg)(Y,Z)=(\nabla_Yg)(X,Z)$). A pair $(\vartheta,\nabla)$  satisfying the analogous integrability condition $(\nabla_{\vartheta(\alpha)}\vartheta)(\beta,\eta)=(\nabla_{\vartheta(\beta)}\vartheta)(\alpha,\eta)$ is called a contravariant pseudo-Hessian or, currently, a Koszul-Vinberg structure. 
Accordingly, the foliation associated to such a structure consists of pseudo-Hessian leaves, whereas linear structures correspond to commutative associative algebras. By Proposition \ref{prop: sym-Poisson-bracket} and Corollary \ref{cor: ssPs-integrability}, the intersection of symmetric Poisson and Koszul-Vinberg structures are precisely strong symmetric Poisson structures with a flat connection. %the torsion-free connection being moreover flat.
We believe that symmetric and strong symmetric Poisson structures offer a more general and fundamental approach, rooted in symmetric Cartan calculus and analogous to Poisson geometry, as we summarize in the next table.   
\end{remark}

% \section{Relation to Koszul-Vinberg structures}\label{app:Koszul-Vinberg}

% out about Koszul-Vinberg structures \cite{pseudo-Hessian-1}, an approach to symmetric bivector fields that was introduced in the context of pseudo-Hessian structures.

% A Koszul-Vinberg structure is defined as a pair $(\vartheta, \nabla)$ consisting of $\vartheta\in\symbi$ and a flat torsion-free connection $\nabla$ such that
% \begin{equation*}
%     (\nabla_{\vartheta(\alpha)}\vartheta)(\beta,\eta)=(\nabla_{\vartheta(\beta)}\vartheta)(\alpha,\eta)
% \end{equation*}
% When $\vartheta$ is non-degenerate, this condition recovers the Codazzi equation for $g=\vartheta^{-1}$, that is, $(\nabla_Xg)(Y,Z)=(\nabla_Yg)(X,Z)$. This is what is called a pseudo-Hessian manifold.

% By Proposition \ref{prop: sym-Poisson-bracket} and Corollary \ref{cor: ssPs-integrability}, the intersection of symmetric Poisson structures and Koszul-Vinberg structures are precisely strong symmetric Poisson structures with a flat connection.

% Despite their similarity, there are three fundamental differences that affect the development of the theory:
% flatness

% the integrability condition does not have an analogous to Poisson 

% and is not based on any Cartan calculus.
% % , at first glance,?

% The algebraic structure associated to the linear case is that of a commutative associative algebra. %% see Proposition 4.2

% The leaves integrating the foliation come equipped with pseudo-Hessian structures.

% FLAT IS A BIG HYPOTHESIS, it implies Lie algebroid and this implies foliation
% 

\newgeometry{left=1.85cm, right=1.85cm, bottom=1.5cm}
\section{Comparison with Poisson geometry}
We show in this table a comparison between Poisson and (strong) symmetric Poisson geometry.
\begin{center}
\noindent\SetTblrInner{rowsep=7pt}
\begin{tblr}{width=\linewidth, colspec={|m{0.375\linewidth,c}||m{0.225\linewidth,c}|m{0.3\linewidth,c}|}, cell{2}{1}={c=3}{c},cell{3}{2}={c=2}{c},cell{4}{2}={c=2}{c},cell{5}{2}={c=2}{c},cell{6}{2}={c=2}{c},cell{7}{1}={c=3}{c},cell{8}{2}={c=2}{c},cell{9}{2}={c=2}{c},cell{10}{2}={c=2}{c},cell{11}{2}={c=2}{c},cell{12}{1}={c=3}{c},cell{14}{2}={c=2}{c},cell{16}{2}={c=2}{c},cell{17}{1}={c=3}{c},cell{19}{1}={c=3}{c}}
\hline
\textbf{Poisson geometry} & \textbf{Symmetric\break Poisson geometry} & \textbf{Strong symmetric\break Poisson geometry}\\
\hline\hline
algebraic features\\
\hline
skew-symmetric bivector fields\break $\pi\in\mathfrak{X}^2(M)$ & symmetric bivector fields\break $\vartheta\in\mathfrak{X}^2_\text{sym}(M)$\\
 Hamiltonian vector fields\break $\Ham:=\pi\circ\dif$ & gradient vector fields\break $\grad:=\vartheta\circ \dif$\\
$\{ f,\cg\}=-\{ \cg,f\}$ & $\{ f,\cg\}=\{ \cg,f\}$\\
$\{ f,\cg h\}=\{ f,\cg\} h+\cg\{ f,h\}$ & $\{ f,\cg h\}=\{ f,\cg\} h+\cg\{ f,h\}$\\
\hline\hline
differential calculi\\
\hline
classical Cartan calculus &symmetric Cartan calculus\\
canonical & depending on the choice of $\nabla$\\
Schouten bracket $[\,\,,\,]$ & symmetric Schouten bracket $[\,\,,\,]_s$\\
graded anti-commutative & commutative\\
\hline\hline
integrability conditions\\
\hline
$[\pi,\pi]=0$ & $[\vartheta,\vartheta]_s=0$ & $\nabla_{\grad f}\,\vartheta=0$\\
$\Jac(f,\cg,h)=0$ & $\Jac(f,\cg,h)=\dif h([\grad f,\grad \cg]_s)+\cyc(f,\cg,h)$\\
 $\Ham\{ f,\cg\}=[\Ham f,\Ham \cg]$ & – & $\grad\{ f,\cg\}=[\grad f,\grad \cg]_s$\\
 $(\nabla_{\pi(\alpha)}\pi)(\beta,\eta)+\cyc(\alpha,\beta,\eta)=0$ & $(\nabla_{\vartheta(\alpha)}\vartheta)(\beta,\eta)+\cyc(\alpha,\beta,\eta)=0$\\
 \hline\hline
non-degenerate structures\\
\hline
symplectic\break structures & non-degenerate\break Killing $2$-tensors & (pseudo-)Riemannian\break metrics\\
\hline\hline
linear structures\\
\hline
Lie\break algebras & Jacobi-Jordan\break algebras & associative Jacobi-Jordan\break algebras\\
\hline
\end{tblr}
\end{center}

\clearpage

\begin{center}
\noindent\SetTblrInner{rowsep=7pt}
\begin{tblr}{width=\linewidth, colspec={|m{0.35\linewidth,c}||m{0.25\linewidth,c}|m{0.3\linewidth,c}|}, cell{2}{1}={c=3}{c},cell{5}{1}={c=3}{c},cell{6}{2}={c=2}{c},cell{7}{2}={c=2}{c}}
\hline
\textbf{Poisson geometry} & \textbf{Symmetric\break Poisson geometry} & \textbf{Strong symmetric\break Poisson geometry}\\
\hline\hline
geometric interpretation\\
\hline
partitions & locally geodesically invariant distributions & totally geodesic partitions\\
symplectic structure\break on leaves & constant square of  speed of $\vartheta$-admissible geodesics for characteristic metric & (pseudo-)Riemannian metric on leaves, whose Levi-Civita connection is the restrictions of the ambient connection\\
%
% symplectic structures\break on leaves & constant square of the speed of $\vartheta$-admissible geodesics with respect to the characteristic metric & (pseudo-)Riemannian metrics on leaves whose Levi-Civita connections are the restrictions of the ambient connection\\
%
\hline
%\hline
% Algebroid structure on $T^*M\rightarrow M$\\
% \hline
% $\rho(\alpha)=\pi(\alpha)$ & $\rho(\alpha)=\vartheta(\alpha)$\\
% $[\alpha,\beta]=L_{\pi(\alpha)}\beta-L_{\pi(\beta)}\alpha-\dif\pi(\alpha,\beta)$& $[\alpha,\beta]=\nabla_{\vartheta(\alpha)}\beta-\nabla_{\vartheta(\beta)}\alpha$\\
% always a Lie algebroid & – & always a pre-Lie algebroid;\phantom{x} a Lie algebroid if and only if $R_\nabla(\vartheta(\alpha), \vartheta(\beta))\eta+\cyc=0$\\ %%(\alpha,\beta,\eta) 
% \hline
\end{tblr}
\end{center}

% \vspace{-.3cm}

\bibliographystyle{alpha}\bibliography{refs}

\end{document}